\documentclass{article}

\usepackage{bbm}
\usepackage{arydshln}
\usepackage{subcaption}
\usepackage{xcolor}
\usepackage{ifdraft}
\usepackage{scalerel}
\usepackage{mathrsfs}  
\usepackage{etoc}
\usepackage{comment}
\usepackage{titletoc}
\usepackage{soul}

\usepackage{latexsym,amsfonts}
\usepackage[utf8]{inputenc}
\usepackage[final=true,protrusion=true,expansion=true]{microtype}
\usepackage[noadjust]{cite}
\usepackage{enumitem}
\usepackage{color}
\usepackage{mathtools}

\usepackage{booktabs}
\usepackage{multirow}

\usepackage{amssymb,amsthm,amsmath}

\usepackage{tikz}
\usetikzlibrary{angles, quotes}
\usepackage{pgffor}

\newtheorem{theorem}{Theorem}
\newtheorem{assumption}{Assumption}
\newtheorem{lemma}[theorem]{Lemma} 
\newtheorem{proposition}[theorem]{Proposition} 
\newtheorem{corollary}[theorem]{Corollary}
\newtheorem{conjecture}{Conjecture}
\theoremstyle{definition}

\newtheorem{example}[theorem]{Example}

 \usepackage[main, final]{neurips_2026}

\usepackage[utf8]{inputenc} 
\usepackage[T1]{fontenc}    
\usepackage{hyperref}       
\usepackage{url}            
\usepackage{booktabs}       
\usepackage{amsfonts}       
\usepackage{nicefrac}       
\usepackage{microtype}      
\usepackage{xcolor}         

\usepackage{algorithm}
\usepackage{algpseudocode}
\usepackage{soul}

\newcommand{\R}{\mathbb{R}}

\newcommand{\ip}[2]{\left\langle #1, #2 \right\rangle}
\newcommand{\proj}[1]{\mathrm{P}_{#1}}
\newcommand{\PE}{\mathrm{P}_E}

\providecommand*{\absbig}[1]{\big|{#1}\big|} 
\newcommand*{\N}[1]{\left\|{#1}\right\|} 
\providecommand*{\Nnormal}[1]{\|{#1}\|} 
\providecommand*{\Nbig}[1]{\big\|{#1}\big\|} 

\DeclareMathOperator*{\argmin}{arg\,min}

\newcommand{\eps}{\varepsilon}

\newcommand{\de}{\,\mathrm{d}}
\newcommand{\sphere}{{\mathbb{S}^{d-1}}}
\DeclareMathOperator{\divtmp}{div}
\renewcommand{\div}{\divtmp}
\newcommand{\restr}{\mathbin{\vrule height 1.6ex depth 0pt width
0.13ex\vrule height 0.13ex depth 0pt width 1.3ex}}
\newcommand{\dummy}{\mathord{\color{black!33}\bullet}}

\renewcommand{\restr}{\mathbin{\vrule height 1.6ex depth 0pt width 0.13ex\vrule height 0.13ex depth 0pt width 0.6ex}}

\definecolor{myred}{HTML}{FF0000}
\definecolor{myorange}{HTML}{EB6B03}

\title{
Quantifying Concentration Phenomena of Mean-Field Transformers in the Low-Temperature Regime
}

\author{%
  Albert Alcalde\\
  Department of Mathematics\\
  Friedrich-Alexander-University\\  
  Erlangen-Nürnberg\\
  \texttt{albert.alcalde@fau.de}\\
  \hphantom{\texttt{konstantin.riedl@maths.ox.ac.uk}}
  \And 
  Leon Bungert\\
  Institute of Mathematics\\
  CAIDAS\\
  University of Würzburg\\
  \texttt{leon.bungert@uni-wuerzburg.de}\\
  \hphantom{\texttt{konstantin.riedl@maths.ox.ac.uk}}
  \And 
  Konstantin Riedl\\
  Mathematical Institute\\
  University of Oxford\\
  Reuben College\\
  \texttt{konstantin.riedl@maths.ox.ac.uk}\\
  \hphantom{\texttt{konstantin.riedl@maths.ox.ac.uk}}
  \And
  Tim Roith\\
  CIT School\\
  Technical University of Munich\\
  Munich Center for Machine Learning\\
  \texttt{tim.roith@tum.de}\\
  \hphantom{\texttt{konstantin.riedl@maths.ox.ac.uk}}
}

\begin{document}

\maketitle


\begin{abstract}
    Transformers with self-attention modules as their core components have become an integral architecture in modern large language and foundation models. In this paper, we study the evolution of tokens in deep encoder-only transformers at inference time which is described in the large-token limit by a mean-field continuity equation. Leveraging ideas from the convergence analysis of interacting multi-particle systems, with particles corresponding to tokens, we prove that the token distribution rapidly concentrates onto the push-forward of the initial distribution under a projection map induced by the key, query, and value matrices, and remains metastable for moderate times. Specifically, we show that the Wasserstein distance of the two distributions scales like \mbox{$\sqrt{\nicefrac{\log(\beta+1)}{\beta}}\exp(C t)+\exp(-ct)$} in terms of the temperature parameter $\beta^{-1}\to 0$ and inference time $t\geq 0$. For the proof, we establish Lyapunov-type estimates for the zero-temperature equation, identify its limit as $t\to\infty$, and employ a stability estimate in Wasserstein space together with a quantitative Laplace principle to couple the two equations. Our result implies that for time scales of order $\log\beta$, which correspond to the early-phase behavior most relevant for practical finite-depth models, the token distribution concentrates at the identified limiting distribution. Numerical experiments confirm this and, beyond that, complement our theory by showing that for finite $\beta$ and large $t$ the dynamics enter a different terminal phase, dominated by the spectrum of the value matrix.
\end{abstract}

\section{Introduction}

The transformer architecture has driven in recent years the unprecedented successes of modern machine learning,
establishing itself as the fundamental backbone for large language models~\cite{bert,gemini2023,gpt,llama}, and foundation models~\cite{rishi2024foundation}.
Central to its design is the self-attention mechanism~\cite{bahdanau2014neural,vaswani2017attention},
which enables the model to capture intricate, long-range contextual dependencies within the input data.
A rigorous characterization of how these internal representations evolve through the network, however, remains a significant analytical challenge.

To mathematically study how tokens~$\{x_i\}_{i=1}^n\subset\mathbb{R}^{d}$ evolve through the self-attention modules of a deep encoder-only transformer,
we follow \cite{geshkovski2023mathematical,sander2022sinkformers,bruno2025multiscale,burger2025analysis} and model the forward pass of a transformer with single-head attention and without fully-connected MLPs, in the infinite-depth limit, as a continuous-time dynamical interacting particle system of the form
\begin{equation}
    \label{eq:SAdynamics}
\begin{aligned}
        \dot{x}_i(t) &= \proj{x_i(t)} \left( \sum_{j = 1}^n \frac{\exp\left(\beta\ip{Q(t)x_i(t)}{K(t) x_j(t)}\right)} {\sum_{k = 1}^n \exp\left(\beta \ip{Q(t)x_i(t)}{K(t) x_k(t)}\right)} V(t) x_j(t) \right)\!.
\end{aligned}
\end{equation}
System~\eqref{eq:SAdynamics} governs the evolution of tokens $\{x_i\}_{i=1}^n$, interpreted as particles,  through the self-attention layers of the transformer at inference.
Here, $Q(t),K(t)\in\mathbb{R}^{d_{\text{int}}\times d}$ and $V(t)\in\mathbb{R}^{d\times d}$ denote the \emph{query}, \emph{key}, and \emph{value} weight matrices of the attention modules, while $\beta > 0$ is an inverse-temperature parameter \cite{geshkovski2023mathematical,chen2025critical,bruno2025multiscale}. 
The weight matrices have been obtained during training and are fixed for the inference dynamics, which are the focus of the present work. 
Here, $\proj{x}(y)\coloneqq y-\left\langle x,y\right\rangle x$ for $y\in\R^d$ is the orthogonal projection onto the tangent space~$\mathrm{T}_x\sphere:=\{y\in\R^d:\langle x,y\rangle=0\}$ of the sphere~$\sphere$ at a point $x\in\sphere$,
and models layer normalization. This is motivated by the simplified root mean square normalization $x\mapsto g\odot \nicefrac{x}{\N{x}}$ with gain $g\in\R^d$ \cite{zhang2019root}, which is used, e.g., in Llama models~\cite{touvron2023llama}.
We set $g=1$ and refer to \cite{karagodin2025normalization} and \cite[Section 1(b)]{burger2025analysis} for a recent review of modeling choices for normalizations.

\paragraph{Mathematical setting.}
Let us now introduce the large-token regime of \eqref{eq:SAdynamics}, which serves as the analytical framework for this paper. As done in the largest part of the current literature \cite{geshkovski2023mathematical,sander2022sinkformers,bruno2025multiscale,burger2025analysis} and justified by memory-efficient variants of transformers such as Universal Transformers \cite{dehghani2018universal} and ALBERT~\cite{lan2019albert}, we assume in our analysis
that the weight matrices $Q$, $K$, and $V$ are identical throughout the layers, i.e., that $Q(t)\equiv Q$, $K(t)\equiv K$, and $V(t)\equiv V$ for all $t\geq0$. For notational simplicity and without loss of generality, we henceforth write $B \coloneqq Q^\top K\in\R^{d\times d}$.

After positional encoding~\cite[Section~3.5]{vaswani2017attention}, the transformer dynamics \eqref{eq:SAdynamics} are invariant under permutations allowing the empirical measure $\rho^{\beta,n}_t = \frac1 n\sum_{i=1}^n \delta_{x_i(t)}$ to fully characterize the dynamics, which is described by a continuity equation of the form~\eqref{eq:SAdynamics_meanfield} below. Our theoretical analysis will take place in the infinite-token regime, where the temporal dynamics of the token distribution~$(\rho_t^\beta)_{t\geq 0}\subset\mathcal{P}(\sphere)$ with $\rho_0^\beta = \rho_0\in\mathcal{P}(\sphere)$ is captured by the \emph{transformer continuity equation}
\begin{equation}
    \label{eq:SAdynamics_meanfield}
    \partial_t \rho_t^\beta
    = - \div\left( \rho_t^\beta \,\proj{x} \left(\int_{\sphere} \frac{\exp\left(\beta \ip{x}{By}\right)}{\int_{\sphere} \exp\left(\beta \ip{x}{Bz}\right)\!\de\rho_t^\beta(z)} V y
    \de
    \rho_t^\beta(y) \right) \right)\!,
\end{equation}
which describes a flow in the space $\mathcal{P}(\sphere)$ of probability measures over the sphere $\sphere$. Here, $\div$ and $\nabla$ denote the manifold divergence and gradient on the sphere $\sphere$, respectively. It is well-known that the empirical measure $\rho_t^{\beta,n}$ defined above is a weak solution to \eqref{eq:SAdynamics_meanfield}, however, this equation allows for more complex solutions if the initial condition is not supported on a finite set.

As shown in the seminal work \cite{bruno2025multiscale}
under invertibility assumptions on $B$ and $V$, and Assumption~\ref{asm:lower_bound} on the initial datum~$\rho_0$,
solutions of \eqref{eq:SAdynamics_meanfield} converge in the zero-temperature limit $\beta^{-1}\to 0$ to those of 
\begin{equation}
    \label{eq:zerotemp_meanfield}
    \partial_t\rho_t=-\div\left(\rho_t\,\proj{x}\left(\frac{VB^\top x}{\N{B^\top x}}\right)\right)\!.
\end{equation}
Moreover, for large times,
solutions of \eqref{eq:zerotemp_meanfield} concentrate on the top eigenspace of the matrix~$VB^\top$, that is, the eigenspace associated with its largest eigenvalue, which we denote by $E$. Correspondingly, for such a linear subspace~$E$ of $\R^d$,
we denote by \mbox{$\PE:\R^d\to E$} the orthogonal projection onto $E$
and define $\Pi:\sphere\setminus E^\perp\to\sphere$ as
\begin{equation}
    \label{eq:Pi}
    \Pi(x)
    \coloneqq \frac{\PE x}{\N{\PE x}}.
\end{equation}

\paragraph{Main Contributions.}
In this work, we study the evolution of tokens in deep encoder-only transformers at inference time by analyzing the mean-field token distribution~$\rho_t^\beta$ of the transformer continuity equation \eqref{eq:SAdynamics_meanfield}.
We quantify its concentration at the push-forward $\Pi_\sharp\rho_0$ of the initial datum under the projection map~\eqref{eq:Pi}, which is determined by the model weights, in terms of the temperature parameter $\beta^{-1}$ and inference time $t$. Our proof leverages tools from the analysis of interacting particle systems, in particular, stability estimates in Wasserstein space involving a quantitative Laplace principle, and Lyapunov techniques to quantify convergence speeds.
Our main contributions are as follows.
\begin{itemize}[label=\large\textbullet, labelsep=8pt,leftmargin=25pt,topsep=-4pt,itemsep=-1pt]
    \item Identification of $\Pi_\sharp\rho_0$ as the stationary distribution of the zero-temperature continuity equation~\eqref{eq:zerotemp_meanfield} for symmetric matrices~$VB^\top$,
    and proof of exponential convergence of the zero-temperature solution~$\rho_t$ to the top eigenspace of $VB^\top$ in the Wasserstein distance.
    \item Derivation of a quantitative Wasserstein stability estimate between solutions $\rho_t^\beta$ of the transformer continuity equation \eqref{eq:SAdynamics_meanfield} and the zero-temperature limit~$\rho_t$.
    \item Quantification of the approximation error between $\rho_t^\beta$ and $\Pi_\sharp\rho_0$ in Wasserstein distance as $\sqrt{\nicefrac{\log(\beta+1)}{\beta}}\exp(C t)+\exp(-c t)$, indicating an interplay between temperature~$\beta^{-1}$ and time~$t$.
    \item Characterization of the non-asymptotic concentration regime of the mean-field transformer dynamics up to times $t=\mathcal{O}(\log\beta)$, corresponding to the early-phase behavior most relevant for practical finite-depth models.
    \item Experimental identification of the terminal phase after times $t=\mathcal{O}(\log\beta)$, where concentration shifts
    from the top eigenspace of $VB^\top$ to an eigenspace of $V$.
\end{itemize}

\paragraph{Organization.} 
Following the presentation and discussion of our theoretical contributions in Section~\ref{sec:mainresults},
Section~\ref{sec:relatedworks} puts these findings into context within the current literature.
Thereafter, in Section~\ref{sec:proof:main}, we present proof details for our statements, with technical results deferred to Appendices~\ref{app:proofs} and \ref{app:conjecture}.
Section~\ref{sec:numericalexperiments} then illustrates and complements our theoretical findings with numerical experiments, with additional experiments provided in Appendix~\ref{app:numericalexperiments}. We conclude the paper with Section~\ref{sec:conclusions}.

\section{Main results}
\label{sec:mainresults}

This section is dedicated to discussing the main theoretical contributions of our paper. We first state the assumptions needed for proving our main results, before elaborating on them in more detail.
\begin{assumption}
    \label{asm:weights}
    The weight matrices~$B$ and $V$ are such that
	\begin{enumerate}[label=A\arabic*,labelsep=10pt,leftmargin=35pt,topsep=-2pt,itemsep=-1pt]
        \item the matrix $B=Q^\top K$ is invertible, \label{asm:Binvertible}
        \item the matrix $VB^\top\in\R^{d\times d}$ is symmetric, with largest eigenvalue $\mu_1$ having multiplicity $k$ and second largest eigenvalue $\mu_2$ satisfying $\gamma \coloneqq \mu_1 - \mu_2 > 0$ (with the convention that $\gamma = \infty$ if $k = d$). We denote by $E$ the eigenspace of $VB^\top$ with respect to $\mu_1$ and by $\PE:\R^d\to E$ the orthogonal projection onto $E$. \label{asm:VBTsymmetric}
    \end{enumerate}
\end{assumption}
\begin{assumption}
    \label{asm:Vp_0}
    The initial datum~$\rho_0\in\mathcal{P}(\sphere)$ is such that $\mathcal{V}_p(\rho_0) \coloneqq \int_\sphere \! \frac{\N{(\mathrm{Id}-\PE)x}^{2p}}{\N{\PE x}^{2p}}\de\rho_0(x)$ is finite for some $p\in (0,1]$.
\end{assumption}
\begin{assumption}
    \label{asm:lower_bound}
    The initial datum~$\rho_0\in\mathcal{P}(\sphere)$ has a continuous density $f_0$ with respect to the surface measure $\mathcal{H}^{d-1}\restr\sphere$ that satisfies $\ell_0\coloneqq\min_{x\in\sphere}f_0(x)>0$.
\end{assumption}
Under Assumptions~\ref{asm:weights}--\ref{asm:lower_bound} on the weight matrices~$B$ and $V$, and the initial datum~$\rho_0$,
we derive in the subsequent statement quantitative concentration bounds for the transformer continuity equation~\eqref{eq:SAdynamics_meanfield}. 
We denote the smallest and largest singular values of a matrix by $\sigma_{\min}(\dummy)$ and $\sigma_{\max}(\dummy)$.
\begin{theorem}
    \label{thm:main}
    Assume that the matrices~$B$ and $V$
    satisfy Assumption~\ref{asm:weights} and that the initial distribution $\rho_0\in\mathcal{P}(\sphere)$ satisfies Assumption~\ref{asm:Vp_0} for some $p\in (0, 1]$ as well as Assumption~\ref{asm:lower_bound}.
    Let $(\rho_t^\beta)_{t\in[0,T]}$ denote the weak solution to the transformer continuity equation~\eqref{eq:SAdynamics_meanfield}.
    Then,
    for every $\beta>0$, it holds
    \begin{equation*}        
        W_2(\rho_t^\beta,\Pi_\sharp\rho_0)
        \leq 
        2 \sqrt{\tfrac{\log(\beta+1)}{\beta}} \left(e^{C_1t}-e^{C_0t}\right)
        + \sqrt{2\mathcal{V}_p(\rho_0)}\exp{\left(-\tfrac{p\gamma}{\sigma_{\max}(B)} t\right)},
    \end{equation*}
    for constants $C_1>C_0>0$ that depend only on $d$, $\sigma_{\min}(B)$, $\sigma_{\max}(B)$, $\sigma_{\max}(V)$, and $\ell_0$.
\end{theorem}
\begin{figure}[h!]
\def\valbeta{30}
\centering%
\begin{subfigure}{.28\textwidth}%
\includegraphics[width=\linewidth, trim={2cm 2cm 2cm 2cm}, clip]{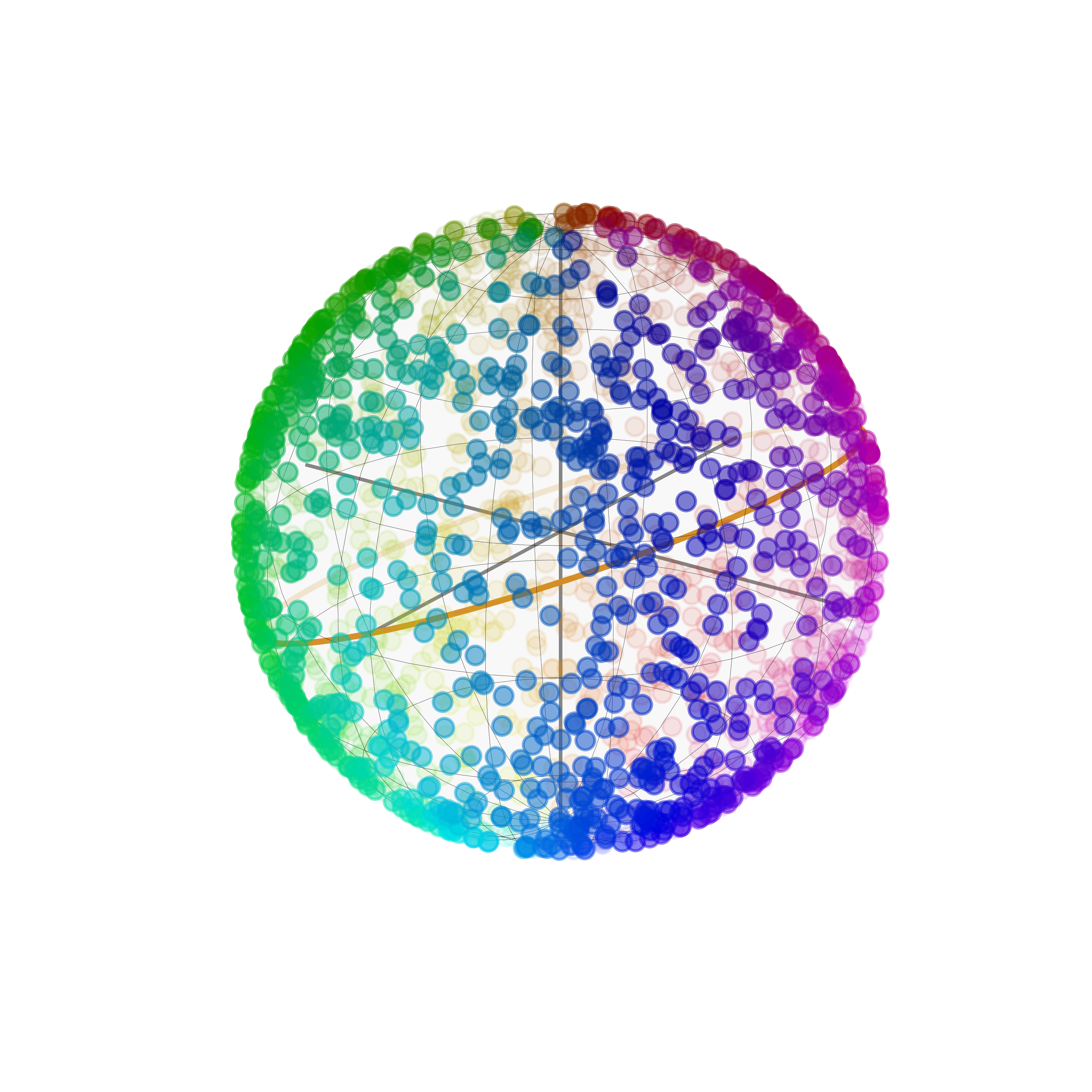}%
\caption{$t=0$}
\end{subfigure}%
\hspace{2.5em}%
\begin{subfigure}{.28\textwidth}%
\includegraphics[width=\linewidth, trim={2cm 2cm 2cm 2cm}, clip]{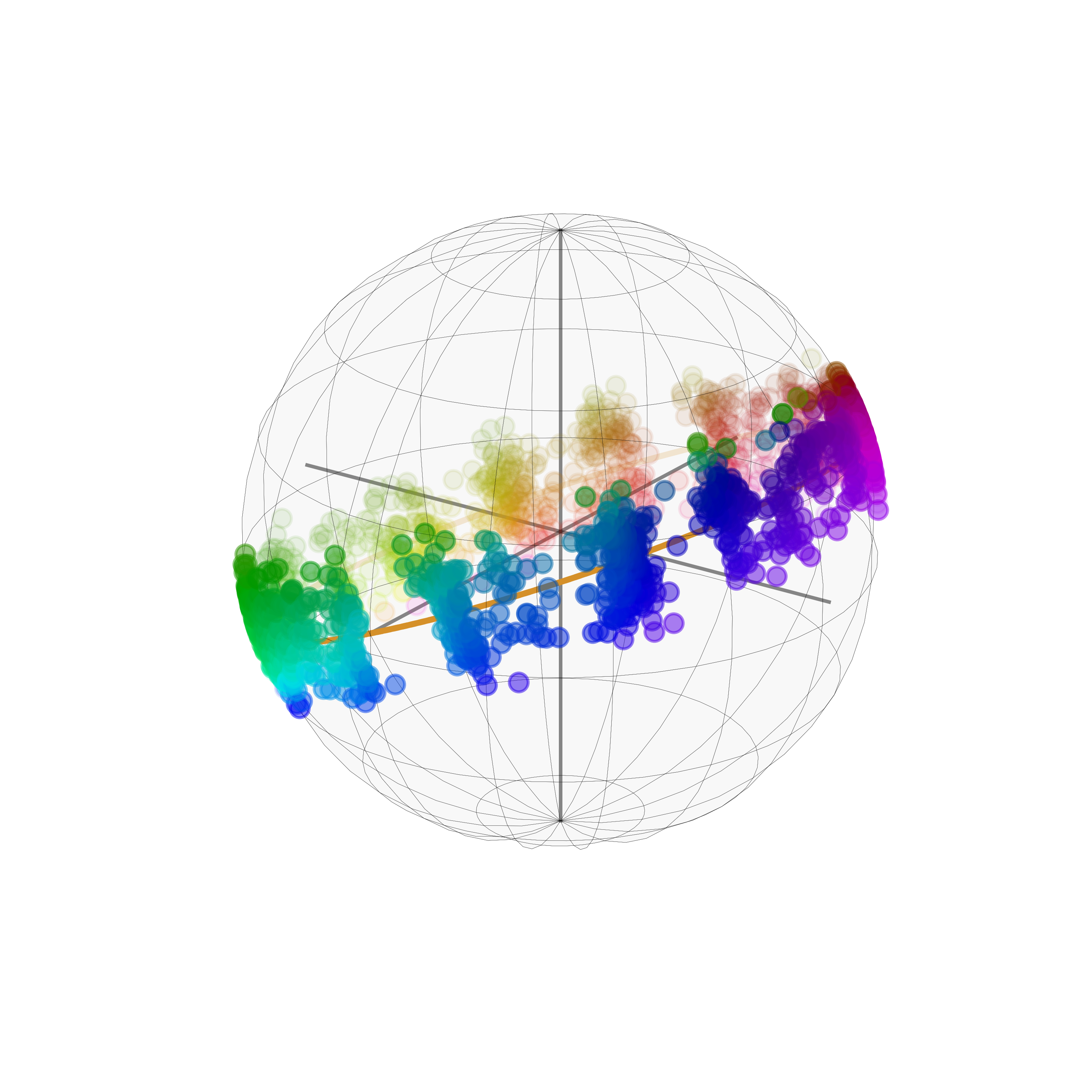}%
\caption{$t=2.5$}
\end{subfigure}%
\hspace{2.5em}%
\begin{subfigure}{.28\textwidth}%
\includegraphics[width=\linewidth, trim={2cm 2cm 2cm 2cm}, clip]{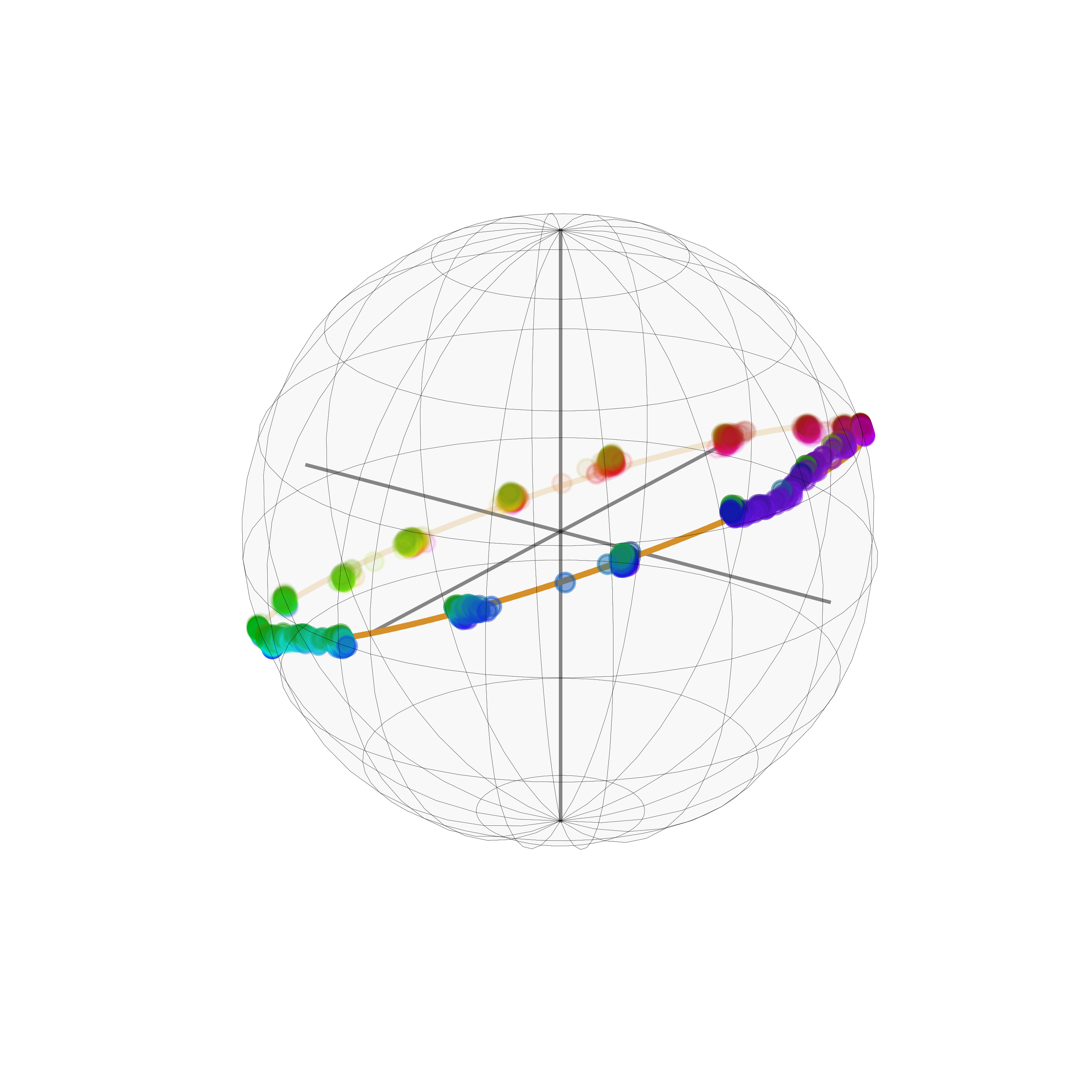}%
\caption{$t=5$}%
\end{subfigure}\\%
\caption{Illustration of Theorem~\ref{thm:main}.
We observe concentration of $\rho_t^{\beta,n}$ at $E\cap \sphere$ ({\color{myorange}orange circle}) when running \eqref{eq:SAdynamics} with $n=2,000$ tokens to simulate \eqref{eq:SAdynamics_meanfield} for $B=R^\top\!\operatorname{diag}(5,5,1)R$ and $V=\mathrm{Id}$, where $R$ is a rotation by $\pi/8$. We use $\beta=\valbeta$ and a von Mises--Fisher mixture as initial density $\rho_0$.}
\label{fig:thm1}
\end{figure}
Before discussing Theorem~\ref{thm:main} and its assumptions in the remainder of this section, let us provide as a corollary a quantitative convergence estimate for the zero-temperature transformer continuity equation~\eqref{eq:zerotemp_meanfield}. It follows directly from the proof of Theorem~\ref{thm:main} (see equation~\eqref{eq:proof:main:B}) in Section~\ref{sec:proof:main}.

\begin{corollary}\label{cor:exponential_convergence}
    Assume that the matrices~$B$ and $V$ satisfy Assumption~\ref{asm:weights}. Moreover, let the initial distribution $\rho_0$ be such that Assumption~\ref{asm:Vp_0} holds for some $p\in (0, 1]$. Let $(\rho_t)_{t\in[0,T]}$ denote the weak solution to the zero-temperature transformer continuity equation~\eqref{eq:zerotemp_meanfield}.
    Then,
    it holds
    \begin{equation*}
        W_2(\rho_t,\Pi_\sharp\rho_0)
        \leq 
        \sqrt{2 \mathcal{V}_p(\rho_0)}\exp{\left(-\tfrac{p\gamma}{\sigma_{\max}(B)} t\right)}.
    \end{equation*}
\end{corollary}
To the best of our knowledge, Corollary~\ref{cor:exponential_convergence} is the first result which identifies the stationary distribution of the zero-temperature transformer continuity equation~\eqref{eq:zerotemp_meanfield}. Note, however, that as observed in \cite[Remark 3.6]{bruno2025multiscale} there exist weight matrices $B$ and $V$ (which do not satisfy \ref{asm:VBTsymmetric} in Assumption~\ref{asm:weights}) such that only the collapse onto $E$ but no convergence can be expected. 

\paragraph{Discussion of the assumptions.}
Demanding invertibility of the product $B=Q^\top K$ of the key and query matrices in \ref{asm:Binvertible} of Assumption~\ref{asm:weights} is relatively common in the field, see, e.g., \cite{geshkovski2023emergence,geshkovski2023mathematical,rigollet2025mean,castin2025unified,bruno2025multiscale}. It is especially important when working with the zero-temperature continuity equation \eqref{eq:zerotemp_meanfield} which is not well-defined unless $B$ is invertible. The symmetry assumption on $VB^\top$ from \ref{asm:VBTsymmetric} might appear to be relatively strong. 
It is important to note, however, that without an assumption of this form one cannot hope to be able to characterize the stationary distribution $\Pi_\sharp\rho_0$ for \eqref{eq:zerotemp_meanfield} or prove exponential convergence to it as done in Corollary~\ref{cor:exponential_convergence}. 
Still, the quantitative stability result from Proposition~\ref{prop:difference_continuity_eq}, stating that $W_2(\rho_t^\beta,\rho_t)\leq 2\sqrt{\nicefrac{\log(\beta+1)}{\beta}} (e^{C_1t}-e^{C_0t})$, is true without \ref{asm:VBTsymmetric}. As a consequence, whenever one can analyze the convergence of the zero-temperature solution $\rho_t$ under more relaxed assumptions than \ref{asm:VBTsymmetric}, such results extend to $\rho_t^\beta$ for large $\beta$ by the triangle inequality. See also Appendix~\ref{sec:nonsymmetric} for related numerical experiments.

Moving on to Assumption~\ref{asm:Vp_0}, we note that this assumption implies that $\rho_0(E^\perp\cap\sphere)=0$ since otherwise the denominator would be zero on a set of positive measure.
This is necessary: Since $E^\perp\cap\sphere$ is invariant for \eqref{eq:zerotemp_meanfield} when $VB^\top$ is symmetric, any mass initially supported on $E^\perp$ remains there for all times and cannot flow toward $E$. However, Assumption~\ref{asm:Vp_0} is stronger requiring that not too much mass accumulates close to $E^\perp\cap\sphere$ either. This is necessary to prove Corollary~\ref{cor:exponential_convergence} with Lyapunov techniques.
We suspect that it might be possible to weaken this assumption or even replace it by $\rho_0(E^\perp\cap\sphere)=0$ if one is willing to worsen or sacrifice having a convergence rate altogether.   

Assumption~\ref{asm:lower_bound} is only needed for the quantitative Laplace principle from Lemma~\ref{lem:LaplacePrinciple}. Loosely speaking, for token distributions $\rho_t^\beta$ with zero mass on open subsets of the sphere, the vector field in \eqref{eq:SAdynamics_meanfield} is a poor approximation of the one in \eqref{eq:zerotemp_meanfield} even for large $\beta$. The authors of \cite{bruno2025multiscale} also used Assumption~\ref{asm:lower_bound} to show that $\rho_t^\beta$ has a positive density for all finite times $t>0$, and similar assumptions were used in \cite{bruno2024emergence,rigollet2025mean,chen2025quantitative}.
One possible way to remove this assumption is to add noise, leading to noisy transformer models as recently considered in \cite{balasubramanian2025structure,rigollet2025mean,shalova2026solutions}.
This is also exploited in the analysis of the conceptually similar consensus-based optimization model~\cite{fornasier2024consensus}. Note, however, that the exponential convergence of the solution to the zero-temperature continuity equation \eqref{eq:zerotemp_meanfield} from Corollary~\ref{cor:exponential_convergence} holds without Assumption~\ref{asm:lower_bound} and is, in particular, valid for singular initial distributions concentrated on finitely many points as long as they do not lie on the lower-dimensional sphere $E^\perp\cap\sphere$, see the numerical experiments for $\beta=\infty$ in Appendix~\ref{app:gradflow}.

Let us remark that at first glance it might seem like Assumptions~\ref{asm:Vp_0} and~\ref{asm:lower_bound} contradict each other. After all, the former prevents mass from accumulating close to $E^\perp\cap\sphere$, while the latter demands a strictly positive density. However, this is no contradiction if the parameter $p$ and the dimension of the eigenspace $E$ are in a correct relationship to each other, as demonstrated by the following example.
\begin{example}
    Since Assumption~\ref{asm:lower_bound} implies that $\rho_0\geq \ell_0\,\mathcal{H}^{d-1}\restr\sphere$, it is necessary for Assumption~\ref{asm:Vp_0} to verify whether $\int_\sphere\!\N{(\mathrm{Id}-\PE)x}^{2p}\N{\PE x}^{-2p}\!\de\mathcal{H}^{d-1}<\infty$.
    In fact, it suffices to check whether $\int_\sphere\!\N{\PE x}^{-2p}\!\de\mathcal{H}^{d-1}(x)<\infty$.
    Note that, for $p >0$, the singular set of the integrand is \mbox{$S\coloneqq E^\perp\cap\sphere$} and equals a $(d-k-1)$-dimensional sphere. In a neighborhood of $S$ the quantity $\N{\PE x}$ scales like the distance of $x$ to $S$. The integral over a tubular $\eps$-neighborhood $S_\eps$ of $S$ in spherical coordinates therefore scales like
    $\int_{S_\eps}\!\N{\PE x}^{-2p}\!\de\mathcal{H}^{d-1}
    \sim \int_0^\eps \frac{1}{r^{2p}} r^{k-1}\de r
    = \int_0^\eps r^{k-1-2p}\de r$, which is finite if and only if $k>2p$. While for $p<\tfrac{1}{2}$ this condition is automatically satisfied, it can be a restriction for larger~$p$. For instance, if $p=1$ it requires the dimension of the top eigenspace~$E$ of $VB^\top$ to be at least three. Interestingly, the condition $\dim E\geq 3$ also appeared in \cite[Theorem 3.1]{chen2025quantitative} for analyzing clustering phenomena of mean-field transformer models.
\end{example}

\paragraph{Discussion of the main results.}
Our main result, Theorem~\ref{thm:main}, encapsulates two different effects that are at play for solutions to the transformer continuity equation~\eqref{eq:SAdynamics_meanfield}. The first one is a temperature effect stating that for large values of $\beta$ (corresponding to low temperatures), solutions $\rho_t^\beta$ of \eqref{eq:SAdynamics_meanfield} are close to solutions $\rho_t$ of the zero-temperature continuity equation \eqref{eq:zerotemp_meanfield} in the Wasserstein distance. In fact, Proposition~\ref{prop:difference_continuity_eq} in Section~\ref{sec:proof_main:stability} states that $W_2(\rho_t^\beta,\rho_t)\leq 2\sqrt{\nicefrac{\log(\beta+1)}{\beta}}(e^{C_1t}-e^{C_0t})$, which quantifies \cite[Theorem 3.2]{bruno2025multiscale}. This estimate is not uniform in time due to the exponential factor, and we do not expect it to be either, as demonstrated by our numerical experiments in Section~\ref{sec:numericalexperiments}. The second effect is the exponential convergence of the zero-temperature limit~\eqref{eq:zerotemp_meanfield} to a stationary distribution $\rho_\infty$ which we show to depend explicitly on the initial datum and to have the form $\rho_\infty\coloneqq\Pi_\sharp\rho_0$.
Taking into account the definition of $\Pi$ in \eqref{eq:Pi},
$\rho_\infty$ can be thought of as a projection of $\rho_0$ onto $E\cap\sphere$.

The time scales at which these two effects are present can be derived directly from Theorem~\ref{thm:main} by investigating which of the two terms dominates. Until times of order $\mathcal{O}(\log\beta)$,
the exponential convergence prevails. Thereafter, the approximation error due to positive temperature takes over, see also Figure~\ref{fig:alignment} for a numerical illustration of this behavior.
\begin{equation*}
    W_2(\rho_t^\beta,\Pi_\sharp\rho_0)
    \lesssim 
    \begin{dcases}
        \exp\left(-\tfrac{p\gamma}{\sigma_{\max}(B)}t\right),
    \qquad
    &0\leq t \lesssim \log\beta,\\
    e^{C_1t}\sqrt{\tfrac{\log(\beta+1)}{\beta}},
    \qquad
    &t \gtrsim\log\beta.
    \end{dcases}
\end{equation*}
Hence, the most relevant regime for our result is that of time scales of at most $\mathcal{O}(\log\beta)$.
We empirically validate this finding in Appendix~\ref{sec:empirical_O_beta} by showing that the time $t^\beta$ at which $W_2(\rho_t^\beta,\Pi_\sharp\rho_0)$ is minimal indeed scales like $\mathcal{O}(\log\beta)$.

On such scale, the following corollary of Theorem \ref{thm:main} states
that for any given target accuracy~$\eps>0$, there exists $\beta>0$ large enough such that after a time of order $\log\frac1\eps$, the Wasserstein distance of $\rho_t^\beta$ to $\Pi_\sharp\rho_0$ is below $\eps$ in a proper time interval. See Appendix~\ref{sec:proof:cor:time_scales} for the proof of the following result.
\begin{corollary}
    \label{cor:time_scales}
    For any $\eps>0$ and for all $\beta\gg\operatorname{poly}\left(\frac{1}{\eps}\right)\log\frac{1}{\eps}$ large enough it holds $W_2(\rho_t^\beta,\Pi_\sharp\rho_0)\leq\eps$
    for all $t\in[t_1,t_2]$ with $0\leq t_1<t_2$, where
    \begin{align*}
        t_1 \coloneqq \tfrac{\sigma_{\max}(B)}{p\gamma}\max\left(\log\left(
        \tfrac{2\sqrt{2\mathcal{V}_p(\rho_0)}}{\eps}\right),0\right)
        \quad 
        \text{ and }
        \quad 
        t_2 \coloneqq \tfrac{1}{C_1}\log\left(1+\tfrac{\eps}{4}\sqrt{\tfrac{\beta}{\log(\beta+1)}}\right).
    \end{align*}
\end{corollary}

\paragraph{Conjectured long-time behavior.}
We close this section by emphasizing that there is no reason to believe that after a time scale of order $\mathcal{O}(\log\beta)$ the token distribution $\rho_t^\beta$ is still close to $\Pi_\sharp\rho_0$.
In fact, we have numerical evidence suggesting convergence to the top eigenspace $F$ of $V$, that is, the eigenspace associated with its largest eigenvalue, or to the dominant eigenspace $F^{\mathrm{abs}}$ of $V$ associated with an eigenvalue of largest magnitude for large times in certain cases (see Figure~\ref{fig:conj}). This observation is supported by the asymptotic stability results from \cite{altafini2026multistability}, and it neither contradicts our result nor the three-phase analysis of \cite{bruno2025multiscale}.
Phases 2 and 3 described therein apply only once the distribution $\smash{\rho_t^\beta}$ has fully collapsed onto $E \cap \sphere$, and they are derived under the assumption that $V = \mathrm{Id}$. Under this assumption, the following conjecture becomes partially trivial, since $F = \mathbb{R}^d$ and hence $F \cap \sphere = \sphere$, while the convergence statement of the conjecture remains non-trivial. 
We provide heuristics and numerical experiments supporting the conjecture in Appendices~\ref{app:conjecture} and \ref{app:expSupportConj}, respectively.
\begin{conjecture}\label{conj:v_max(V)}
    There exist weight matrices $B$ and $V\neq \mathrm{Id}$ satisfying Assumption~\ref{asm:weights},
    and a probability measure $\rho_\infty\in\mathcal{P}(\sphere)$ supported on $F\cap\sphere$ or $F^{\mathrm{abs}}\cap\sphere$ such that $W_2(\rho_t^\beta,\rho_\infty)\to 0$ as $t\to\infty$.
\end{conjecture}

\begin{figure}[h!]
\def\valbeta{30}%
\centering%
\hphantom{----}%
\begin{subfigure}{.23\textwidth}%
\includegraphics[width=\linewidth, trim={2cm 2cm 2cm 2cm}, clip]{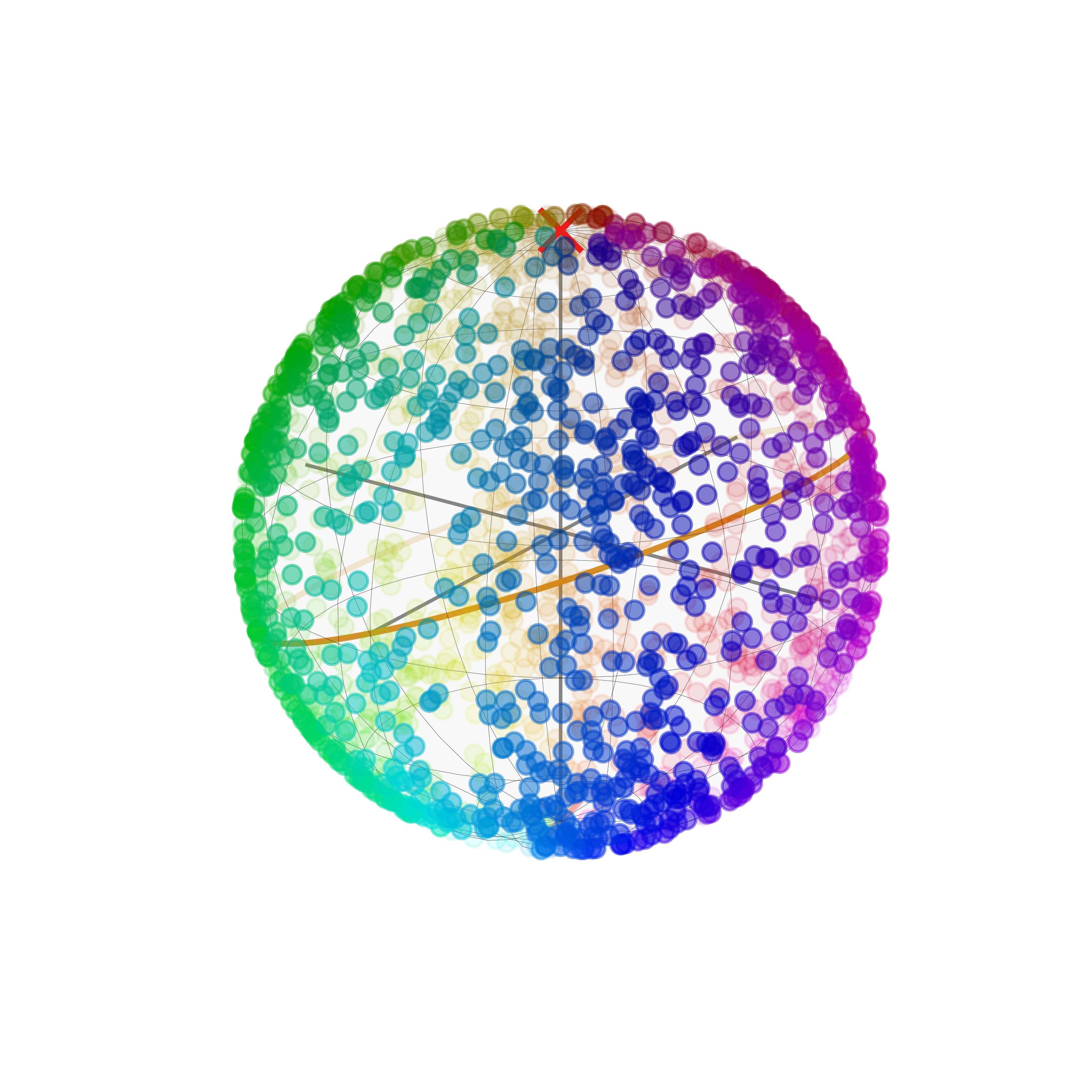}%
\caption{$t=0$}
\end{subfigure}%
\hfill%
\begin{subfigure}{.23\textwidth}%
\includegraphics[width=\linewidth, trim={2cm 2cm 2cm 2cm}, clip]{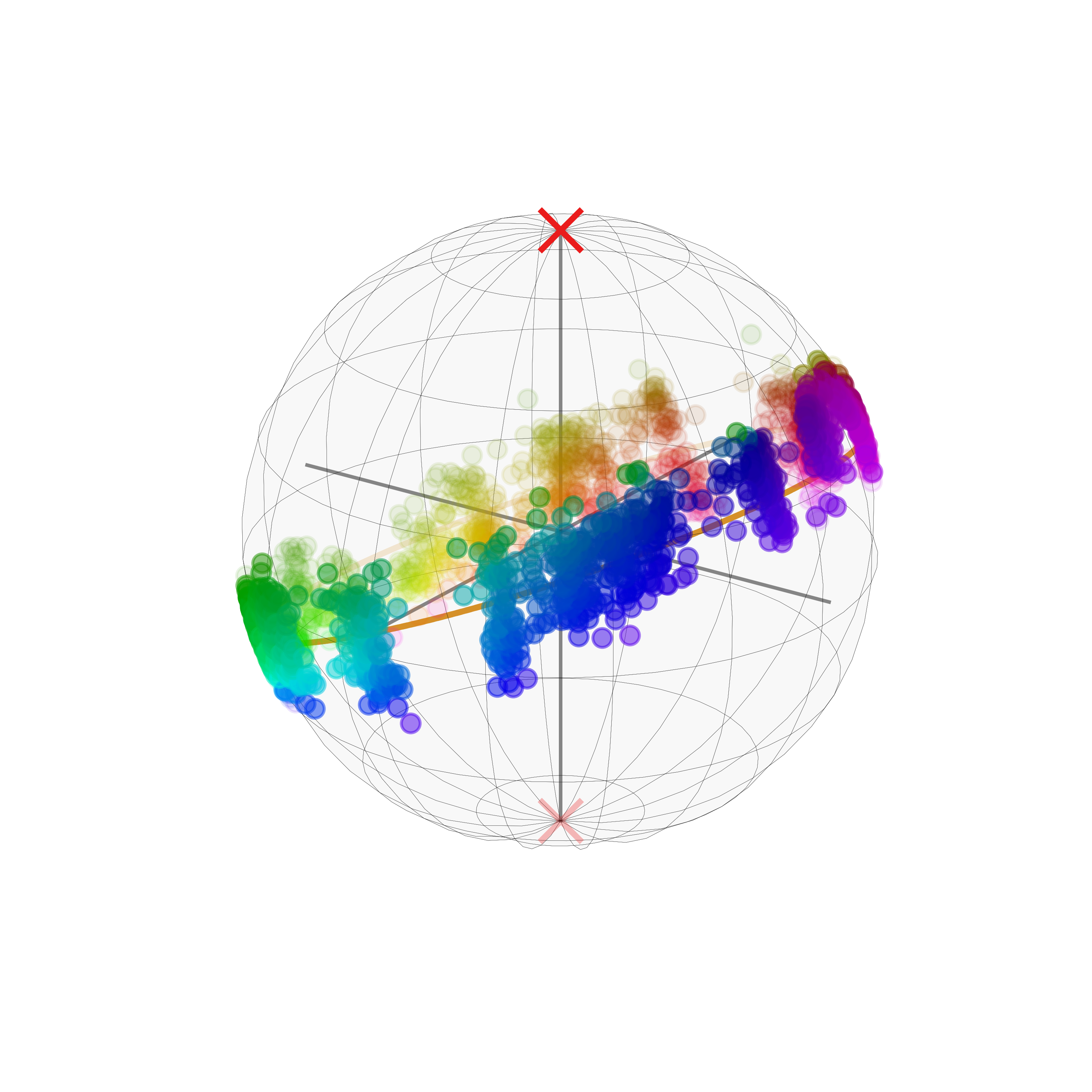}%
\caption{$t=2.5$}
\end{subfigure}%
\hfill%
\begin{subfigure}{.23\textwidth}%
\includegraphics[width=\linewidth, trim={2cm 2cm 2cm 2cm}, clip]{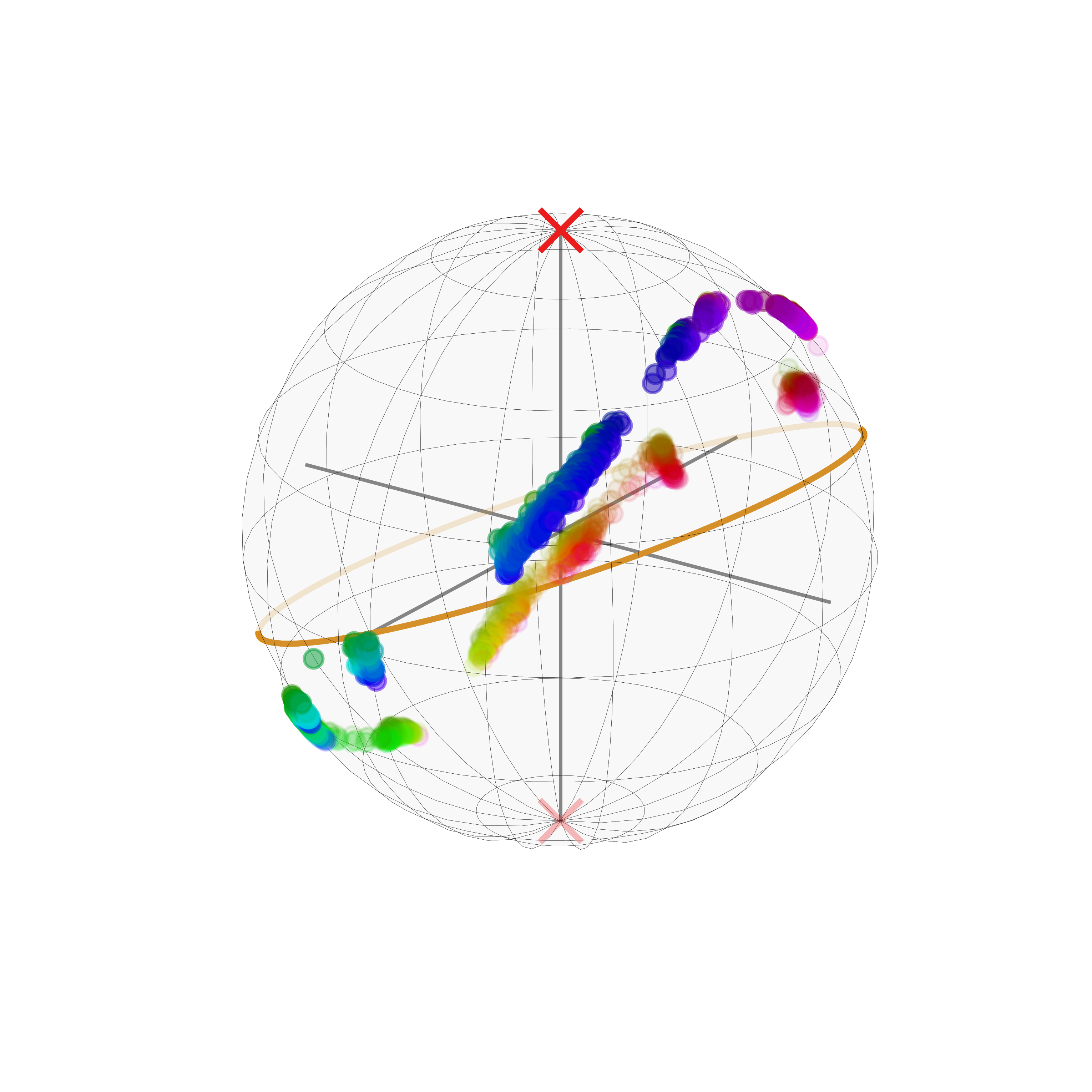}%
\caption{$t=4$}
\end{subfigure}%
\hfill%
\begin{subfigure}{.23\textwidth}%
\includegraphics[width=\linewidth, trim={2cm 2cm 2cm 2cm}, clip]{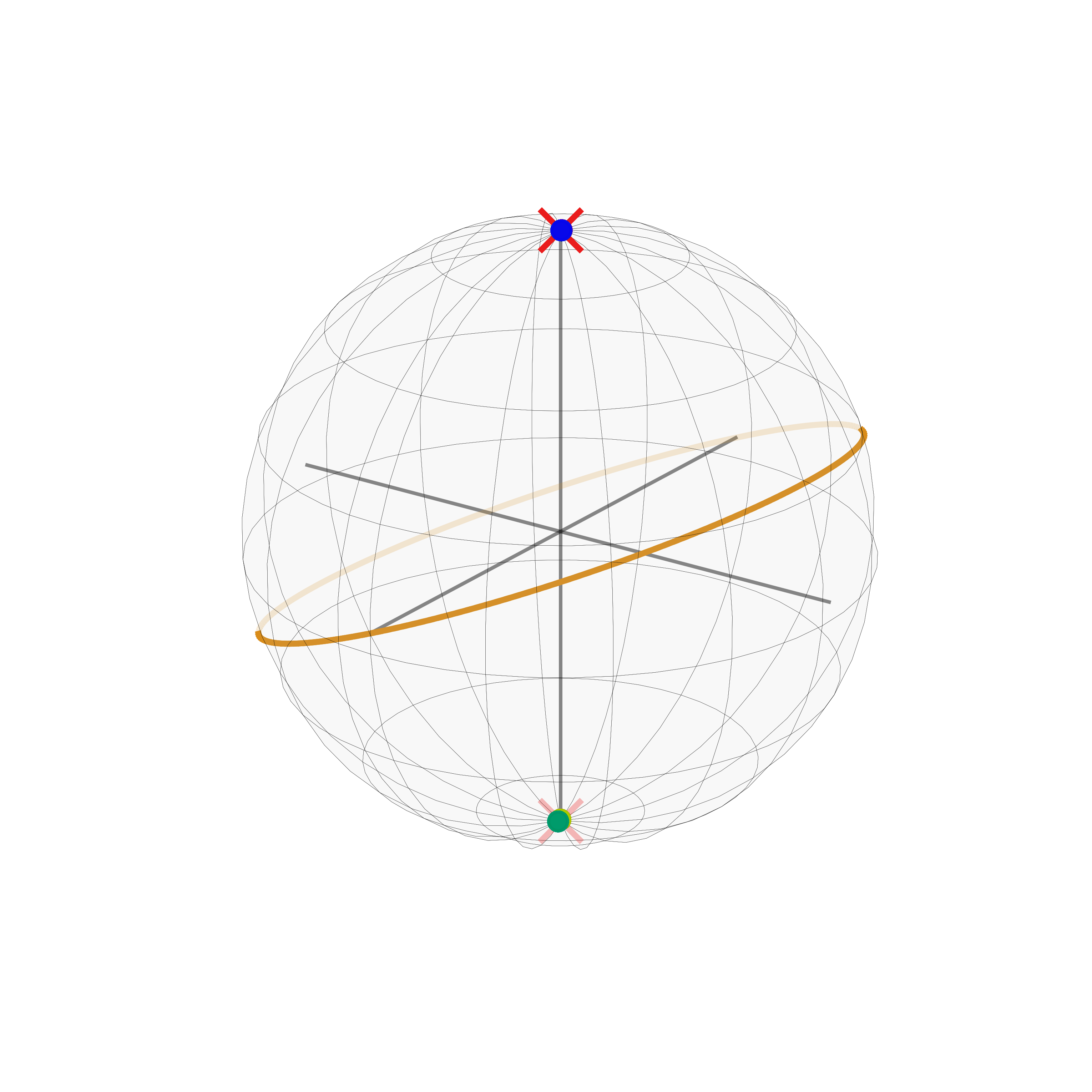}%
\caption{$t=10$}%
\end{subfigure}\hphantom{----}%
\caption{Illustration of Conjecture~\ref{conj:v_max(V)}. Using the setup of Figure~\ref{fig:thm1}, with the same matrix $VB^\top$, but with $V=\operatorname{diag}(1,1,2)$ and $B$ adapted accordingly, we observe initial concentration of $\rho_t^{\beta,n}$ at $E\cap\sphere$ ({\color{myorange}orange circle}) before the terminal clustering in $F\cap\sphere$ ({\color{myred}red crosses}).
}
\label{fig:conj}
\end{figure}

\section{Related work}
\label{sec:relatedworks}

Let us now relate our results to previous works in the literature on the mathematical analysis of transformer models, and outline their connection to a class of interacting particle systems.

\paragraph{Mathematical analysis of transformers.}
A central body of work models transformer~layers~through the lens of continuous dynamical systems and partial differential equations (PDEs).
This perspective was pioneered in \cite{sander2022sinkformers} by linking attention to transport-type PDEs, and expanded in \cite{geshkovski2023mathematical} by incorporating key architectural components such as layer normalization within a mean-field framework.
More recent contributions such as \cite{castin2025unified} provide a unified treatment of these approaches. Together, these works yield a concise mathematical description of transformer dynamics that serves also as our modeling foundation. The authors of \cite{bruno2025multiscale} identify the zero-temperature limit \eqref{eq:zerotemp_meanfield} and establish convergence towards it as $\beta^{-1} \to0$. However, their result is not quantitative, i.e., no rate in $\beta$ nor its interplay with inference time $t$ is provided. The analysis further identifies additional phases of the dynamics once the solution concentrates on the top eigenspace~$E$ of $VB^\top$. In contrast, our work takes low but non-zero temperature effects into account and provides quantitative convergence guarantees. 
This perspective is motivated by the numerical evidence in Section~\ref{sec:numericalexperiments} which suggests that $\beta^{-1}>0$ may prevent full collapse onto the eigenspace~$E$. 
Related to this, \cite{burger2025analysis} also focuses on equation~\eqref{eq:SAdynamics_meanfield} with $\beta^{-1}>0$ and characterizes the stationary states of \eqref{eq:SAdynamics_meanfield} in the case $B = \pm V$, where the system can be interpreted as a gradient flow of the energy $\mathcal{E}(\varrho) = \iint e^{\ip{x}{By}}\de\varrho(x)\de\varrho(y)$, either maximizing it if $B=V$ or minimizing it if $B=-V$. We discuss this case in Appendix \ref{app:gradflow} and only emphasize here that our results complement their theory in two ways: we go beyond the gradient flow setting by allowing $B\neq V$ and, in particular, make concrete statements about the asymptotic behavior of $\rho_t^\beta$.
In \cite{massucco2026multi} this gradient flow interpretation was extended to multi-head attention with time-dependent parameter matrices.

Besides mean-field descriptions, a complementary line of research investigates transformer dynamics at the finite-particle level, with an emphasis on qualitative and asymptotic behavior. In particular, clustering and concentration phenomena are analyzed in \cite{geshkovski2023mathematical,rigollet2025mean,chen2025quantitative}, with extensions to causal attention settings in \cite{karagodin2024clustering,abella2024asymptotic}. The long-time behavior of transformer dynamics has also been studied, including the emergence of metastability \cite{geshkovski2024dynamic,alcalde2025attention},
multistability with several equilibria related to the spectrum of the value matrix \cite{altafini2026multistability}, and phase transitions \cite{MR5004864}, as well as propagation-of-chaos results in large-token limits \cite{bruno2024emergence}. Our work differs by considering token distributions with positive densities and providing quantitative estimates of concentration phenomena in an initial non-asymptotic phase of the dynamics where the attention matrix $B$ plays a crucial role.

\paragraph{Interacting particle systems.}
Transformers share striking similarities with classical interacting particle systems (e.g., the Kuramoto model~\cite{polyanskiy2025synchronization,li2026diverse}), and optimization methods (e.g., gradient descent \cite{von2023transformers} or the Frank--Wolfe algorithm \cite{alcalde2025attention}). In particular, they are connected (as noted in \cite{geshkovski2023mathematical}) to \emph{consensus-based optimization} (CBO)~\cite{pinnau2017consensus,riedl2024perspective} which we formalize and exploit in what follows. In line with our setting, we consider CBO constrained to the sphere, as in \cite{fornasier2020consensus}.
The aim is to locate a global minimizer of some objective function $J:\sphere\to\R$ by evolving a particle ensemble $\{x_i\}_{i=1}^n$ that follows a coupled system of stochastic differential equations of the form
\begin{equation}\label{eq:CBO}
\dot{x}_i(t) =  \proj{x_i(t)}\big(m_{\beta,\rho^{\beta,n}_t}\big)
+ \sigma\cdot \operatorname{noise}.
\end{equation}
The so-called consensus point $m_{\beta,\varrho}$, defined for any probability measure $\varrho\in\mathcal{P}(\sphere)$, is a weighted mean to which particles with smaller values of $J$ contribute more. In fact, particle weights can also be determined through a pairwise interaction via polarization, as proposed in~\cite{bungert2025polarized}.
Mathematically, this amounts to kernelizing the consensus point as
\begin{equation}
\label{eq:polarized_m}
    m_{\beta,\varrho}^\kappa(x) = \int_{\sphere} \frac{\kappa(x, y) \exp{(-\beta J(y))}}{\int_{\sphere} \kappa(x, z) \exp{(-\beta J(z))} \de \varrho(z)} y \de \varrho(y),
\end{equation}
where $\kappa:\R^d\times\R^d\to\R$ is some kernel.
With $\N{\dummy}_B\coloneqq\sqrt{\langle \dummy, B\, \dummy \rangle}$ being the energy-norm, we choose
\begin{equation*}
    J(x) = -\tfrac{1}{2}\N{x}_{B}^2, \quad \text{and} \quad \kappa(x,y) =  \exp{\left(-\tfrac{\beta}{2}\N{x - y}_{B}^2\right)}.
\end{equation*}
In the noiseless case $\sigma=0$, equation~\eqref{eq:CBO} with consensus point as in \eqref{eq:polarized_m} leads exactly to the transformer model~\eqref{eq:SAdynamics} when $V\equiv\mathrm{Id}$ and $B\equiv Q^\top K$ is time-independent. This interpretation~motivated us to leverage techniques from the mean-field analysis of interacting particle systems for optimization---most prominently, the quantitative Laplace principle and Lyapunov techniques---to analyze the transformer model \eqref{eq:SAdynamics_meanfield}. Apart from this, by embedding the self-attention dynamics~\eqref{eq:SAdynamics} into the rich family of CBO models~\cite{fornasier2024consensus,riedl2022leveraging,bungert2025polarized}, one can hope to derive more expressive transformers by changing the kernel, the objective function, or by adding additional drift terms and noise (the latter of which leads to noisy transformers as previously investigated in, e.g., \cite{shalova2026solutions,balasubramanian2025structure,engel2026random,shalova2025noisy,peletier2025nonlinear,koubbi2026homogenized,rigollet2025mean,agazzi2026stochasticscalinglimitssynchronization}).

\section{Proof details for the main results}
\label{sec:proof:main}

The proof of Theorem~\ref{thm:main} relies on two key ingredients:
quantitative stability estimates for continuity equations in Wasserstein space which leverage a quantitative Laplace principle,
and Lyapunov-type estimates for the Wasserstein distance to control the convergence of the zero-temperature transformer continuity equation~\eqref{eq:zerotemp_meanfield}. While the former allow to control the Wasserstein distance between the flows~$\rho_t^\beta$ and $\rho_t$ (see Proposition~\ref{prop:difference_continuity_eq}),
the latter enable us to quantify the convergence rate of $\rho_t$ to $\Pi_\sharp\rho_0$ (see Proposition~\ref{prop:Lyapunov_zerotemp_p} in combination with Lemmas~\ref{lem:W2_to_Pi_p} and \ref{lem:Pi_dist}).
Before presenting these key ingredients in Sections~\ref{sec:proof_main:stability} and \ref{sec:proof_main:Lyapunov}, we give the proof of Theorem~\ref{thm:main}, which demonstrates how those tools are used.
\begin{proof}[Proof of Theorem~\ref{thm:main}]
    Denoting by $(\rho_t^\beta)_{t\in[0,T]}$ the solution of \eqref{eq:SAdynamics_meanfield},
    and letting $(\rho_t)_{t\in[0,T]}$ denote the solution of \eqref{eq:zerotemp_meanfield} with coinciding initial datum $\rho_0=\rho_0^\beta$, we can bound by the triangle inequality
    \begin{equation}
        \label{eq:proof:thm:main:0}
        W_2(\rho_t^\beta,\Pi_\sharp\rho_0)
        \leq W_2(\rho_t^\beta,\rho_t)
        +
        W_2(\rho_t,\Pi_\sharp\rho_0).
    \end{equation}
    The first term on the right-hand side of \eqref{eq:proof:thm:main:0} can be estimated by employing the quantitative stability estimate in Wasserstein space from Proposition~\ref{prop:difference_continuity_eq} for the continuity equations~\eqref{eq:SAdynamics_meanfield} and \eqref{eq:zerotemp_meanfield},
    yielding
    \begin{equation}
        \label{eq:proof:main:A}
        W_2(\rho_t^\beta,\rho_t)
        \leq 
        2 \sqrt{\tfrac{\log(\beta+1)}{\beta}} \left(e^{C_1t}-e^{C_0t}\right)
    \end{equation}
    for all $t\in[0,T]$. For the second term on the right-hand side of \eqref{eq:proof:thm:main:0}, we first note that thanks to Proposition~\ref{prop:Lyapunov_zerotemp_p} it holds for the functional $\mathcal{V}_p$ defined in \eqref{eq:VpRp} that
    \begin{equation}
        \label{eq:proof:main:10}
        \mathcal{V}_p(\rho_t)
        \leq \mathcal{V}_p(\rho_0)
        \exp\left(-\tfrac{2p\gamma}{\sigma_{\max}(B)} t\right)\!.
    \end{equation}
    With $\mathcal{V}_p(\rho_0)<\infty$ as of Assumption~\ref{asm:Vp_0},
    we have that $\mathcal{V}_p(\rho_t)<\infty$ for all $t\in[0,T]$.
    This implies that $\rho_t(\sphere \cap E^\perp)=0$ for all $t\in[0,T]$, since otherwise the denominator in the integrand of $\mathcal{V}_p$ would be zero on a set of positive measure.
    Since according to Lemma~\ref{lem:Pi_dist} it furthermore holds that $\Pi_\sharp \rho_t = \Pi_\sharp \rho_0$ for all $t\in [0,T]$,
    an application of Lemma~\ref{lem:W2_to_Pi_p} together with the bound \eqref{eq:proof:main:10} shows that
    \begin{equation}
        \label{eq:proof:main:B}
        W_2(\rho_t,\Pi_\sharp\rho_0)
        = W_2(\rho_t,\Pi_\sharp\rho_t)
        \leq \sqrt{2\mathcal{V}_p(\rho_t)}
        \leq \sqrt{2\mathcal{V}_p(\rho_0)}
        \exp\left(-\tfrac{p\gamma}{\sigma_{\max}(B)} t\right)\!,
    \end{equation}
    verifying that $\sqrt{2\mathcal{V}_p(\rho_t)}$ is a suitable Lyapunov functional for $W_2(\rho_t,\Pi_\sharp\rho_0)$.
    Inserting \eqref{eq:proof:main:A} and \eqref{eq:proof:main:B} into \eqref{eq:proof:thm:main:0} concludes the proof.
\end{proof}

\subsection{Quantitative estimates for the Wasserstein distance between the flows \texorpdfstring{$\rho_t^\beta$}{rhot beta} and \texorpdfstring{$\rho_t$}{rhot}}
\label{sec:proof_main:stability}

To estimate the Wasserstein distance $W_2(\rho_t^\beta,\rho_t)$ of the solutions~$\rho_t^\beta$ and $\rho_t$ of the transformer continuity equations \eqref{eq:SAdynamics_meanfield} and \eqref{eq:zerotemp_meanfield} in terms of the temperature parameter $\beta$, we develop a stability estimate in Wasserstein space,
which is based on a quantitative version~\cite{fornasier2024consensus} of the Laplace principle~\cite{dembo2009large}. Denoting by $\varrho\in\mathcal{P}(\sphere)$ a fully supported probability measure, for any fixed $x\in\sphere$ the vector
\begin{equation}
    \label{eq:consensus_point}
    m_{\beta,\varrho}(x)
    \coloneqq 
    \int_{\sphere} \frac{\exp\left(\beta\ip{B^\top x}{y}\right)}{\int_{\sphere} \exp\left(\beta\ip{B^\top x}{z}\right)\!\de\varrho(z)}y\de\varrho(y)
\end{equation}
converges as $\beta\to\infty$ to a maximizer of the function $\sphere\ni y\mapsto\ip{B^\top x}{y}$, which is given by
\begin{equation}
    \label{eq:y*}
    y^*(x)
    \coloneqq
    \frac{B^\top x}{\N{B^\top x}},
    \qquad
    x\in\sphere,
\end{equation}
if $B$ is invertible. More rigorously, the next lemma holds; see Appendix~\ref{app:sec:Laplace} for the proof.
\begin{lemma}
\label{lem:LaplacePrinciple}
    Assume that the matrix $B$ satisfies \ref{asm:Binvertible}.
    Let $\varrho\in\mathcal{P}(\sphere)$, $r>0$, and define the spherical cap $\mathcal{B}_r(y^*(x))\coloneqq \left\{y\in\sphere : \N{y-y^*(x)} \leq r \right\}$. 
    Then, for every $\beta>0$ and every $q>0$, it holds
    \begin{equation}
        \label{eq:QL-sphere}
        \Nbig{m_{\beta,\varrho}(x)-y^*(x)}
        \leq
        \sqrt{r^2+\tfrac{2q}{\sigma_{\min}(B)}} + \tfrac{2e^{-\beta q}}{\varrho \left(\mathcal{B}_r(y^*(x))\right)}.
\end{equation}
\end{lemma}
Using this result, we quantify the convergence $\rho_t^\beta\rightarrow\rho_t$ as $\beta\to\infty$, which has been shown earlier in \cite[Theorem 3.2]{bruno2025multiscale} under invertibility assumptions on $B$ and $V$ as well as Assumption~\ref{asm:lower_bound} on the initial datum $\rho_0$.
We show the following; see Appendix~\ref{app:sec:stability} for~the~proof.
\begin{proposition}
    \label{prop:difference_continuity_eq}
    Assume that the matrix $B$ satisfies \ref{asm:Binvertible}.
    Let $(\rho_t^\beta)_{t\in[0,T]}$ and $(\rho_t)_{t\in[0,T]}$ denote weak solutions to the continuity equations~\eqref{eq:SAdynamics_meanfield} and \eqref{eq:zerotemp_meanfield} with coinciding initial datum $\rho_0^\beta=\rho_0$,
    which is assumed to satisfy Assumption \ref{asm:lower_bound}.
    Then, there exist constants $C_1>C_0>0$ that depend only on $d$, $\sigma_{\min}(B)$, $\sigma_{\max}(B)$, $\sigma_{\max}(V)$, and $\ell_0$, such that for every $\beta>0$, it holds
    \begin{equation*}
        W_2(\rho_t^\beta,\rho_t)
        \leq
        2 \sqrt{\tfrac{\log(\beta+1)}{\beta}} \left(e^{C_1t}-e^{C_0t}\right).
    \end{equation*}
\end{proposition}

\subsection{Lyapunov-type estimates in Wasserstein space controlling the convergence of \texorpdfstring{$\rho_t$}{rhot} to \texorpdfstring{$\Pi_\sharp\rho_0$}{Pi rho0}}
\label{sec:proof_main:Lyapunov}

To measure how well the token distribution~$\rho_t$ of ~\eqref{eq:zerotemp_meanfield} concentrates near the top eigenspace $E$ of $VB^\top$, let us define for a measure $\varrho\in\mathcal{P}(\sphere)$ and for any $p\in (0, 1]$, the energy (Lyapunov) functional
\begin{equation}
    \label{eq:VpRp}
    \mathcal{V}_p(\varrho)
    \coloneqq \int_{\sphere} R_p(x) \de \varrho(x)
    \quad\text{ with }\quad
    R_p(x)
    \coloneqq
    \frac{\N{(\mathrm{Id}-\PE)x}^{2p}}{\N{\PE x}^{2p}}.
\end{equation}
Here, $R_p$ is equal to zero on $E$ and infinity on $E^\perp$. The following auxiliary result shows that $R_p$ can be regarded as a distance from $E$; see Appendix~\ref{app:sec:Lyapunov_aux} for the proof.
\begin{lemma}
    \label{lem:W2_to_Pi_p}
    Let $E\subset \mathbb R^d$ be a linear subspace
    and let $p\in(0,1]$. 
    Then, for all $x\in \mathbb S^{d-1}\setminus E^\perp$, it holds that
    $\N{x-\Pi(x)}^2 \leq 2 R_p(x)$ and for any probability measure $\varrho\in\mathcal P(\mathbb S^{d-1})$ with $\varrho(E^\perp \cap \sphere)=0$ we have
    \begin{equation*}
        W_2^2(\varrho,\Pi_\sharp\varrho)
        \le 2\int_{\mathbb S^{d-1}} R_p(x)\de\varrho(x)
        = 2\mathcal{V}_p(\varrho).
    \end{equation*}
\end{lemma}
Along the flow~$\rho_t$ of the zero-temperature equation~\eqref{eq:zerotemp_meanfield}, we can quantify the decay of $\mathcal{V}_p(\rho_t)$ as demonstrated in the subsequent result;
see Appendix~\ref{app:sec:Lyapunov} for the proof.
\begin{proposition}
    \label{prop:Lyapunov_zerotemp_p}
    Assume that the matrices~$B$ and $V$
    satisfy Assumption~\ref{asm:weights}.
    Moreover, let the initial distribution $\rho_0$ be such that Assumption~\ref{asm:Vp_0} holds for some $p\in (0, 1]$.
    Let $(\rho_t)_{t\in[0,T]}$ denote the weak solution of \eqref{eq:zerotemp_meanfield} with initial datum $\rho_0$.
    Then, for all $t\in[0,T]$, it holds
    \begin{equation*}
        \mathcal{V}_p(\rho_t)
        \leq
        \mathcal{V}_p(\rho_0)\exp\left(-\tfrac{2 p\gamma}{\sigma_{\max}(B)} t\right)\!.
    \end{equation*}
\end{proposition}
Noticing that the projection~$\Pi$ defined in \eqref{eq:Pi} is invariant in time for a solution~$\rho_t$ of \eqref{eq:zerotemp_meanfield},
as made rigorous by the next lemma, Lemma~\ref{lem:W2_to_Pi_p} justifies that $\mathcal{V}_p$ is indeed a suitable Lyapunov functional for the flow $\rho_t$ since $W_2^2(\rho_t,\Pi_\sharp\rho_0) = W_2^2(\rho_t,\Pi_\sharp\rho_t) \leq 2\mathcal{V}_p(\rho_t)$
with $\mathcal{V}_p(\rho_t)$ decaying exponentially fast thanks to Proposition~\ref{prop:Lyapunov_zerotemp_p}. The proof of Lemma~\ref{lem:Pi_dist} is given in Appendix~\ref{app:lem:Pi_dist}.
\begin{lemma}
    \label{lem:Pi_dist}
    Assume that the matrices~$B$ and $V$
    satisfy Assumption~\ref{asm:weights}.
    Moreover, let the initial distribution $\rho_0$ be such that Assumption~\ref{asm:Vp_0} holds for some $p\in (0, 1]$.
    Let $(\rho_t)_{t\in[0,T]}$ denote the weak solution of \eqref{eq:zerotemp_meanfield} with initial datum $\rho_0$. Then it holds $\Pi_\sharp \rho_t = \Pi_\sharp \rho_0$ for all $t\in [0,T]$. 
\end{lemma}

\section{Numerical experiments}
\label{sec:numericalexperiments}

We provide numerical experiments that validate Theorem~\ref{thm:main}, and investigate the long-time dynamics beyond the non-asymptotic concentration regime.
In particular, we observe the predicted concentration near the top eigenspace $E$ of $VB^\top$ up to time scales of order $\log\beta$, followed by a terminal phase in which tokens concentrate on either $F$ or $F^{\mathrm{abs}}$, consistent with Conjecture~\ref{conj:v_max(V)}.

For our simulations, we use an explicit Euler discretization of \eqref{eq:SAdynamics} with timestep $\Delta t=0.01$.\footnote{We employ the \texttt{CBX} library \cite{bailo2024cbx} to run our normalized self-attention dynamics with suitable objective and kernel function. Our code can be found in the following \texttt{GitHub} repository: \url{https://github.com/albertalcalde/mean-field-transformer-concentration} (DOI: 10.5281/zenodo.23726584).}
The remaining parameters are specified individually for each experiment.
\begin{figure}[h]%
\centering
\def\hp{\hphantom{----}}%
\hp{}\begin{subfigure}{.22\textwidth}
\includegraphics[width=\linewidth,trim={2.5cm 2.5cm 2.5cm 2.5cm}, clip]{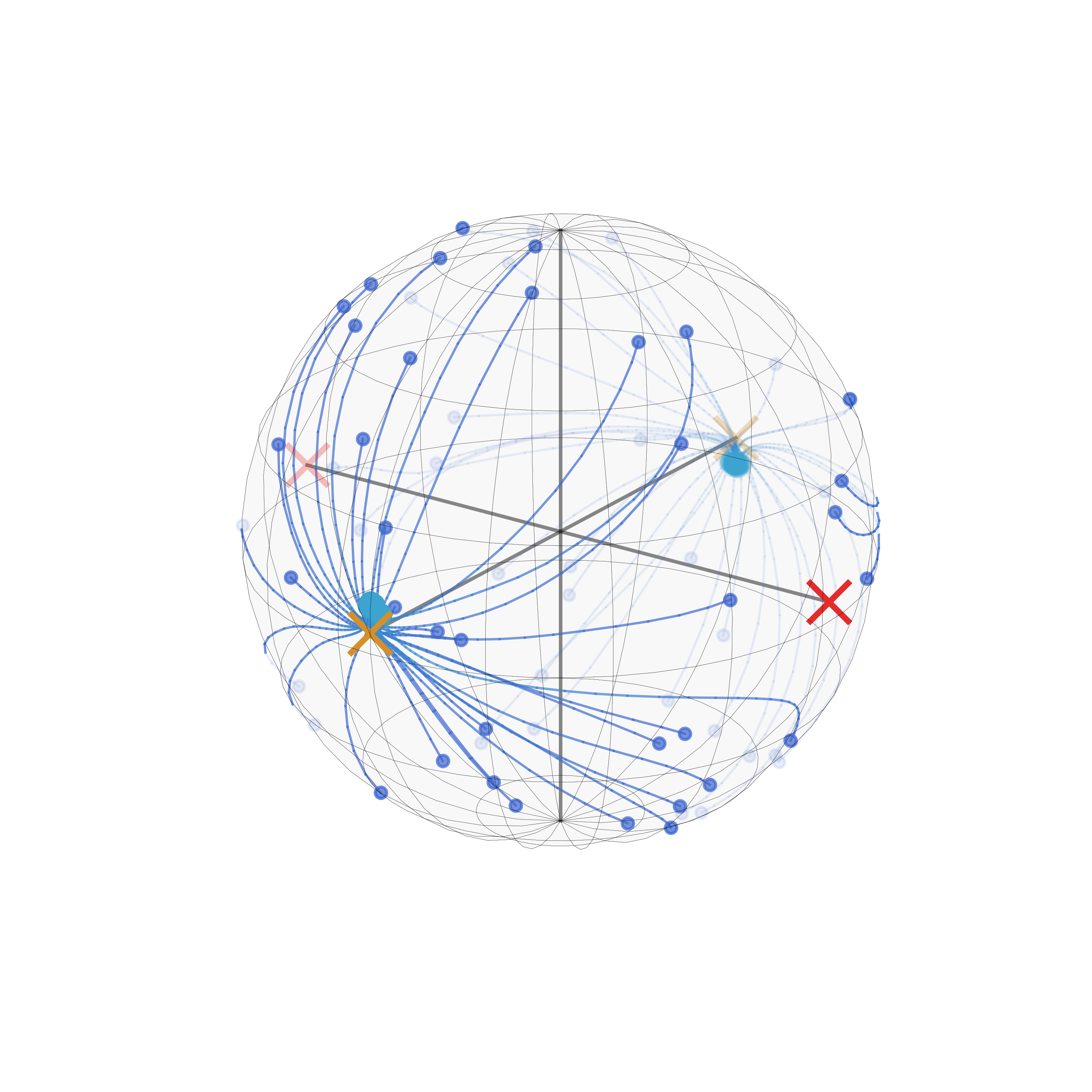}
\caption{$t\in [0,4)$}%
\end{subfigure}\hfill%
\begin{subfigure}{.22\textwidth}%
\includegraphics[width=\linewidth, trim={2.5cm 2.5cm 2.5cm 2.5cm}, clip]{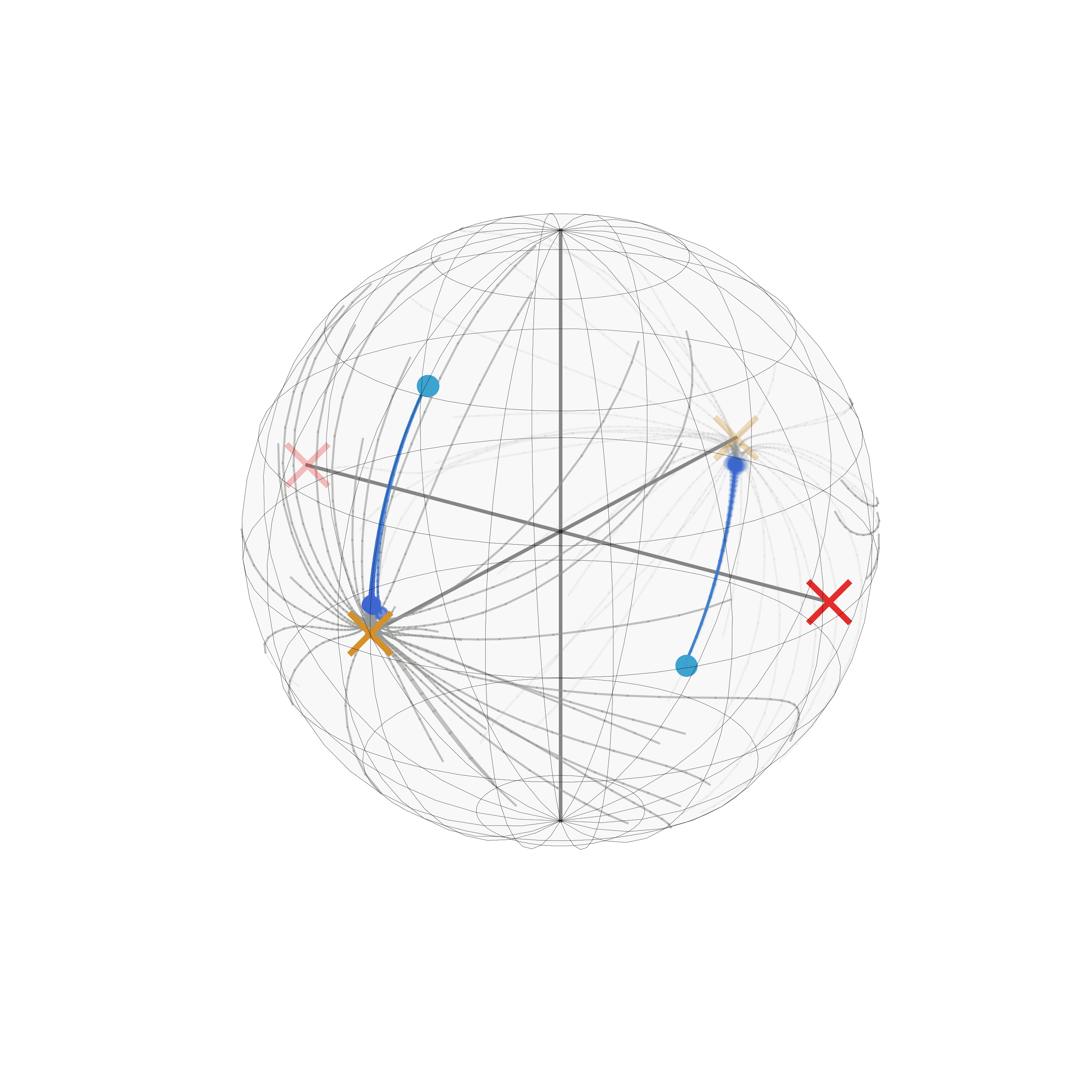}\hspace{2.5em}%
\caption{$t\in [4,9)$}%
\end{subfigure}\hfill%
\begin{subfigure}{.22\textwidth}%
\includegraphics[width=\linewidth, trim={2.5cm 2.5cm 2.5cm 2.5cm}, clip]{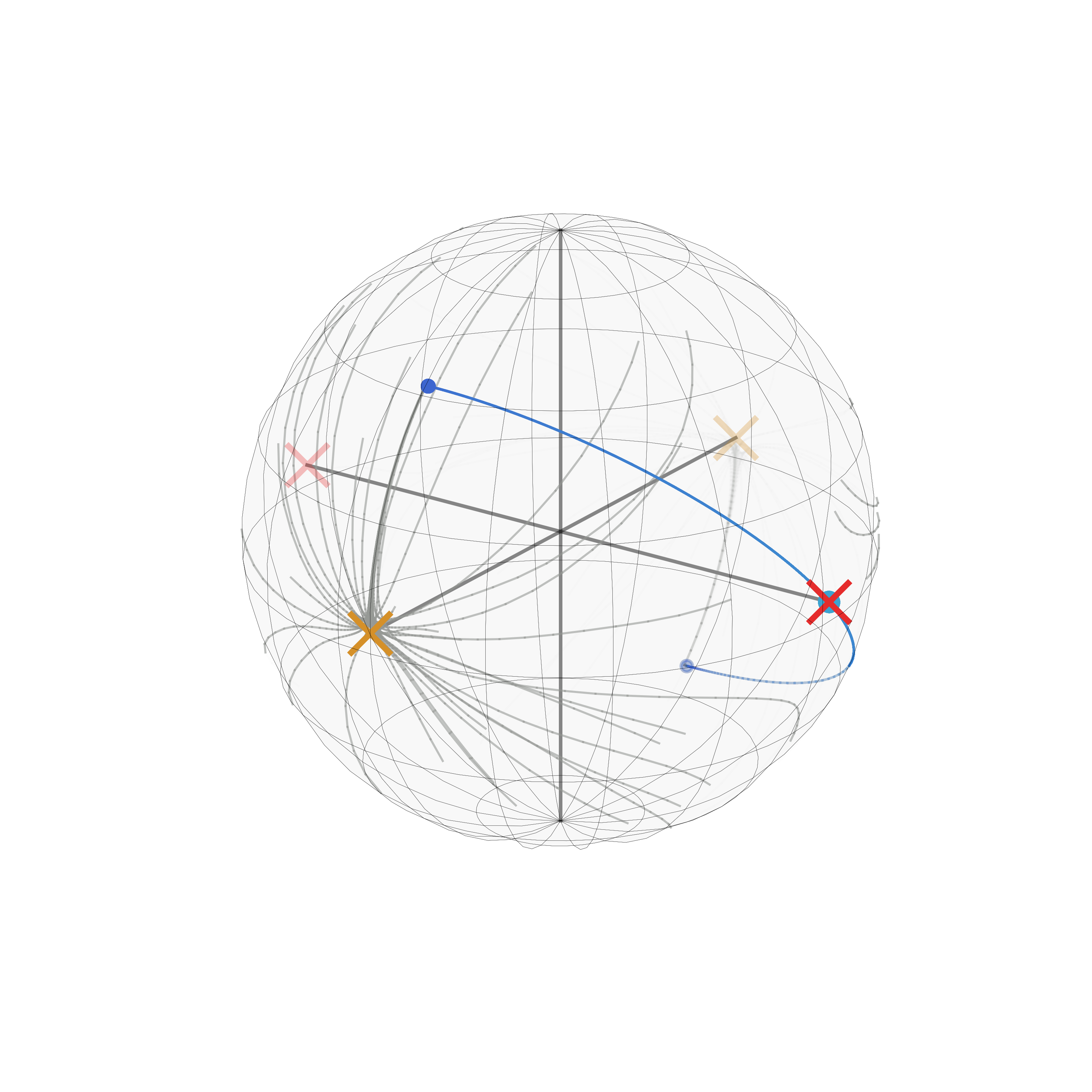}%
\caption{$t\in [9,20)$}%
\end{subfigure}\hp{}%
\caption{Dynamics of $n = 200$ tokens evolving according to \eqref{eq:SAdynamics} with $\beta = 100$. 
We show the top eigenspace $E$ of $VB^\top$ ({\color{myorange}orange crosses}) and the top eigenspace $F$ of $V$ ({\color{myred}red crosses}).\label{fig:twoscales}}
\end{figure}%
We first fix $\beta = 100$, $d = 3$, $n = 200$ tokens, and parameter matrices $V = \operatorname{diag}(-1,1,-2)$ and $B = \operatorname{diag}(-1,-1,1)$ such that $VB^\top = \operatorname{diag}(1, -1, -2)$. 
In Figure~\ref{fig:twoscales}, we depict the dynamics until a time horizon $T\approx 20$. Initially, we observe concentration on the $e_1$-direction which corresponds to the top eigenspace of $VB^\top$. 
On a second time scale, tokens align with the $e_2$-direction, corresponding to $F$, the top eigenspace of $V$, in line with Conjecture \ref{conj:v_max(V)}. Figure~\ref{fig:alignment} quantitatively compares the mean-field dynamics for moderate, large, and infinite inverse temperature~$\beta$. For moderate $\beta$, the solution first approaches the projection $\Pi_\sharp \rho_0$ associated with the top eigenspace $E$ of $VB^\top$, as indicated by the decrease of the Wasserstein distance and the transient growth of the alignment with $E$, as predicted by Theorem \ref{thm:main}. After a time of order $\log \beta$, this metastable behavior gives way to the second time scale, on which the solution aligns instead with the top eigenspace $F$ of $V$. As shown in panel (b) of Figure~\ref{fig:alignment}, increasing $\beta$ delays this transition, while in the limiting case $\beta=\infty$, the second phase is absent by
Corollary~\ref{cor:exponential_convergence} and the dynamics remain aligned with~$E$. We point to Appendix \ref{app:numericalexperiments} for parameter choices of $B$ and $V$ that instead lead to concentration on $F^{\mathrm{abs}}$, the  eigenspace of $V$ associated with an eigenvalue of largest magnitude.
\begin{figure}[h]%
\centering
\includegraphics[width=0.85\textwidth,trim=0cm 0.3cm 0cm 0.2cm,clip]{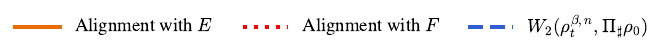}

\begin{subfigure}{.33\textwidth}
\includegraphics[width=\textwidth]{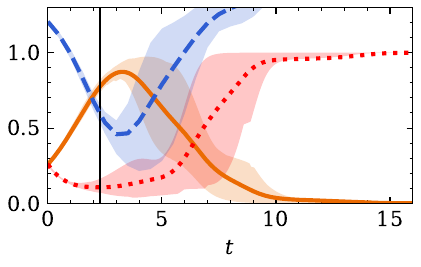}
\caption{$\beta=10$}
\end{subfigure}\hfill%
\begin{subfigure}{.33\textwidth}%
\includegraphics[width=\textwidth]{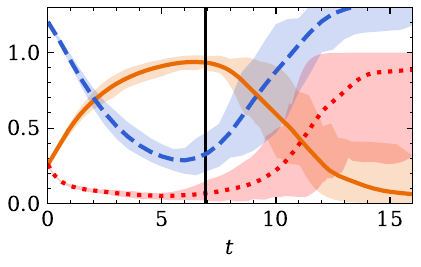}
\caption{$\beta=10^3$}
\end{subfigure}\hfill%
\begin{subfigure}{.33\textwidth}%
\includegraphics[width=\textwidth]{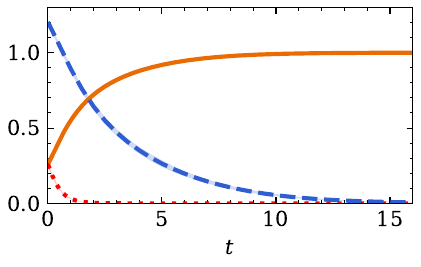}
\caption{$\beta=\infty$}
\end{subfigure}%
\caption{%
\label{fig:alignment}%
Alignment of $\rho_t^{\beta,n}$ with the top eigenspaces $E$ of $VB^\top$, and $F$ of $V$ for $d=10$ (maximum alignment corresponds to a value of $1$). The blue curve shows the Wasserstein distance quantified in Theorem \ref{thm:main}, and the vertical line marks $t=\log\beta$. 
The diagonal matrices $V$, $B$ are fixed across trials, with normally distributed entries. 
Curves show the mean over $20$ runs with $n=500$ tokens sampled from the uniform distribution $\rho_0$; shaded regions show empirical $0.10$--$0.90$ quantile.}
\end{figure}%
\section{Concluding remarks}
\label{sec:conclusions}
We provided a quantitative analysis of token concentration phenomena in mean-field transformers, showing that their zero-temperature limit~$\beta^{-1}\to0$ is accurate up to time scales of order $\log \beta$ and identifying the corresponding limiting distribution as a projection of the initial distribution $\rho_0$ onto the top eigenspace of the matrix~$VB^\top$ together with convergence rates thereto.
This metastable regime corresponds to the early-phase behavior most relevant for practical finite-depth models.
In addition, we provided experimental and theoretical evidence for a terminal phase in which the dynamics align with one of the eigenspaces $F$ or $F^{\mathrm{abs}}$ of the value matrix $V$. 

\paragraph{Extensions and limitations.}
There are a few limitations of our work which constitute important directions to address in future work. 
A natural next step is to analyze the regime of finitely many layers and tokens, i.e., time and space discretizations of~\eqref{eq:SAdynamics_meanfield}.
This includes in particular quantitative estimates for the mean-field approximation with respect to the number of tokens.
Further generalizations include investigating multi-head attention, where the projection appearing in~\eqref{eq:SAdynamics} is applied to the average of the outputs of multiple attention heads,
and enriching the dynamics by adding noise~\cite{balasubramanian2025structure,rigollet2025mean,shalova2026solutions}
or a multi-layer perceptron~\cite{alvarez2026perceptrons}.
Another limitation concerns the symmetry assumption on the matrix $VB^\top$.
While this assumption enables us to characterize the stationary distribution of the zero-temperature continuity equation, it restricts the dynamics compared to the general case where no stationary distribution might exist, cf.\@ \cite{bruno2025multiscale}.
Besides these extensions, a systematic understanding of Conjecture~\ref{conj:v_max(V)} and the quantification of the convergence in the terminal phase that we observed numerically, as well as its possible connections to the Oja flow of the value matrix $V$, see~\cite{altafini2026multistability}, are also of interest.

\begin{ack}
AA was funded by the European Union through the ModConFlex project (``Modelling and control of flexible structures interacting with fluids''; Grant Agreement No.\@ 101073558; DOI: 10.3030/101073558). AA also thanks the University of Würzburg for its hospitality during a research stay in which a substantial part of this work was carried out.
LB and TR acknowledge funding by the Federal Ministry of Research, Technology and Space of Germany (BMFTR) under grant agreement No.\@ 16IS24072 (COMFORT).
LB acknowledges funding by the Deutsche Forschungsgemeinschaft (DFG, German Research Foundation) – project number 544579844 (GeoMAR).  
KR was supported by the grant ``DMS-EPSRC: Asymptotic Analysis of Online Training Algorithms in Machine Learning: Recurrent, Graphical, and Deep Neural Networks'' (NSF DMS-2311500). TR acknowledges the support of the Munich Center for Machine Learning and the ERC
Advanced Grant NEITALG, grant agreement No.\@ 101198055.

For the purpose of Open Access, the authors have applied a CC BY public copyright license to any Author Accepted Manuscript (AAM) version arising from this submission.

\end{ack}
\newcommand{\FundingLogos}{%
\raisebox{0pt}{\includegraphics[height=1.5cm]{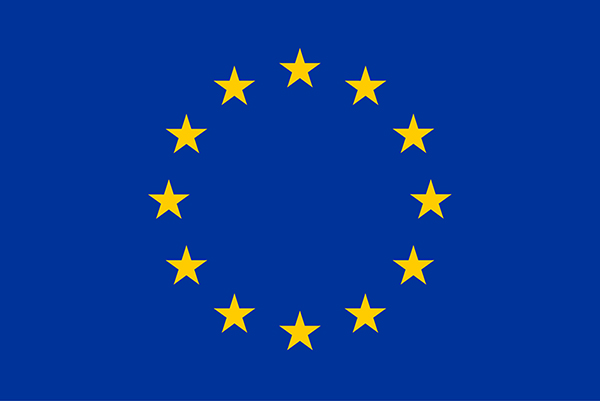}}%
\hspace{2em}%
\raisebox{0pt}{\includegraphics[height=1.5cm]{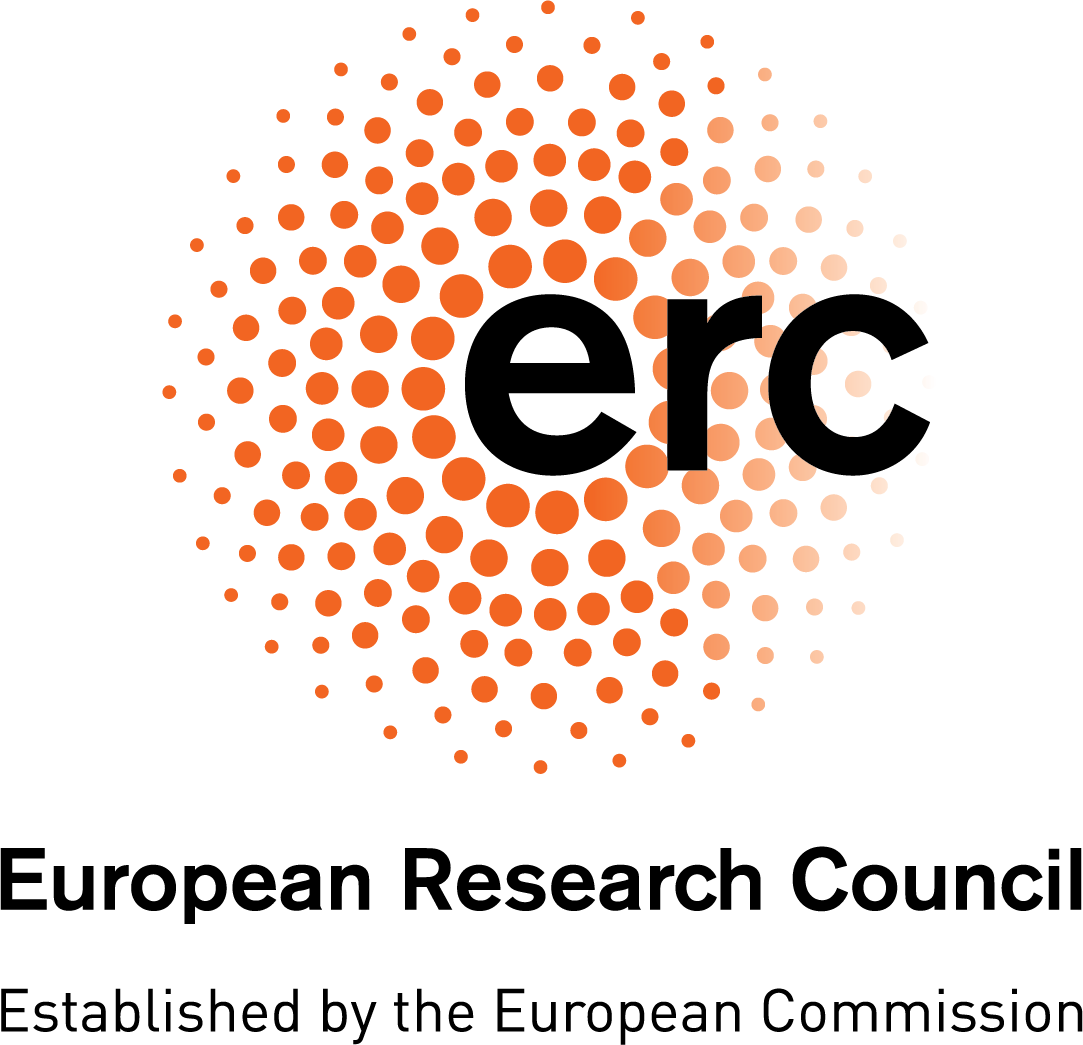}}%
}

\begin{center}
\FundingLogos
  
\vspace{0.5em}\small
Funded by the European Union. Views and opinions expressed are however those of the author(s) only and do not necessarily reflect those of the European Union or the European Research Council Executive Agency. Neither the European Union nor the granting authority can be held responsible for them. This project has received funding from the European Research Council (ERC) under the European Union’s Horizon Europe research and innovation programme (grant agreement No.\@ 101198055, project acronym NEITALG).
\end{center}

\medskip
\bibliographystyle{abbrv}
\bibliography{biblio.bib}

\newpage
\appendix

\section*{Supplementary material}

This supplementary material is organized into the following appendices.

\startcontents[main]
\printcontents[main]{}{1}{}

\section{Proof details for the main result, Theorem~\ref{thm:main}}
\label{app:proofs}

In this appendix, we present auxiliary technical results for our main theoretical result, Theorem~\ref{thm:main}, together with their proof details.

\subsection{Quantitative estimate for the Wasserstein distance between the flows \texorpdfstring{$\rho_t^\beta$}{rhot beta} and \texorpdfstring{$\rho_t$}{rhot}}
\label{app:sec:TERM_1}

The objective of this section is to derive a bound on $W_2(\rho_t^\beta,\rho_t)$ in terms of the temperature parameter $\beta$ to prove Proposition~\ref{prop:difference_continuity_eq}.

Our proof is based on a stability estimate in Wasserstein space, which we detail in Appendix~\ref{app:sec:stability}. It employs the quantitative Laplace principle from Lemma~\ref{lem:LaplacePrinciple}, see also \cite{fornasier2024consensus}, to control the discrepancy between the vector fields of the continuity equations~\eqref{eq:SAdynamics_meanfield} and \eqref{eq:zerotemp_meanfield}. The proof of the quantitative Laplace principle is provided in Appendix~\ref{app:sec:Laplace}.

\subsubsection{A quantitative Laplace principle}
\label{app:sec:Laplace}

We first give a proof of the quantitative Laplace principle on the sphere $\sphere$, Lemma~\ref{lem:LaplacePrinciple},
which was originally introduced in \cite{fornasier2024consensus} on the plane. Here we extend it to the sphere and specialize it to a certain linear function on the sphere.

\begin{proof}[Proof of Lemma~\ref{lem:LaplacePrinciple}]
    The proof follows the lines of \cite[Proposition~4.5]{fornasier2024consensus}, adapted to the sphere $\sphere$.

    To simplify notation, let us define for $x\in \sphere$ the non-negative functions
    \begin{align*}
        \mathcal E_x(y)
        \coloneqq \left\langle B^\top x, y^*(x)\right\rangle-\left\langle B^\top x,y\right\rangle
        = \N{B^\top x}\bigl(1-\left\langle y^*(x),y\right\rangle\bigr),
        \qquad
        y\in\sphere,
    \end{align*}
    where we recall from \eqref{eq:y*} that $y^*(x)=\frac{B^\top x}{\N{B^\top x}}$.
    With this definition,
    we have the identity $e^{\beta \ip{B^\top x}{y}} =e^{\beta \|B^\top x \|} e^{-\beta \mathcal{E}_x(y)}$. 
    Using this, we obtain
    \begin{align*}
        m_{\beta, \varrho}(x) 
        = \frac{\int_{\sphere} e^{ - \beta \mathcal{E}_x(y)} y \de\varrho (y)}{\int_{\sphere} e^{- \beta \mathcal{E}_x(y)} \de\varrho (y)},
    \end{align*}
    where we canceled $e^{\beta \|B^\top x \|}$ in the numerator and denominator.
    Using that we are on the sphere $\sphere$, where $\|y - y^*(x)\|^2 = \|y\|^2 - 2\langle y^*(x),y\rangle + \|y^*(x)\|^2 = 2 - 2\langle y^*(x),y\rangle$,
    we notice that $\mathcal{E}_x(y) = \frac{\|B^\top x\|}{2} \|y - y^*(x)\|^2$.
    With this we directly observe that
    \begin{align*}
        \mathcal E_{x,r}\coloneqq \sup_{y\in \mathcal{B}_r(y^*(x))}\mathcal E_x(y) \le \frac{\|B^\top x\|}{2} r^2.
    \end{align*}
    The quantity $Z_{\beta}(x)\coloneqq \int_{\sphere} e^{-\beta \mathcal E_x(y)}\de\varrho(y)$
    can be lower bounded by
    \begin{equation}\label{eq:Zbeta}
        Z_{\beta}(x) \geq \int_{\mathcal{B}_r(y^*(x))} e^{-\beta \mathcal E_x(y)}\de\varrho(y) \geq \inf_{y\in \mathcal{B}_r(y^*(x))} e^{-\beta \mathcal E_x(y)} \varrho(\mathcal{B}_r(y^*(x))) = e^{-\beta \mathcal{E}_{x,r}} \varrho(\mathcal{B}_r(y^*(x))),
    \end{equation}
    where the first inequality is due to the integrand being positive.
    Now let $\tilde r\ge r>0$.
    We can decompose \eqref{eq:QL-sphere} into contributions from inside and outside of $\mathcal{B}_{\tilde r}(y^*(x))$.
    By Jensen's inequality we have
    \begin{align}
    &\|m_{\beta,\varrho}(x)-y^*(x)\|\nonumber\\
    &\qquad\,\le
    \int_{\mathcal{B}_{\tilde r}(y^*(x))}
    \|y-y^*(x)\|\,\frac{e^{-\beta\mathcal E_x(y)}}{Z_\beta(x)}\de\varrho(y)
    +
    \int_{\mathcal{B}_{\tilde r}(y^*(x))^c}
    \|y-y^*(x)\|\,\frac{e^{-\beta\mathcal E_x(y)}}{Z_\beta(x)}\de\varrho(y).
    \label{eq:split}
    \end{align}
    Inside $\mathcal{B}_{\tilde r}(y^*(x))$ we have $\|y-y^*(x)\|\le \tilde r$,
    which allows to bound the first term in \eqref{eq:split} by $\tilde r$ since the integrand is nonnegative and
    $\int_{\mathcal{B}_{\tilde r}(y^*(x))} \frac{e^{-\beta\mathcal E_x(y)}}{Z_\beta(x)}\de\varrho(y)\le 1$.
    For the second term in \eqref{eq:split}, we use the lower bound \eqref{eq:Zbeta}, which gives
    \begin{align}
    &\int_{\mathcal{B}_{\tilde r}(y^*(x))^c}
    \|y-y^*(x)\|\,\frac{e^{-\beta\mathcal E_x(y)}}{Z_\beta(x)}\de\varrho(y)\nonumber\\
    &\qquad\,\le
    \frac{e^{\beta \mathcal E_{x,r}}}{\varrho(\mathcal{B}_r(y^*(x)))}
    \int_{\mathcal{B}_{\tilde r}(y^*(x))^c}
    \|y-y^*(x)\|\,e^{-\beta\mathcal E_x(y)}\de\varrho(y) \label{eq:tail-1}\\ 
    &\qquad\,\le \frac{2}{\varrho(\mathcal{B}_r(y^*(x)))} \exp\left(-\beta\left(\inf_{u\in \mathcal{B}_{\tilde r}(y^*(x))^c}\mathcal E_x(u)-\mathcal E_{x,r}\right)\right). \nonumber
    \end{align}
    To obtain the last inequality, we first note that $\|y-y^*(x)\|\le 2$ since $y, y^*(x)\in\sphere$.
    Secondly, inside $\mathcal{B}_{\tilde r}(y^*(x))^c$ we have $\mathcal E_x(y)\ge \inf_{u\in \mathcal{B}_{\tilde r}(y^*(x))^c}\mathcal E_x(u)$, or equivalently after taking the exponential $e^{-\beta\mathcal E_x(y)} \le e^{-\beta\inf_{u\in \mathcal{B}_{\tilde r}(y^*(x))^c}\mathcal E_x(u)}$.
    And thirdly, $\varrho\big(\mathcal{B}_{\tilde r}(y^*(x))^c\big)\leq1$.
    Combining \eqref{eq:tail-1} with the observation directly before allows us to conclude \eqref{eq:split} with the bound
    \begin{align}
    &\|m_{\beta,\varrho}(x)-y^*(x)\|
    \le
    \tilde r + \frac{2}{\varrho(\mathcal{B}_r(y^*(x)))} \exp\left(-\beta\left(\inf_{u\in \mathcal{B}_{\tilde r}(y^*(x))^c}\mathcal E_x(u)-\mathcal E_{x,r}\right)\right).
    \label{eq:split_mid}
    \end{align}
    Now, with the convention that the infimum over the empty set is $+\infty$, since $\inf_{u\in \mathcal{B}_{\tilde r}(y^*(x))^c}\mathcal E_x(u) = \inf_{u\in \mathcal{B}_{\tilde r}(y^*(x))^c} \frac{\|B^\top x\|}{2} \|u - y^*(x)\|^2 \geq \frac{\|B^\top x\|}{2}\tilde r^2$, we choose $\tilde r$ such that
    \begin{align*}
    \inf_{u\in \mathcal{B}_{\tilde r}(y^*(x))^c}\mathcal E_x(u)-\mathcal E_{x,r}  \geq  \frac{\|B^\top x\|}{2}(\tilde r^2 - r^2) = q,
    \end{align*}
    i.e., $\tilde r = \sqrt{\frac{2}{\|B^\top x\|} q + r^2}$.
    With this, \eqref{eq:tail-1} is bounded by $\frac{2e^{-\beta q}}{\varrho(\mathcal{B}_r(y^*(x)))}$
    and we obtain for \eqref{eq:split_mid} the upper bound
    \begin{align*}
    &\|m_{\beta,\varrho}(x)-y^*(x)\|
    \le
    \tilde r + \frac{2e^{-\beta q}}{\varrho(\mathcal{B}_r(y^*(x)))}.
    \end{align*}
    Noticing that $\|B^\top x\| \geq \sigma_{\min}(B)$ gives \eqref{eq:QL-sphere}, which concludes the proof.
\end{proof}

Before employing Lemma~\ref{lem:LaplacePrinciple}, we state and prove a technical result which lower-bounds the mass $\rho_t^\beta\big(\mathcal{B}_r(y^*(x))\big)$, that appears when using \eqref{eq:QL-sphere} in the proof of Proposition~\ref{prop:difference_continuity_eq}, appropriately on finite time intervals $[0,T]$.
In the proof of this result, a key ingredient, which might be of independent interest, is to control the divergence of the vector field $v_\beta[\rho_t^\beta](x) \coloneqq\proj{x}\bigl(Vm_{\beta,\rho_t^\beta}(x)\bigr)$ of the continuity equation~\eqref{eq:SAdynamics_meanfield} in terms of $\beta$, see \eqref{eq:final_divergence_bound}.
Note that, at least heuristically, at a point $x_t\in\argmin_{x\in\sphere} f_t^\beta(x)$, where $f_t^\beta$ denotes the density of $\rho_t^\beta$ with respect to the surface measure, we have $\partial_t f_t^\beta(x_t)=-f_t^\beta(x_t)\div v_\beta[\rho_t^\beta](x_t)$, where we used the product rule and $\nabla f_t^\beta(x_t)=0$.
Hence, if we had a bound on that divergence, we could use Grönwall's lemma to obtain a lower bound for the minimal density $f_t^\beta(x_t)$ and hence for the mass $\rho_t^\beta\big(\mathcal{B}_r(y^*(x))\big)$.
Making this formal reasoning rigorous is the content of the following lemma.
\begin{lemma}
\label{lem:density}
    Assume that the matrix $B$ satisfies \ref{asm:Binvertible}.
    Let $(\rho_t^\beta)_{t\in[0,T]}$ denote a weak solution to the continuity equation~\eqref{eq:SAdynamics_meanfield} with initial datum $\rho_0$,
    which is assumed to satisfy Assumption \ref{asm:lower_bound}.
    Then, for every $\beta\geq0$,
    $\rho_t^\beta$ admits a continuous, strictly positive density $f_t^\beta$ with respect to the surface measure $\mathcal{H}^{d-1}\restr\sphere$ for every $t\in[0,T]$. 

    Furthermore, let $\bar f\coloneqq\nicefrac{1}{\mathcal H^{d-1}(\sphere)}$ denote the uniform density with respect to the surface measure $\mathcal{H}^{d-1}\restr\sphere$, and
    set $\ell_\beta(t)\coloneqq\min_{x\in\sphere}f_t^\beta(x)$ as well as $L_\beta(t)\coloneqq\log\nicefrac{\bar f}{\ell_\beta(t)}$.    
    Then, for every $t\in[0,T]$, it holds $L_\beta(t)\geq0$ and for all $x\in\sphere$ it holds
    \begin{equation}
    \label{eq:lem:density:mass_sphericalcap}                   \rho_t^\beta\big(\mathcal B_{r}(y^*(x)) \big)
        \geq c_d\bar f e^{-L_\beta(t)} r^{d-1}
    \end{equation}
    for all $0<r\leq 1$.
    Here, $c_d>0$ denotes a constant depending only on $d$ with $c_d\bar f\leq 1$.
    
    Moreover, there exists a constant $c>0$ depending only on $d$, $\sigma_{\min}(B)$, $\sigma_{\max}(B)$, and $\sigma_{\max}(V)$ such that for every $t\in[0,T]$ and $\beta\geq 1$ it holds
    \begin{align}\label{eq:final_divergence_bound}
        \Nbig{\div v_\beta[\rho_t^\beta]}_\infty
        \leq 
        c\left(1+\log\beta+L_\beta(t)\right)    
    \end{align}
    with the function $L_\beta(t)$ satisfying
    \begin{equation}
        \label{eq:lem:density:Sbeta}
        1+\log\beta+L_\beta(t)
        \leq
        \left(
          1+\log\beta+\log\frac{\bar f}{\ell_0}
        \right)e^{ct}.
    \end{equation}     
\end{lemma}

\begin{proof}
    We start by showing that $\rho_t^\beta$ has a continuous density $f_t^\beta$ by applying standard arguments for continuity equations.
    
    For $t\in[0,T]$ and $x\in\sphere$, define the so-called tilted probability measure
    \begin{align}\label{eq:tilted_measure}
        \de\nu_{t,x}(y)
        \coloneqq
        \frac{
          \exp\left({\beta\ip{B^\top x}{y}}\right)
        }{
          \int_{\sphere}
          \exp\left({\beta\ip{B^\top x}{z}}\right)\!\de\rho_t^\beta(z)
        }
        \de\rho_t^\beta(y),
    \end{align}
    where we recall the notation $m_{\beta,\rho_t^\beta}$ defined in \eqref{eq:consensus_point},
    and its covariance matrix
    \[
        \Sigma_t(x)
        \coloneqq
        \int_{\sphere}
        \bigl(y-m_{\beta,\rho_t^\beta}(x)\bigr)
        \bigl(y-m_{\beta,\rho_t^\beta}(x)\bigr)^\top
        \de\nu_{t,x}(y).
    \]
    Since $\int_{\sphere} \exp\left({\beta\ip{B^\top x}{z}}\right)\!\de\rho_t^\beta(z) \geq e^{-\beta\sigma_{\max}(B)}>0$
    and $\|m_{\beta,\rho_t^\beta}(x)\|\leq1$,
    differentiation under the integral and the quotient rule give the following formula for the manifold Jacobian of the vector field $v_\beta[\rho_t^\beta]$:
    \begin{equation}
        \label{eq:Dm-density-repair}
        Dv_{\beta}[\rho_t^\beta](x)
        =
        \beta\proj{x}V\Sigma_t(x)B^\top\proj{x} - \langle x,Vm_{\beta,\rho_t^\beta}(x)\rangle\proj{x}.
    \end{equation}
    Here, we used standard calculus for tangential differentiation of tangential vector fields on the sphere~\cite{boumal2023intromanifolds}.
    Moreover,
    \begin{equation}
        \label{eq:covariance-crude-repair}
        0
        \leq
        \|\Sigma_t(x)\|_{\mathrm{op}}
        \leq
        \operatorname{tr}\Sigma_t(x)
        =
        1-\Nbig{m_{\beta,\rho_t^\beta}(x)}^2
        \leq1.
    \end{equation}
    The preceding identities \eqref{eq:Dm-density-repair} and \eqref{eq:covariance-crude-repair}, together with weak continuity of
    $t\mapsto\rho_t^\beta$ (see \cite[Remark 3.1]{geshkovski2023mathematical}) and compactness of $\sphere$, imply
    \[
        t\mapsto v_\beta[\rho_t^\beta]
        \in C\bigl([0,T];C^1(\sphere;\R^d)\bigr).
    \]
    Let $\Phi_{t}:\sphere\to\sphere$ be the associated characteristic flow, defined through
    \[
        \frac{\de}{\de t}\Phi_{t}(x)
        =
        v_\beta[\rho_t^\beta]\big(\Phi_{t}(x)\big),
        \qquad
        \Phi_{0}(x)=x,
        \qquad x\in\sphere,\;t>0.
    \]
    Since the velocity $v_\beta[\rho_t^\beta](x)\in T_x\sphere$ lies in the tangent space of the sphere at $x\in\sphere$ and the sphere is compact, standard ODE theory \cite{Lee_2012} provides the existence of a global flow of diffeomorphisms $\Phi_t$ of $\sphere$.
    We next justify that $\rho_t^\beta$ can be written as pushforward of the initial datum under this flow.
    To see this, note that for every $\varphi\in C^1([0,T]\times\sphere)$ the chain rule
    gives
    \begin{align*}
        \frac{\de}{\de t}
        \int_{\sphere}
        \varphi(t,x)\,
        \de\bigl((\Phi_t)_\sharp\rho_0\bigr)(x)
        &=
        \frac{\de}{\de t}
        \int_{\sphere}
        \varphi(t,\Phi_t(y))\de\rho_0(y) \\
        &=
        \int_{\sphere}
          \partial_t\varphi(t,\Phi_t(y))
          +
          \ip{\nabla\varphi(t,\Phi_t(y))}
          {v_\beta[\rho_t^\beta](\Phi_t(y))}
        \!\de\rho_0(y) \\
        &=
        \int_{\sphere}
          \partial_t\varphi(t,x)
          +
          \ip{\nabla\varphi(t,x)}
          {v_\beta[\rho_t^\beta](x)}
          \!\de\bigl((\Phi_t)_\sharp\rho_0\bigr)(x).
    \end{align*}
    Hence $(\Phi_t)_\sharp\rho_0$ is a weak solution of the linear
    continuity equation with prescribed velocity
    $v_\beta[\rho_t^\beta]$ and initial datum $\rho_0$. 
    Also $\rho^\beta_t$ solves the same equation. 
    Since the prescribed velocity is in particular Lipschitz in space, uniqueness for this linear continuity equation yields $\rho_t^\beta=(\Phi_t)_\sharp\rho_0$.
    Since $\rho_0$ is assumed to have a continuous density $f_0$, the same holds for $\rho_t^\beta$ for any $t\in(0,T]$ thanks to this pushforward formula with the diffeomorphism $\Phi_t$.
    We call this density $f_t^\beta$.

    Our next goal is to show the lower bound for the mass.
    Note that, because $f_t^\beta$ is a probability density, we have
    \[
        1
        =
        \int_\sphere\de\rho_t^\beta
        =
        \int_{\sphere}
        f_t^\beta\de\mathcal H^{d-1}
        \geq
        \ell_\beta(t)\mathcal H^{d-1}(\sphere)
        =
        \frac{\ell_\beta(t)}{\bar f}.
    \]
    Hence, $\ell_\beta(t)\leq\bar f$ and $L_\beta(t)=\log\nicefrac{\bar f}{\ell_\beta(t)}\geq0$.
    Here, we use the convention that $L_\beta(t)=\infty$ if $\ell_\beta(t)=0$, however, as we show later in \eqref{eq:density-min-prelim-repair}, $\ell_\beta(t)>0$ for all $t>0$.
    
    Standard arguments show the existence of a constant $c_d>0$ such that
    \[
        \mathcal H^{d-1}(\mathcal B_r(x))
        \geq
        c_dr^{d-1},
        \qquad
        0<r\leq1,\quad x\in\sphere.
    \]
    Decreasing $c_d$ if necessary, we can assume that $c_d\bar f\leq1$.
    Hence, for all $0<r\leq 1$ we get
    \begin{align}\label{eq:easy_mass_LB}
        \rho_t^\beta\left(\mathcal B_{r}(y^*(x)) \right)
       &=
        \int_{\mathcal B_{r}(y^*(x))}
        f_t^\beta \de\mathcal H^{d-1}
        \geq c_d\ell_\beta(t)r^{d-1}
        = 
        c_d\bar f e^{-L_\beta(t)} r^{d-1},
    \end{align}
    which shows \eqref{eq:lem:density:mass_sphericalcap}.

    The hard part of the proof consists in showing the estimate \eqref{eq:lem:density:Sbeta} for the function $L_\beta(t)$ appearing in~\eqref{eq:lem:density:mass_sphericalcap}.
    Upper-bounding $L_\beta(t)=\log\nicefrac{\bar f}{\ell_\beta(t)}$ amounts to lower-bounding $\ell_\beta(t):=\min_{\sphere}f_t^\beta$.
    To do so, we will use the pushforward formula $\rho_t^\beta=(\Phi_t)_\sharp\rho_0$.
    By the change-of-variables formula the densities of $\rho_0$ and $\rho_t^\beta$ satisfy $f_t^\beta(\Phi_t(x))J_t(x)=f_0(x)$ or equivalently   
    \begin{align}\label{eq:densities_COV}
        f_t^\beta(\Phi_t(x))=\frac{f_0(x)}{J_t(x)},\qquad x\in\sphere,
    \end{align}
    where $J_t(x):=\det D\Phi_t(x)\neq 0$ denotes the manifold Jacobian determinant of the diffeomorphism $\Phi_t$ at $x$.
    By \eqref{eq:densities_COV} and since $f_0$ is strictly positive on $\sphere$, lower-bounding the minimum of $f_t^\beta$ amounts to finding a good upper-bound for the Jacobian determinant $J_t(x)$.

    For this, we apply Jacobi's determinant formula for a differentiable curve of matrices $t\mapsto A(t)$ which reads $\frac{\mathrm d}{\mathrm dt}\det A(t) = 
        \det A(t)
        \operatorname{tr}\left(A(t)^{-1}\frac{\mathrm d}{\mathrm dt}A(t)\right)$,
        see \cite[Section 8.3]{magnus2019matrix}.
    With the choice $A(t):=D\Phi_t(x)$,
    we get
    \begin{align}
        \frac{\mathrm d}{\mathrm dt} J_t(x)
        &=
        J_t(x)
        \operatorname{tr}\left(
        D\Phi_t(x)^{-1}
        \frac{\mathrm d}{\mathrm dt}
        D\Phi_t(x)
        \right)
        \nonumber
        \\
        &=
        J_t(x)
        \operatorname{tr}\left(
        Dv_\beta[\rho_t^\beta](\Phi_t(x))
        \right)
        =
        J_t(x)
        \div v_\beta[\rho_t^\beta]
        (\Phi_t(x)),
        \label{eq:Jacobian_ODE}
    \end{align}
    where we applied the Jacobian to both sides of the ODE $\frac{\mathrm d}{\mathrm dt}\Phi_t(x)=v_\beta[\rho_t^\beta](\Phi_t(x))$,
    leading to $\frac{\mathrm d}{\mathrm dt}D\Phi_t(x)
    =
    Dv_\beta[\rho_t^\beta](\Phi_t(x))\,D\Phi_t(x)$ by chain rule, together with the cyclicity of the trace to cancel $D\Phi_t(x)^{-1}$ to obtain the second equality in \eqref{eq:Jacobian_ODE}.
    Solving the ODE \eqref{eq:Jacobian_ODE} explicitly, using also $J_0(x)=1$, we obtain
    \begin{align}\label{eq:Jacobian_formula}
        J_t(x)
        =
        \exp\left(\int_0^t 
        \div v_\beta[\rho_s^\beta]
        (\Phi_{s}(x))
        \de s\right),\qquad x\in\sphere.
    \end{align}
    Combining \eqref{eq:densities_COV} and \eqref{eq:Jacobian_formula} we arrive at
    \[
    f_t^\beta(\Phi_t(x))
    =
    f_0(x)\exp\left(
    -\int_0^t
    \operatorname{div}v_\beta[\rho_s^\beta]
        (\Phi_{s}(x))\,\mathrm ds
    \right),\qquad x\in\sphere.
    \]
    Taking the minimum over $x\in\sphere$ on both sides shows
    \begin{equation}
            \label{eq:density-min-prelim-repair}
            \ell_\beta(t)
            \geq
            \ell_0
            \exp\left(
              -\int_0^t
              \left\|
                \div v_\beta[\rho_s^\beta]
              \right\|_\infty\!\de s
            \right)
    \end{equation}
    and it therefore remains to obtain an estimate for $\Nnormal{\div v_\beta[\rho_t^\beta]}_\infty$ that is sufficiently well-behaved in terms of $\beta$.
    We first compute the divergence by taking the trace of the manifold Jacobian from \eqref{eq:Dm-density-repair} and obtain
    \begin{align}\nonumber 
        \div v_{\beta}[\rho_t^\beta](x)
        &=
        \beta
        \operatorname{tr}
        \left(
        \proj{x}V\Sigma_t(x)B^\top \proj{x}
        \right)
        -
        \langle x,V m_{\beta,\rho_t^\beta}(x)
        \rangle
        \operatorname{tr}\proj{x}
        \\
        &=
        \beta
        \operatorname{tr}
        \left(
        V\Sigma_t(x)B^\top \proj{x}
        \right)
        -
        (d-1)
        \langle x,V m_{\beta,\rho_t^\beta}(x)
        \rangle,
        \label{eq:divergence_formula}
    \end{align}
    where we used cyclicity of the trace, $\proj{x}^2=\proj{x}$,
    and $\operatorname{tr}\proj{x}=d-1$.
    Consequently, taking absolute values in \eqref{eq:divergence_formula} and estimating the traces on the right hand side using $\Nnormal{m_{\beta,\rho_t^\beta}(x)}\leq 1$, we obtain
    \begin{equation}
        \label{eq:divergence-covariance-repair}
        \absbig{\div v_\beta[\rho_t^\beta](x)}
        \leq
        \sigma_{\max}(V)\sigma_{\max}(B)
        \beta
        \operatorname{tr}\Sigma_t(x) + (d-1)\sigma_{\max}(V).
    \end{equation}
    A naive estimate of this quantity would be to bound $\operatorname{tr}\Sigma_t(x)$ by a constant which would result in a bound for the divergence of order $\mathcal{O}(1+\beta)$,
    which is not sufficient for our purposes and which is why we continue with a more careful analysis.
    
    Recalling that
    $y^*(x)=B^\top x/\|B^\top x\|$, the tilted measure from \eqref{eq:tilted_measure} can equivalently
    be written as
    \begin{align}\label{eq:tilted_measure_gaussian}
        \de\nu_{t,x}(y)
        =
        \frac{
          e^{-\frac{\beta\|B^\top x\|}{2}
          \|y-y^*(x)\|^2}
        }{
          \int_{\sphere}
          e^{-\frac{\beta\|B^\top x\|}{2}
          \|z-y^*(x)\|^2}\de\rho_t^\beta(z)
        }
        \de\rho_t^\beta(y).
    \end{align}
    Note that in Section~\ref{sec:relatedworks} we used exactly this rewriting to connect self-attention transformers to the polarized CBO method \cite{bungert2025polarized}.
    A simple computation gives the bias-variance decomposition
    \begin{align}\label{eq:bias-variance}
        \int_{\sphere}
        \N{y-y^*(x)}^2\de\nu_{t,x}(y)
        =
        \operatorname{tr}\Sigma_t(x)
        +
        \Nbig{m_{\beta,\rho_t^\beta}(x)-y^*(x)}^2.
    \end{align}
    Hence, the integral on the left\,---\,which we can think of as a mean-squared error\,---\,controls the term
    $\operatorname{tr}\Sigma_t(x)$ in \eqref{eq:divergence-covariance-repair} from above and we will continue by estimating this integral which, by using \eqref{eq:tilted_measure_gaussian}, takes the form
    \begin{align}\label{eq:MSE}
        \int_{\sphere}
        \N{y-y^*(x)}^2\de\nu_{t,x}(y)
        =
        \frac{\int_{\sphere}
        \N{y-y^*(x)}^2 e^{-\frac{\beta\|B^\top x\|}{2}
          \|y-y^*(x)\|^2}\de\rho_t^\beta(y)}{\int_{\sphere}
          e^{-\frac{\beta\|B^\top x\|}{2}
          \|z-y^*(x)\|^2}\de\rho_t^\beta(z)}.
    \end{align}
    Fix $0<r\leq a$.
    Lower-bounding the denominator on
    $\mathcal B_r(y^*(x))$, and splitting the numerator into integrals over the set 
    $\{y\in\sphere:\|y-y^*(x)\|\leq a\}$ and
    its complement, we get from \eqref{eq:bias-variance} and \eqref{eq:MSE} that
    \begin{equation}
        \label{eq:second-moment-repair}
        \operatorname{tr}\Sigma_t(x)
        \leq
        a^2
        +
        \frac{
          4e^{-\frac{\beta\|B^\top x\|}{2}(a^2-r^2)}
        }{
          \rho_t^\beta(\mathcal B_r(y^*(x)))
        }.
    \end{equation}
    Note that this inequality also holds when $a>2$, in which case the complementary region is empty.

    Next we choose suitable parameters as follows.
    Set $r=\beta^{-1/2}$ and define
    \[
        a^2
        \coloneqq
        r^2
        +
        \frac{2}{\beta\|B^\top x\|}
        \log\frac{4\beta}
        {\rho_t^\beta(\mathcal B_r(y^*(x)))}.
    \]
    Since $\beta\geq 1>\frac{1}{4}$, the logarithm is positive. With this choice,
    the second term in \eqref{eq:second-moment-repair} equals
    $\beta^{-1}$, and therefore
    \begin{equation}
        \label{eq:covariance-cap-repair}
        \beta\operatorname{tr}\Sigma_t(x)
        \leq
        2+\frac{2}{\sigma_{\min}(B)}
        \log\frac{4\beta}
        {\rho_t^\beta(
          \mathcal B_{\beta^{-1/2}}(y^*(x))
        )}.
    \end{equation}
    Using the lower mass bound \eqref{eq:lem:density:mass_sphericalcap}, which we proved above, in \eqref{eq:covariance-cap-repair} yields
    \begin{align*}
        \beta\operatorname{tr}\Sigma_t(x)
        &\leq
        2+\frac{2}{\sigma_{\min}(B)}
        \left(
          L_\beta(t)
          +\frac{d+1}{2}\log\beta
          +\log\frac4{c_d\bar f}
        \right) \\
        &\leq
        C\bigl(1+\log\beta+L_\beta(t)\bigr)
    \end{align*}
    for some constant $C=C(d,\sigma_{\min}(B))$. Inserting this estimate into
    \eqref{eq:divergence-covariance-repair} shows 
    \begin{align*} 
        \Nbig{\div v_\beta[\rho_t^\beta]}_\infty
        &\leq
        \sigma_{\max}(V)\sigma_{\max}(B)
        C\bigl(1+\log\beta+L_\beta(t)\bigr)
        +(d-1)\sigma_{\max}(V)
        \\
        &\leq 
        c\bigl(1+\log\beta+L_\beta(t)\bigr)
    \end{align*}
    for some constant $c>0$ depending only on $d,\sigma_{\min}(B),\sigma_{\max}(B)$, and $\sigma_{\max}(V)$.
    This proves \eqref{eq:final_divergence_bound}.
    
    Taking
    logarithms in \eqref{eq:density-min-prelim-repair} and using \eqref{eq:final_divergence_bound}, we obtain
    \begin{align*}
        \log\ell_\beta(t)
        \geq 
        \log\ell_0 - c\int_0^t(1+\log\beta+L_\beta(s))\de s.
    \end{align*}
    Using the definition of $L_\beta(t)=\log\nicefrac{\bar f}{\ell_\beta(t)}$ and $\ell_\beta(0)=\ell_0$, this can be rearranged to
    \begin{align*}
        1+\log\beta+L_\beta(t)
        &\leq
        1+\log\beta+L_\beta(0) + c\int_0^t
        \bigl(1+\log\beta+L_\beta(s)\bigr)\de s.
    \end{align*}
    Remember that
    $(t,x)\mapsto f_t^\beta(x)$ is continuous. Compactness of $\sphere$
    then implies that $t\mapsto\ell_\beta(t)$ and hence
    $t\mapsto L_\beta(t)$ is continuous.
    Grönwall's inequality therefore gives
    \[
        1+\log\beta+L_\beta(t)
        \leq
        \left(
          1+\log\beta+L_\beta(0)
        \right)e^{ct},
    \]
    which proves \eqref{eq:lem:density:Sbeta}.
\end{proof}

\subsubsection{Stability estimates in Wasserstein space}
\label{app:sec:stability}

We now have available all necessary technical tools to provide a proof of Proposition~\ref{prop:difference_continuity_eq}.
\begin{proof}[Proof of Proposition~\ref{prop:difference_continuity_eq}]
    
    Using $m_{\beta,\rho_t^\beta}$ and $y^*$ as defined in \eqref{eq:consensus_point} and \eqref{eq:y*}, respectively, denote the vector fields in \eqref{eq:SAdynamics_meanfield} and
    \eqref{eq:zerotemp_meanfield} by
    \[
        v_\beta[\rho_t^\beta](x)
        \coloneqq \proj{x}\bigl(Vm_{\beta,\rho_t^\beta}(x)\bigr),
        \qquad
        v_\infty(x)
        \coloneqq \proj{x}\bigl(Vy^*(x)\bigr).
    \]
    Also denoting by $\pi_t\in\Pi(\rho_t^\beta, \rho_t)$ the optimal coupling between $\rho_t^\beta$ and $\rho_t$ for the Wasserstein distance,
    we differentiate the squared Wasserstein distance along the solutions to the associated continuity equations~\cite[Theorem~8.4.7]{savare2008gradientflows} to obtain for almost every $t\in[0,T]$ that
    \begin{align*}
    \begin{aligned}
        \frac{\de}{\de t} W_2^2(\rho_t^\beta,\rho_t)
        &= \frac{\de}{\de t}  \iint_{\sphere\times\sphere} \|x-z\|^2 \de\pi_t(x,z) \\
        &= 2 \iint_{\sphere\times\sphere} \big\langle x-z, v_\beta[\rho_t^\beta](x)-v_\infty(z)\big\rangle \de\pi_t(x,z) \\
        &\leq 2 W_2(\rho_t^\beta,\rho_t) \sqrt{\iint_{\sphere\times\sphere} \big\|v_\beta[\rho_t^\beta](x)-v_\infty(z)\big\|^2 \de\pi_t(x,z)},
    \end{aligned}
    \end{align*}
    where we used the Cauchy--Schwarz inequality to obtain the last step. 
     
    By the chain rule it holds that $\frac{\de}{\de t} W_2^2(\rho_t^\beta,\rho_t) = 2W_2(\rho_t^\beta,\rho_t) \frac{\de}{\de t} W_2(\rho_t^\beta,\rho_t)$.
    If $W_2(\rho_t^\beta,\rho_t)=0$ for some $t>0$ we have nothing to show.
    Hence, we can assume that for $t>0$ this distance is positive and after dividing both sides by $2W_2(\rho_t^\beta,\rho_t)$ we get
    
    \begin{align}
        \label{eq:app:proof:prop:difference_continuity_eq:10}
    \begin{aligned}
        \frac{\de}{\de t} W_2(\rho_t^\beta,\rho_t)
        &\leq \sqrt{\iint_{\sphere\times\sphere} \big\|v_\beta[\rho_t^\beta](x)-v_\infty(z)\big\|^2 \de\pi_t(x,z)}.
    \end{aligned}
    \end{align}
    Let us now estimate the integrand in \eqref{eq:app:proof:prop:difference_continuity_eq:10}.
    For this, we decompose this term into an approximation error due to the finite inverse temperature $\beta$, and a stability term.
    More precisely, with the triangle inequality we have for all $x,z\in\sphere$ that
    \begin{align}
        \label{eq:app:proof:prop:difference_continuity_eq:20}
        \big\|v_\beta[\rho_t^\beta](x)-v_\infty(z)\big\|
        \leq \big\|v_\beta[\rho_t^\beta](x)-v_\infty(x)\big\| + \big\|v_\infty(x)-v_\infty(z)\big\|.
    \end{align}
    To estimate the approximation error, i.e., the first term on the right-hand side of \eqref{eq:app:proof:prop:difference_continuity_eq:20},
    we note after recalling $\|\proj{x}\|\le 1$, that
    \begin{align}
        \big\|v_\beta[\rho_t^\beta](x)-v_\infty(x)\big\| \nonumber
        &= \big\| \proj{x} \big(V m_{\beta, \rho_t^\beta}(x) \big) - \proj{x}\big(Vy^*(x)\big)\big\| \nonumber \\
        &= \left\| \proj{x}\left(V \big(m_{\beta, \rho_t^\beta}(x) - y^*(x)\big)\right)\right\| \nonumber\\
        &\leq \sigma_{\max}(V) \big\| m_{\beta, \rho_t^\beta}(x) - y^*(x)\big\|.
        \label{eq:app:proof:prop:difference_continuity_eq:20_aux1_prelim}
    \end{align}
    We can now employ Lemma~\ref{lem:LaplacePrinciple} to estimate for arbitrary $r>0$ and $q>0$ that    
    \begin{align}
        \label{eq:app:proof:prop:difference_continuity_eq:20_aux1_prelim_aux1}
    \begin{aligned}
        \big\| m_{\beta, \rho_t^\beta}(x) - y^*(x)\big\|
        & \leq
       \sqrt{r^2+\frac{2q}{\sigma_{\min}(B)}} + \frac{2e^{-\beta q}}{\rho_t^\beta \bigl(\mathcal{B}_r(y^*(x))\bigr)}.
    \end{aligned}
    \end{align}
    
    Let us now first consider the case $\beta\geq1$ and set $r:=\beta^{-1/2}\leq 1$.
    By \eqref{eq:lem:density:mass_sphericalcap} in Lemma~\ref{lem:density} we have
    \[
        \rho_t^\beta\bigl(\mathcal B_{r}(y^*(x))\bigr)
        \geq
        c_d\bar f e^{-L_\beta(t)}\beta^{-(d-1)/2}
    \]
    for a constant $c_d>0$ depending only on the dimension and such that $c_d\bar f\leq1$.
    Since $L_\beta(t)\geq0$ by Lemma~\ref{lem:density},
    we can choose
    \[
        q
        \coloneqq \frac1\beta\left( L_\beta(t)+\frac{d+1}{2}\log\beta +\log\frac{2}{c_d\bar f}
          \right)>0
    \]
    in \eqref{eq:app:proof:prop:difference_continuity_eq:20_aux1_prelim_aux1},
    which leads to the following upper bounds for the summands on the right-hand side of \eqref{eq:app:proof:prop:difference_continuity_eq:20_aux1_prelim_aux1}
    \[
        \frac{2e^{-\beta q}}{\rho_t^\beta\bigl(\mathcal B_r(y^*(x))\bigr)}
        \leq \frac1\beta
        \quad\text{ and }\quad
        \sqrt{r^2+\frac{2q}{\sigma_{\min}(B)}}
        \leq C\sqrt{\frac{1+\log\beta+L_\beta(t)}{\beta}},
    \]
    for a constant $C=C(d,\sigma_{\min}(B))$.
    Therefore, using that $\beta\geq1$ and $L_\beta(t)\geq0$
    \begin{equation}
        \label{eq:app:proof:prop:difference_continuity_eq:20_aux1_prelim_aux1_final}
        \big\| m_{\beta, \rho_t^\beta}(x) - y^*(x)\big\|
        \leq
        C\sqrt{\frac{1+\log\beta+L_\beta(t)}{\beta}}
    \end{equation}
    after possibly increasing $C$.
    Leveraging \eqref{eq:lem:density:Sbeta} from Lemma~\ref{lem:density} in the first step and $\beta\geq1$ in the second,
    we have
    \[
        1+\log\beta+L_\beta(t)
        \leq \left(1+\log\beta+\log\frac{\bar f}{\ell_0} \right)e^{ct}
        \leq 2\left(1+\log\frac{\bar f}{\ell_0}\right)\log(\beta+1)e^{ct}
    \]
    for a constant $c=c\left( d,\sigma_{\min}(B),\sigma_{\max}(B),\sigma_{\max}(V) \right)\geq 0$.
    Inserting this into \eqref{eq:app:proof:prop:difference_continuity_eq:20_aux1_prelim_aux1_final},
    we conclude from \eqref{eq:app:proof:prop:difference_continuity_eq:20_aux1_prelim} in the case $\beta\geq1$ that
    \[
        \sup_{x\in\sphere}
        \big\|
          v_\beta[\rho_t^\beta](x)-v_\infty(x)
        \big\|
        \leq
        C\sqrt{\frac{\log(\beta+1)}{\beta}}e^{ct/2}
    \]
    for a constant $C=C(d,\sigma_{\min}(B),\sigma_{\max}(V),\ell_0)$.
    For the remaining case $0<\beta<1$, we use \eqref{eq:app:proof:prop:difference_continuity_eq:20_aux1_prelim} to obtain 
    \[
        \sup_{x\in\sphere} \big\|v_\beta[\rho_t^\beta](x)-v_\infty(x)\big\|
        \leq \sigma_{\max}(V) \sup_{x\in\sphere}\big\| m_{\beta, \rho_t^\beta}(x) - y^*(x)\big\|
        \leq 2\sigma_{\max}(V).
    \]
    Since $\log(\beta+1)/\beta\geq\log2$,
    we can combine both cases
    to eventually obtain the following estimate, where $C=C(d,\sigma_{\min}(B),\sigma_{\max}(V),\ell_0)$ is a constant and we denote $C_0\coloneqq1+\frac c2$: 
    \begin{equation}
        \label{eq:app:proof:prop:difference_continuity_eq:20_aux1}
        \sup_{x\in\sphere}
        \big\|
          v_\beta[\rho_t^\beta](x)-v_\infty(x)
        \big\|
        \leq
        C\sqrt{\frac{\log(\beta+1)}{\beta}}e^{C_0t}
    \end{equation}
    for every $\beta>0$. 
    
    For the stability term, i.e., the second term in \eqref{eq:app:proof:prop:difference_continuity_eq:20},
    we first notice that for all $x,z\in \sphere$ it holds by the triangle inequality and reverse triangle inequality that
    \begin{align*}
        \|y^*(x)-y^*(z)\|
        &=
        \left\|\frac{B^\top x}{\|B^\top x\|}-\frac{B^\top z}{\|B^\top z\|}\right\|
        \leq \frac{\|B^\top x - B^\top z\|}{\|B^\top x\|} + \|B^\top z\|\left\| \frac{1}{\|B^\top x\|} - \frac{1}{\|B^\top z\|} \right\|
        \\
        &= \frac{\|B^\top x - B^\top z\|}{\|B^\top x\|} + \left\| \frac{\|B^\top z\| - \|B^\top x\|}{\|B^\top x\|} \right\| 
        \leq 2 \frac{\|B^\top x - B^\top z\|}{\|B^\top x\|} \\ 
        &\leq
        2 \frac{\sigma_{\max}(B)}{\sigma_{\min}(B)}\,\|x-z\|, 
    \end{align*}
    where we have used that $B$ is invertible and hence $\|B^\top x\|\ge \sigma_{\min}(B)$ for all \(x\in\sphere\). 
    Moreover, one has for all $x,z\in\sphere$ that
    \begin{align*}
        \|\proj{x}-\proj{z}\|
        = \|x x^\top -z z^\top\|
        \leq \|x (x - z)^\top \| + \| (x - z) z^\top\| \leq 2 \|x - z\|.
    \end{align*}
    Using these two auxiliary estimates, and recalling that $\|\proj{x}\|\le 1$, $\|y^*(x)\|=1$, we obtain by the triangle inequality
    \begin{align}
    \begin{aligned}
        \label{eq:app:proof:prop:difference_continuity_eq:20_aux2}
        \big\|v_\infty(x)-v_\infty(z)\big\|
        &= \big\|\proj{x}\left(Vy^*(x)\right) - \proj{z}\left(Vy^*(z)\right)\big\| \\
        &\leq \big\|\proj{x}\big(V(y^*(x)-y^*(z))\big)\big\| + \big\|(\proj{x}-\proj{z})(V y^*(z))\big\| \\
        &\leq \big\|\proj{x}\big\| \big\|V(y^*(x)-y^*(z))\big\| + \big\|\proj{x}-\proj{z}\big\|\big\|V y^*(z)\big\| \\
        &\leq 2 \sigma_{\max}(V)\frac{\sigma_{\max}(B)}{\sigma_{\min}(B)}\|x-z\| + 2\sigma_{\max}(V)\|x-z\| \\
        &= 2 \sigma_{\max}(V)\left(1+\frac{\sigma_{\max}(B)}{\sigma_{\min}(B)}\right)\|x-z\|.
    \end{aligned}
    \end{align}
    Inserting \eqref{eq:app:proof:prop:difference_continuity_eq:20_aux1} and \eqref{eq:app:proof:prop:difference_continuity_eq:20_aux2} into \eqref{eq:app:proof:prop:difference_continuity_eq:20},
    we obtain
    \begin{align*}
        &\big\|v_\beta[\rho_t^\beta](x)-v_\infty(z)\big\| \\
        &\quad\quad\,\leq C(d,\sigma_{\min}(B),\sigma_{\max}(B),\sigma_{\max}(V),\ell_0)\left(\sqrt{\frac{\log(\beta+1)}{\beta}}e^{C_0t} + \|x-z\|\right).
    \end{align*}
    Exploiting this estimate in
    \eqref{eq:app:proof:prop:difference_continuity_eq:10} and abbreviating $C_1\coloneqq C(d,\sigma_{\min}(B),\sigma_{\max}(B),\sigma_{\max}(V),\ell_0)>0$ as well as $\widetilde C_1 \coloneqq 2C_0 + C_1$, we get
    \begin{align}
        \nonumber
        \frac{\de}{\de t} W_2(\rho_t^\beta,\rho_t)
        &\leq 
        C_1
        \left(W_2(\rho_t^\beta,\rho_t) + \sqrt{\frac{\log(\beta+1)}{\beta}}e^{C_0t}\right)
        \\
        \label{eq:app:proof:prop:difference_continuity_eq:41}
        &\leq \widetilde{C}_1\left(W_2(\rho_t^\beta,\rho_t) + \sqrt{\frac{\log(\beta+1)}{\beta}}e^{C_0t}\right),
    \end{align}
    where the last bound is trivial by the definition of $\widetilde C_1$, and we also recall that---as $\pi_t$ denotes the optimal coupling between $\rho_t^\beta$ and $\rho_t$---we have $W_2^2(\rho_t^\beta,\rho_t)=\iint_{\sphere\times\sphere} \|x-z\|^2 \de\pi_t(x,z)$. Using the integrating factor method with integrating factor~$e^{-\widetilde{C}_1t}$, we multiply both sides of \eqref{eq:app:proof:prop:difference_continuity_eq:41} by the integrating factor to observe that, after a little bit of algebra, \eqref{eq:app:proof:prop:difference_continuity_eq:41} is equivalent to
    \begin{align*}
        \frac{\de}{\de t} \left(e^{-\widetilde{C}_1t}W_2(\rho_t^\beta,\rho_t)\right)
        \leq \widetilde{C}_1 \sqrt{\frac{\log(\beta+1)}{\beta}} e^{(C_0-\widetilde{C}_1)t}.
    \end{align*}
    Since $\widetilde{C}_1>C_0$, we can integrate this inequality in time to obtain that
    \begin{align*}
    \begin{aligned}
        e^{-\widetilde{C}_1t}W_2(\rho_t^\beta,\rho_t)
        &\leq \widetilde{C}_1 \sqrt{\frac{\log(\beta+1)}{\beta}} \int_0^t e^{(C_0-\widetilde{C}_1)s} \de s \\
        &= \frac{\widetilde{C}_1}{\widetilde{C}_1-C_0} \sqrt{\frac{\log(\beta+1)}{\beta}} \left(1-e^{(C_0-\widetilde{C}_1)t}\right),
    \end{aligned}
    \end{align*}
    where we used that $\rho_0^\beta=\rho_0$.
    Dividing now both sides by the integrating factor and using the definition of $\widetilde C_1$, we are left with
    \begin{align*}
    \begin{aligned}
        W_2(\rho_t^\beta,\rho_t)
        &\leq \frac{\widetilde{C}_1}{\widetilde{C}_1-C_0} \sqrt{\frac{\log(\beta+1)}{\beta}} \left(e^{\widetilde{C}_1t}-e^{C_0t}\right)\\
        &= \frac{2C_0+C_1}{C_0+C_1}  \sqrt{\frac{\log(\beta+1)}{\beta}} \left(e^{(2C_0 + C_1)t}-e^{C_0t}\right)\\
        &\leq 2 \sqrt{\frac{\log(\beta+1)}{\beta}} \left(e^{(2C_0 + C_1)t}-e^{C_0t}\right)
    \end{aligned}
    \end{align*}
    for all $t\in[0,T]$,
    which concludes the proof, upon renaming the constants.
\end{proof}

\subsection{Convergence of the flow \texorpdfstring{$\rho_t$}{rhot} to \texorpdfstring{$\Pi_\sharp\rho_0$}{Pi rho0}}
\label{app:sec:TERM_2}

The aim of this section is to estimate the convergence of the flow $\rho_t$ to $\Pi_\sharp\rho_0$, i.e., to estimate
$W_2(\rho_t,\Pi_\sharp\rho_0)$ by proving Proposition~\ref{prop:Lyapunov_zerotemp_p} in combination with Lemmas~\ref{lem:W2_to_Pi_p} and \ref{lem:Pi_dist}.

As a first step, we derive in Proposition~\ref{prop:Lyapunov_zerotemp_p} a time-evolution inequality for the functional $\mathcal{V}_p$ defined in \eqref{eq:VpRp}, as we detail in Appendix~\ref{app:sec:Lyapunov}. Using $\Pi_\sharp\rho_t=\Pi_\sharp\rho_0$, as demonstrated by Lemma~\ref{lem:Pi_dist}, it holds $W_2(\rho_t,\Pi_\sharp\rho_0) = W_2(\rho_t,\Pi_\sharp\rho_t)$,
we verify in Lemma~\ref{lem:W2_to_Pi_p} that the functional $\mathcal{V}_p$ is indeed a suitable Lyapunov functional for the flow $\rho_t$ since $W_2(\rho_t,\Pi_\sharp\rho_0) = W_2(\rho_t,\Pi_\sharp\rho_t) \leq \sqrt{2\mathcal{V}_p(\rho_t)}$,
where $\mathcal{V}_p(\rho_t)$ decays exponentially fast thanks to Proposition~\ref{prop:Lyapunov_zerotemp_p}.

\subsubsection{Auxiliary result for Lyapunov-type convergence estimates in Wasserstein space}
\label{app:sec:Lyapunov_aux}

We first show that $R_p$ as defined in \eqref{eq:VpRp} can be regarded as a distance from $E$ (despite not formally being a distance).
\begin{proof}[Proof of Lemma~\ref{lem:W2_to_Pi_p}]
    Fix $x\in \mathbb S^{d-1}\setminus E^\perp$.
    To simplify notation, let us use the orthogonal decomposition of $x$, 
    \begin{align*}
        x=u+v,\qquad\text{with}\quad
        u=\PE x\in E,\quad v=(\mathrm{Id}-\PE)x\in E^\perp,
        \quad\N{u}^2+\N{v}^2=1.
    \end{align*}
    We note that $u\not=0$ since $x\in \mathbb S^{d-1}\setminus E^\perp$. 
    Using that $\Pi(x)=\frac{u}{\N{u}}$,
    a simple computation shows
    \begin{align*}
        \N{x-\Pi(x)}^2
        &=
        2\left(1-\N{u}\right).
    \end{align*}
    Denoting
    \begin{align*}
        q(x)\coloneqq\frac{\N{(\mathrm{Id}-\PE)x}^2}{\N{\PE x}^2} = \frac{\N{v}^2}{\N{u}^2},
    \end{align*}
    and observing that 
    $1+q(x)
    = 1+\frac{\N{v}^2}{\N{u}^2}
    = \frac{\N{u}^2+\N{v}^2}{\N{u}^2}
    = \frac1{\N{u}^2}$
    and hence
    $\N{u}=\frac1{\sqrt{1+q(x)}}$,
    we have
    \begin{align*}
        \N{x-\Pi(x)}^2
        = 2\left(1-\frac1{\sqrt{1+q(x)}}\right).
    \end{align*}
    \clearpage
    
    Consider now the scalar function
    \begin{align*}
        f(q)\coloneqq1-\frac1{\sqrt{1+q}},
        \qquad q\ge 0.
    \end{align*}
    First,
    $f(q)\le 1$ for all $q\ge 0$.
    Second, we claim that $f(q)\le q$ for all $q\ge 0$.
    To verify this, define $g(q) \coloneqq q-\big(1-\frac1{\sqrt{1+q}}\big)$.
    It holds $g(0)=0$ and $g'(q)=1-\frac12(1+q)^{-3/2}\ge \frac12>0$ for all $q\ge 0$, 
    showing that $g(q)\ge 0$,
    or equivalently $f(q)\le q$.
    Hence $f(q)\le \min\{1,q\}$.
    Since $p\in(0,1]$, we have for all $q\ge 0$ that
    $\min\{1,q\}\le q^p$.
    Indeed, if $0\le q\le 1$, then $q\le q^p$, while if $q\ge 1$, then $1\le q^p$.
    
    Combining the above bounds shows $1-\frac1{\sqrt{1+q(x)}}\le q(x)^p$
    and therefore
    \begin{align*}
        \N{x-\Pi(x)}^2
        =
        2\left(1-\frac1{\sqrt{1+q(x)}}\right)
        \le 2 q(x)^p
        =
        2R_p(x),
        \end{align*}
    validating the first part of the claim.
    
    To prove the second part of the statement,
    define the coupling $\pi\coloneqq(\mathrm{Id},\Pi)_\sharp\varrho$
    between $\varrho$ and $\Pi_\sharp\varrho$.
    Then, by definition of the Wasserstein distance,
    \begin{align*}
        W_2^2(\varrho,\Pi_\sharp\varrho)
        &\le \iint_{\sphere\times\sphere} \N{x-y}^2\de\pi(x,y) \\
        &=
        \int_{\sphere} \N{x-\Pi(x)}^2\de\varrho(x)
        \le
        2\int_{\sphere} R_p(x)\de\varrho(x),
    \end{align*}
    which concludes the proof.
\end{proof}

\subsubsection{Lyapunov-type convergence estimates in Wasserstein space}
\label{app:sec:Lyapunov}

We now have available all the necessary technical tools to prove Proposition~\ref{prop:Lyapunov_zerotemp_p}.

\begin{proof}[Proof of Proposition~\ref{prop:Lyapunov_zerotemp_p}]
    Let us denote by $\rho_t$ the solution to the continuity equation \eqref{eq:zerotemp_meanfield} in the zero-temperature limit.
    
    We first deal with the trivial case of $k=\dim E=d$ in which case $VB^\top$ equals a multiple of the identity matrix and hence the continuity equation \eqref{eq:zerotemp_meanfield} simplifies to $\partial_t\rho_t=0$ (using that $\proj{x}(x)=0$). 
    Hence, trivially, $\rho_t=\rho_0$ for all $t>0$ and $\Pi=\operatorname{Id}$ such that $\mathcal{V}_p(\rho_t)=0$ and also $W_2(\rho_t,\Pi_\sharp\rho_t)=0$ for all $t\geq 0$.
    
    Ideally, we would like to use $R_p$ as defined in \eqref{eq:VpRp} as test function for the weak formulation of \eqref{eq:zerotemp_meanfield}.
    However, since $R_p$ is infinite on $E^\perp$ and not Lipschitz continuous for $p<\tfrac{1}{2}$, this is not possible. 
    We therefore smoothly approximate the $p$-th power as $\psi_{p,\delta}(t)\coloneqq(t+\delta)^p-\delta^p$ for $\delta>0$ and choose a smooth cutoff function $\xi_K:\mathbb R\to[0,1]$ for $K\in\mathbb N$ which satisfies $\xi_K(t)=1$ for $0\leq t\leq K$ and $\xi_K(t)=0$ for $2K\leq t < \infty$.
    With these two ingredients, we define the smooth and truncated approximation of the $p$-th power as $\Psi_{p,\delta,K}(t)\coloneqq\int_0^t\psi_{p,\delta}'(s)\xi_K(s)\de s$ with $\Psi_{p,\delta,K}(\infty) \coloneqq \Psi_{p,\delta,K}(2K)$.
    Denoting $R(x)\coloneqq R_1(x)=\nicefrac{\left\|x-P_Ex\right\|^2}{\left\|P_Ex\right\|^2}$,
    we can define the smooth test function
    \begin{align*}
        R_{p,\delta,K}(x)
        \coloneqq 
        \Psi_{p,\delta,K}\left(R(x)\right)
        = \Psi_{p,\delta,K}\left(\frac{\left\|x-P_Ex\right\|^2}{\left\|P_Ex\right\|^2}\right),
    \end{align*}
    which is smooth on $\sphere$, as it is constant in a neighborhood of $\sphere\cap E^\perp$.
    For the corresponding approximation of our Lyapunov functional~$\mathcal{V}_p(\varrho)$ from \eqref{eq:VpRp} we have
    \begin{align*}
        \mathcal{V}_{p,\delta,K}(\varrho)
        \coloneqq 
        \int_{\sphere} R_{p,\delta,K}(x)\de\varrho(x).
    \end{align*}
    Then, by the chain rule and using the weak formulation of \eqref{eq:zerotemp_meanfield}, it holds for almost all $t\in[0,T]$ that
    \begin{align}
        \label{eq:proof:lem:Lyapunov_zerotemp:5}
        \frac{\de}{\de t}\mathcal{V}_{p,\delta,K}(\rho_t) = \int_{\sphere} \frac{\left\langle\nabla R_{p,\delta,K}(x),\proj{x}\left(VB^\top x\right)\right\rangle}{\N{B^\top x}} \de\rho_t(x)
    \end{align}
    and it remains to estimate the integrand from above.
    We start by deriving an upper bound for the inner product~$\ip{\nabla R_{p,\delta,K}(x)}{\proj{x}(VB^\top x)}$. 
    To simplify notation, let us use the orthogonal decomposition of $x\in\sphere$, 
    \begin{align*}
        x=u+v,\qquad\text{with}\quad
        u=\PE x\in E,\quad v=(\mathrm{Id}-\PE)x\in E^\perp,
        \quad\N{u}^2+\N{v}^2=1,
    \end{align*}
    allowing us to write $R(x)=\nicefrac{\N{v}^{2}}{\N{u}^{2}}$ for all $x\in\sphere\setminus E^\perp$.
    Using the idempotency of $\PE$,
    we obtain for all such $x$ that
    \begin{align}
        \label{eq:grad_R_1}
        \nabla R(x) = \frac{2 v \N{u}^2 - 2 u \N{v}^2}{\N{u}^4} = 2 R(x) \left( \frac{v}{\N{v}^2} - \frac{u}{\N{u}^2}\right),
    \end{align}
    where the last equality holds for \(v\neq0\)
    and \(\nabla R(x)=0\) if \(v=0\).
    Furthermore, by chain rule, it holds
    \begin{align}\label{eq:grad_R_p_del_k}
        \nabla R_{p,\delta,K}(x) = \Psi_{p,\delta,K}'(R(x))\nabla R(x).
    \end{align}
    Let us now bound the quantity $\ip{\nabla R_{p,\delta,K}(x)}{\proj{x} (VB^\top x)}$. 
    First, note that $\ip{\nabla R(x)}{x} = 0$.
    In the case that \(v=0\), this holds trivially since \(\nabla R(x)=0\).
    Otherwise, if \(v\neq0\) it holds
    \begin{align*}
        \left\langle \frac{v}{\N{v}^2} - \frac{u}{\|u \|^2}, x \right\rangle 
        &= 
        \frac{\ip{v}{u}}{\| v\|^2} + \frac{\N{v}^2}{\N{v}^2} - \frac{\N{u}^2}{\N{u}^2}- \frac{\ip{u}{v}}{\| u\|^2} =0
    \end{align*}
    thanks to \eqref{eq:grad_R_1} and using that $\ip{u}{v} = 0$ by orthogonality.
    Recalling that $\proj{x}(VB^\top x) = VB^\top x - \ip{x}{VB^\top x} x$, we compute using \eqref{eq:grad_R_p_del_k} for all $x\in\sphere\setminus E^\perp$ and in the case \(v\neq0\) that
    \begin{align*}
        \ip{\nabla R(x)}{\proj{x} (VB^\top x)}
        &= 
        \ip{\nabla R(x)}{VB^\top x} 
        -
        \ip{x}{VB^\top x}
        \ip{\nabla R(x)}{x}
        = 
        \ip{\nabla R(x)}{VB^\top x}
        \\
        &= 
        \ip{\nabla R(x)}{VB^\top(u+v)}
        =
        \ip{\nabla R(x)}{VB^\top u} + \ip{\nabla R(x)}{VB^\top v}
        \\
        &=
        2 R(x) \left( \frac{\ip{v}{VB^\top u}}{\N{v}^2} - \frac{\ip{u}{VB^\top u}}{\N{u}^2} + \frac{\ip{v}{VB^\top v}}{\N{v}^2} - \frac{\ip{u}{VB^\top v}}{\N{u}^2}\right)
        \\
        &=
        2 R(x) 
        \left( \frac{\ip{v}{VB^\top v}}{\N{v}^2} - \frac{\ip{u}{VB^\top u}}{\N{u}^2}\right) 
        \\
        &\leq
        -2 (\mu_1-\mu_2) R(x)
        =
        -2 \gamma R(x),
    \end{align*}
    where we used \eqref{eq:grad_R_1} to obtain the fourth line
    and for the fifth line that 
    by orthogonality and since $E$ is an eigenspace of $VB^\top$, 
    $\ip{v}{VB^\top u}=\mu_1\ip{v}{u}=0$ as well as $\ip{u}{VB^\top v}=\ip{VB^\top u}{v}=\mu_1\ip{u}{v}=0$.
    To obtain the inequality in the last line, note that
    $\ip{v}{VB^\top v} \leq \mu_2 \N{v}^2$
    and
    $\ip{u}{VB^\top u} = \mu_1 \N{u}^2$,
    since $E$ is the top eigenspace of $VB^\top$.
    Further recall the identity $\gamma=\mu_1-\mu_2$.
    The same bound holds trivially in the case \(v=0\). 
       
    Combining \eqref{eq:grad_R_p_del_k} with the above inequality yields in the case of general $p\in(0,1]$ that
    \begin{align}\label{eq:proof.p.1}
    \begin{aligned}
        \ip{\nabla R_{p,\delta,K}(x)}{\proj{x}(VB^\top x)}
        &=\Psi_{p,\delta,K}'(R(x))\ip{\nabla R(x)}{\proj{x}(VB^\top x)}\\
        &\leq -2 \gamma \Psi_{p,\delta,K}'(R(x)) R(x)
    \end{aligned}
    \end{align}
    for all $x\in\sphere\setminus E^\perp$.
    If $R(x)\leq K$ a direct computation yields the Euler-type (recall $t\frac{\de}{\de t}(t^p)=pt^p$ for $t>0$) inequality
    \begin{align}\label{eq:Euler}
        \Psi_{p,\delta,K}'(R(x))R(x)\geq p\Psi_{p,\delta,K}(R(x))=p R_{p,\delta,K}(x).
    \end{align}
    Combining \eqref{eq:proof.p.1} and \eqref{eq:Euler} yields for all $x\in\sphere$ with $R(x)\leq K$ that
    \begin{align}\label{eq:proof.p.2}
        \ip{\nabla R_{p,\delta,K}(x)}{\proj{x}(VB^\top x)}
        \leq 
        -2p\gamma R_{p,\delta,K}(x).
    \end{align}
    On the other hand, if $K<R(x)\leq 2K$, the right hand side in \eqref{eq:proof.p.1} is non-positive,
    and if $R(x)>2K$, it is equal to zero. 
    Combining these observations with \eqref{eq:proof.p.2}, it holds for all $x\in\sphere$ that
    \begin{align}\label{eq:proof.p.3}
        \ip{\nabla R_{p,\delta,K}(x)}{\proj{x}(VB^\top x)}
        \leq 
        -2p\gamma R_{p,\delta,K}(x)\mathbbm{1}_{R(x)\leq K}.
    \end{align}
    Employing \eqref{eq:proof.p.3} in \eqref{eq:proof:lem:Lyapunov_zerotemp:5}
    and noticing that 
    $\N{B^\top x}\leq \sigma_{\max}(B) \N{x} = \sigma_{\max}(B)$,
    we obtain the differential inequality
    \begin{align*}
        \frac{\de}{\de t}\mathcal{V}_{p,\delta,K}(\rho_t)
        &\leq -2p\gamma\int_{\sphere} \frac{ R_{p,\delta,K}(x)}{\N{B^\top x}} \mathbbm{1}_{R(x)\leq K} \de\rho_t(x) \\
        &\leq 
        -\frac{2p\gamma}{\sigma_{\max}(B)}
        \int_{\sphere} 
        R_{p,\delta,K}(x)\mathbbm{1}_{R(x)\leq K}
        \de\rho_t(x).
    \end{align*}
    Integrating this inequality yields by the fundamental theorem of calculus
    \begin{align}
        \label{eq:proof:lem:Lyapunov_zerotemp:80}
        \mathcal{V}_{p,\delta,K}(\rho_t) - \mathcal{V}_{p,\delta,K}(\rho_0) \leq -\frac{2p\gamma}{\sigma_{\max}(B)} \int_0^t \int_{\sphere} R_{p,\delta,K}(x) \mathbbm{1}_{R(x)\leq K} \de\rho_s(x) \de s.
    \end{align}
    Passing first to the limit $K\to\infty$ and then to the limit
    $\delta\to0$, and applying Beppo Levi's monotone convergence
    theorem at each step,
    we obtain
    \begin{align*}
        \mathcal{V}_p(\rho_t) - \mathcal{V}_p(\rho_0) \leq -\frac{2p\gamma}{\sigma_{\max}(B)} \int_0^t \mathcal{V}_p(\rho_s) \de s.
    \end{align*}
    An application of Grönwall's inequality yields
    \begin{align*}
        \mathcal{V}_p(\rho_t)
        \leq 
        \mathcal{V}_p(\rho_0)
        \exp\left(-\frac{2p\gamma}{\sigma_{\max}(B)} t\right),
    \end{align*}
    which concludes the proof.
\end{proof}

\subsubsection{Invariance of \texorpdfstring{$\Pi$}{Pi} under the zero-temperature equation}
\label{app:lem:Pi_dist}

Before giving the proof of Lemma~\ref{lem:Pi_dist},
let us first verify the following technical result. 

\begin{lemma}
    \label{lem:jacobian}
    Assume that the matrices~$B$ and $V$
    satisfy Assumption~\ref{asm:weights}.
    Then, for $x\in\sphere\setminus E^\perp$, the Jacobian $D\Pi(x)$ of $\Pi$ as defined in \eqref{eq:Pi} exists and equals
    \begin{align*}
        D\Pi(x) =
        \frac{1}{\N{\PE x}}
        \left(\PE - 
        \frac{\PE x \left(\PE x\right)^\top}{\N{\PE x}^2}
        \right).
    \end{align*}
    Moreover, it holds
    \begin{align*}
        D\Pi(x)\proj{x}\left(\frac{VB^\top x}{\N{B^\top x}}\right) = 0.
    \end{align*}
\end{lemma}
\begin{proof}
    To obtain the first claim,
    we compute for $x\in\sphere\setminus E^\perp$ with quotient rule that
    \begin{align*}
        D\Pi(x)
        =
        \frac{\PE}{\N{\PE x}}
        -
        \frac{\PE x \left(\PE x\right)^\top}{\N{\PE x}^3}
        =
        \frac{1}{\N{\PE x}}
        \left(\PE - 
        \frac{\PE x \left(\PE x\right)^\top}{\N{\PE x}^2}
        \right).
    \end{align*}
    For the second claim,
    using that $\proj{x}(y) = y - \langle x, y\rangle x$ and that $\ip{\PE x}{x}=\N{\PE x}^2$,
    we derive for arbitrary $y\in\sphere$ that
    \begin{align}
        \label{eq:proof:lem:jacobian:aux}
    \begin{aligned}
        \left(
        \PE - \frac{\PE x (\PE x)^\top}{\N{\PE x}^2}
        \right) \proj{x}(y)
        &=
        \PE y - \frac{\langle \PE x, y \rangle}{\N{\PE x}^2} \PE x
        -\ip{x}{y}
        \left(
        \PE x - \frac{\PE x\ip{\PE x}{x}}{\N{\PE x}^2}
        \right)
        \\
        &=
        \PE y - \frac{\langle \PE x, y \rangle}{\N{\PE x}^2} \PE x.
    \end{aligned}
    \end{align}
    To simplify notation, let us use the orthogonal decomposition of $x$, 
    \begin{align*}
        x=u+v,\qquad\text{with}\quad
        u=\PE x\in E,\quad v=(\mathrm{Id}-\PE)x\in E^\perp,
        \quad\N{u}^2+\N{v}^2=1.
    \end{align*}
    Specifically, for $\hat{y} = \frac{V B^\top x}{\N{B^\top x}}$,
    we have $\hat{y} = \frac{\mu_1 u}{\N{B^\top x}}  + \frac{V B^\top v}{\N{B^\top x}}$.
    Since $VB^\top v \in E^\perp$ (here, we crucially use the symmetry of $VB^\top$),
    it holds
    \begin{align*}
        \PE \hat{y} = \frac{\mu_1 u}{\N{B^\top x}} = \frac{\mu_1 \PE x}{\N{B^\top x}}
        \quad\text{ and hence }\quad
        \langle \PE x, \hat{y} \rangle = \mu_1 \frac{\N{\PE x}^2}{\N{B^\top x}}.
    \end{align*}
    Using those identities in the last step,
    we have with \eqref{eq:proof:lem:jacobian:aux}
    that
    \begin{align*}
        D\Pi(x)\proj{x}(\hat{y})
        = \frac{1}{\N{\PE x}} \left( \PE - \frac{\PE x (\PE x)^\top}{\N{\PE x}^2} \right) \proj{x}(\hat{y})
        = \frac{1}{\N{\PE x}} \left(\PE \hat{y} - \frac{\langle \PE x, \hat{y} \rangle}{\N{\PE x}^2} \PE x\right)
        = 0,
    \end{align*}
    which concludes the proof.
\end{proof}
With this auxiliary technical tool, we can now provide a proof of Lemma~\ref{lem:Pi_dist}.
\begin{proof}[Proof of Lemma~\ref{lem:Pi_dist}]
    To keep the notation concise,
    let us write
    \begin{align*}
        v_\infty(x)
        \coloneqq
        \proj{x}\left(Vy^*(x)\right)
        \quad\text{ with }\quad
        y^*(x)
        \coloneqq
        \frac{B^\top x}{\N{B^\top x}}
    \end{align*}
    for the vector field of the continuity equation \eqref{eq:zerotemp_meanfield}.
    Let $\varphi\in \mathcal{C}^\infty(\sphere,\mathbb{R})$ be a smooth test function and choose a smooth cut-off function $\eta\in \mathcal{C}^\infty([0,\infty),[0,1])$ satisfying
    \begin{align*}
        \eta(r)=0 \quad\text{for } r\le 1,
        \qquad
        \eta(r)=1 \quad\text{for } r\ge 2,
        \qquad
        |\eta'|\leq 2.
    \end{align*}
    For $\delta>0$, define $\eta_\delta(r)\coloneqq\eta(r/\delta)$ and 
    \begin{align*}
        \varphi_\delta(x)
        \coloneqq\eta_\delta\left(\|\PE x\|\right)\varphi(\Pi(x)).
    \end{align*}
    Clearly, $\varphi_\delta\in \mathcal{C}^\infty(\sphere,\mathbb{R})$ and hence an admissible test function.
    Using the weak formulation of \eqref{eq:zerotemp_meanfield} for $\varphi_\delta$,
    it holds
    \begin{align*}
        \frac{\de}{\de t}\int_{\sphere}\varphi_\delta(x)\de\rho_t(x)
        = \int_{\sphere}\left\langle\nabla\varphi_\delta(x), v_\infty(x)\right\rangle\!\de\rho_t(x),
    \end{align*}
    where we can compute by product and chain rule
    \begin{align*}
        \nabla\varphi_\delta(x)
        =
        \eta_\delta(\|\PE x\|)\nabla(\varphi\circ\Pi)(x)
        +
        \eta_\delta'(\|\PE x\|)\varphi(\Pi(x))\,\nabla\|\PE x\|,
        \qquad x\in\sphere,
    \end{align*}
    noting that $\nabla\N{\PE x}$ exists on the support of $\eta_\delta'$.
    Since Lemma~\ref{lem:jacobian} gives $D\Pi(x)v_\infty(x)=0$ for all $x\in \sphere\setminus E^\perp$, we have for every such $x$ that
    \begin{align*}
        \left\langle
        \nabla(\varphi\circ\Pi)(x), v_\infty(x)
        \right\rangle
        =
        \left\langle
        D\Pi(x)^\top\nabla\varphi(\Pi(x)), v_\infty(x)
        \right\rangle
        =
        \left\langle
        \nabla\varphi(\Pi(x)),D\Pi(x) v_\infty(x)
        \right\rangle
        =0.
    \end{align*}
    Furthermore, for $x\in\sphere\setminus E^\perp$ we have $\nabla\N{\PE x} = \frac{\PE x}{\N{\PE x}}$.
    Therefore, we obtain
    \begin{align}
        \label{eq:proof:lem:Pi_dist:10}
        \frac{\de}{\de t}\int_{\sphere}\varphi_\delta(x)\de\rho_t(x)
        =
        \int_{\sphere}
        \eta_\delta'\left(\|\PE x\|\right)
        \varphi(\Pi(x))
        \left\langle
        \frac{\PE x}{\N{\PE x}}, v_\infty(x)
        \right\rangle
        \!\de\rho_t(x).
    \end{align}
    To simplify notation, let us use the orthogonal decomposition of $x$, 
    \begin{align*}
        x=u+v,\qquad\text{with}\quad
        u=\PE x\in E,\quad v=(\mathrm{Id}-\PE)x\in E^\perp,
        \quad\N{u}^2+\N{v}^2=1.
    \end{align*}
    Using this notation and that $v_\infty(x)=\proj{x}\left(\frac{VB^\top x}{\N{B^\top x}}\right)=\frac{VB^\top x-\langle x,VB^\top x\rangle x}{\N{B^\top x}}$,
    we compute
    \begin{align*}
        \left\langle
        \frac{\PE x}{\N{\PE x}},v_\infty(x)
        \right\rangle
        =
        \frac{\langle u,v_\infty(x)\rangle}{\N{u}}
        =
        \frac{
        \mu_1\|u\|^2-\langle x,VB^\top x\rangle\|u\|^2
        }{\N{B^\top x}\N{u}}
        =
        \frac{\|u\|(\mu_1-\langle x,VB^\top x\rangle)}{\N{B^\top x}},
    \end{align*}
    where we used in the second step that by the invariance of $E$ and $E^\perp$, it holds $\PE(VB^\top x)=\mu_1 u$.
    This allows us to bound for a constant $C>0$ depending only on $B$ and $VB^\top$ that
    \begin{align*}
        \left|\left\langle
        \frac{\PE x}{\N{\PE x}},v_\infty(x)
        \right\rangle\right|
        \le C\|u\|
        =
        C\|\PE x\|.
    \end{align*}
    Since $\eta'_\delta$ is supported in $A_\delta\coloneqq\{x\in \sphere:\delta\le \|\PE x\|\le 2\delta\}$,
    we can bound \eqref{eq:proof:lem:Pi_dist:10} as
    \begin{align}
        \label{eq:proof:lem:Pi_dist:20}
        \left|\frac{\de}{\de t}\int_{\sphere}\varphi_\delta(x)\de\rho_t(x)\right|
        \leq 
        C\int_\sphere
        \underbrace{|\eta_\delta'(\N{\PE x})|}_{\leq\frac{2}{\delta}\mathbbm{1}_{A_\delta}}\,
        |\varphi(\Pi(x))|\,
        \underbrace{\N{\PE x}}_{\leq2\delta}\de\rho_t(x)
        \leq C\rho_t(A_\delta),    
    \end{align}
    where $C>0$ is another constant depending solely on $B$, $VB^\top$, and the maximum norm $\|\varphi\|_\infty$ of $\varphi\in \mathcal{C}^\infty(\sphere,\mathbb{R})$ on the compact sphere $\sphere$.
    It remains to estimate $\rho_t(A_\delta)$.
    Since on $A_\delta$
    it holds $\|\PE x\|^{2}\leq4\delta^2$
    and for $\delta\le 1/4$ moreover
    \begin{align*}
        \N{(\mathrm{Id}-\PE)x}^2
        =\N{v}^2
        =1-\N{u}^2
        =1-\|\PE x\|^2
        \ge 1-4\delta^2\ge 3/4,
    \end{align*}
    we have for all $x\in A_\delta$ that
    \begin{align*}
        R_p(x)=\frac{\|(\mathrm{Id}-\PE)x\|^{2p}}{\|\PE x\|^{2p}}
        \ge c_p\delta^{-2p}
    \end{align*}
    for a suitable constant $c_p>0$.
    We can therefore estimate
    \begin{align*}
        \rho_t(A_\delta)
        = \int_{A_\delta}1\de\rho_t(x)
        \leq \frac{\delta^{2p}}{c_p}\int_{A_\delta}R_p(x)\de\rho_t(x)
        \le \frac{\delta^{2p}}{c_p}\int_{\sphere}R_p(x)\de\rho_t(x)
        =
        \frac{\delta^{2p}}{c_p}\mathcal V_p(\rho_t).
    \end{align*}
    Thanks to Proposition~\ref{prop:Lyapunov_zerotemp_p},
    $\sup_{t\in[0,T]}\mathcal V_p(\rho_t)\le \mathcal V_p(\rho_0)<\infty$, 
    and thus, as $\delta\to 0$,
    we get from \eqref{eq:proof:lem:Pi_dist:20} that
    \begin{align}
        \label{eq:proof:lem:Pi_dist:40}
        \sup_{t\in[0,T]}
        \left|\frac{\de}{\de t}\int_{\sphere}\varphi_\delta(x)\de\rho_t(x)\right|
        \le C\delta^{2p}\mathcal V_p(\rho_0)\to 0,
    \end{align}
    where $C>0$ now depends also on $p$.
    By the fundamental theorem of calculus,
    \begin{align}
        \label{eq:proof:lem:Pi_dist:50}
        \int_{\sphere}\varphi_\delta(x)\de\rho_t(x)
        -\int_{\sphere}\varphi_\delta(x)\de\rho_0(x)
        = \int_0^t \left( \frac{\de}{\de s} \int_{\sphere}\varphi_\delta(x)\de\rho_s(x) \right) \!\de s.
    \end{align}
    Recall that by definition of the cut-off function $\eta$ it holds $|\varphi_\delta| \leq \|\varphi\|_\infty$
    for any $\delta>0$ and $\varphi_\delta\rightarrow\varphi\circ\Pi$ pointwise on $\sphere\setminus E^\perp$ as $\delta\rightarrow0$.
    Since $\mathcal V_p(\rho_t)<\infty$ by Proposition~\ref{prop:Lyapunov_zerotemp_p},
    it holds $\rho_t(E^\perp)=0$.
    Hence, $\int_{\sphere} (\varphi\circ\Pi)(x)\de\rho_t(x) = \int_{\sphere\backslash E^\perp} (\varphi\circ\Pi)(x)\de\rho_t(x) \leq \max_{z\in\sphere}|\varphi(z)|<\infty$ since $\varphi$ is a smooth function on a compact space.
    We can now pass to the limit $\delta\to0$ using the dominated convergence theorem to infer the convergence of the integrals on the left-hand side of \eqref{eq:proof:lem:Pi_dist:50} to $\int_{\sphere}\varphi(\Pi(x))\de\rho_t(x)-\int_{\sphere}\varphi(\Pi(x))\de\rho_0(x)$.
    With the right-hand side of \eqref{eq:proof:lem:Pi_dist:50} vanishing due to \eqref{eq:proof:lem:Pi_dist:40} as $\delta\rightarrow0$, we obtain that
    \begin{align*}
        \int_{\sphere} \varphi(x) \de(\Pi_\sharp\rho_t)(x)
        &=
        \int_{\sphere}\varphi(\Pi(x))\de\rho_t(x) \\
        &=
        \int_{\sphere}\varphi(\Pi(x))\de\rho_0(x)
        =
        \int_{\sphere} \varphi(x) \de(\Pi_\sharp\rho_0) (x).
    \end{align*}
    Since this holds for all smooth $\varphi$, we conclude that $\Pi_\sharp\rho_t=\Pi_\sharp\rho_0$, as desired.
\end{proof}

\subsection{Proof of Corollary~\ref{cor:time_scales}}
\label{sec:proof:cor:time_scales}

\begin{proof}[Proof of Corollary~\ref{cor:time_scales}]
    It is easy to verify that for $t\geq t_1$ it holds $\sqrt{2\mathcal{V}_p(\rho_0)}\exp\left(-\nicefrac{p\gamma}{\sigma_{\max}(B)}t\right)\leq\frac{\eps}{2}$, and that for $t\leq t_2$ it holds $2\sqrt{\nicefrac{\log(\beta+1)}{\beta}}(e^{C_1t}-e^{C_0t})\leq2\sqrt{\nicefrac{\log(\beta+1)}{\beta}}(e^{C_1t}-1)\leq\frac\eps2$.
    By choosing $\beta$ large enough (in fact, $\beta\gg\operatorname{poly}\left(\tfrac{1}{\eps}\right)\log\tfrac{1}{\eps}$ suffices) we have $t_1<t_2$ and the conclusion follows from Theorem~\ref{thm:main}.
\end{proof}

\section{Heuristics supporting Conjecture~\ref{conj:v_max(V)}}
\label{app:conjecture}

In this section, we provide heuristics supporting Conjecture~\ref{conj:v_max(V)}. Some of them can actually be found in the recent work \cite{altafini2026multistability}, but we include them here for completeness.

\paragraph{One-point dynamics.}

First, we consider the trivial case where the token distribution is completely collapsed. This captures the situation where the initial distribution $\rho_0$ is just supported on a single point and approximates the more interesting situation where at some positive time $t>0$ the token distribution is (almost) collapsed.

We can show that from there on the distribution aligns with the top eigenspace of $V$. To see this note that for $\varrho=\delta_{x_0}$ with $x_0\in\sphere$, we have
\begin{align*}
    \int_\sphere
    \frac{\exp(\beta\ip{x}{By})}{\int_\sphere\exp(\beta\ip{x}{Bz})\de\varrho(z)}Vy\de\varrho(y)
    =
    Vx_0.
\end{align*}
It is then easy to observe that the solution of \eqref{eq:SAdynamics_meanfield} is given by $\rho_t^\beta=\delta_{x(t)}$ where $t\mapsto x(t)$ solves the Oja flow
\begin{equation*}
    \dot x(t) = \proj{x(t)}Vx(t),\qquad x(0)=x_0.
\end{equation*}
It follows from \cite[Lemmas A.13, A.15]{bruno2025multiscale} (or a straightforward computation using the quotient rule) that $x(t)=\frac{y(t)}{\N{y(t)}}$ where $t\mapsto y(t)$ solves the system of linear ordinary differential equations
\begin{align*}
    \dot y(t) = Vy(t).
\end{align*}
Assuming for simplicity that $V$ is symmetric and that its largest eigenvalue is simple, it follows from the theory of linear ordinary differential equations that $y(t)=\exp(Vt)y(0)$ and hence that $x(t)$ converges exponentially fast to an eigenvector $v$ lying in the eigenspace of $V$ that corresponds to its largest eigenvalue as $t\to\infty$ provided that $\ip{x_0}{v}\neq 0$. As a consequence, $W_2(\rho_t^\beta,\delta_v)=\N{x(t)-v}\to 0$ as $t\to\infty$. This reasoning can be extended to more general classes of matrices $V$ by using their Jordan normal form. 

\paragraph{Some stationary distributions.}

The above shows, in particular, that $\delta_v$ is a stationary distribution of \eqref{eq:SAdynamics_meanfield}.
Let us now investigate slightly more complicated stationary distributions of this equation.
For this, let us assume that $\varrho = \frac{1}{n}\sum_{i=1}^n\delta_{v_i}$ where $v_i$ are eigenvectors of $V$ with eigenvalue $\lambda>0$. 
Abbreviating $N(x)\coloneqq\sum_{i=1}^n\exp\left(\beta\ip{x}{Bv_i}\right)$, we can compute the drift vector field in \eqref{eq:SAdynamics_meanfield} as
\begin{align*}
    \int_\sphere
    \frac{\exp(\beta\ip{x}{By})}{\int_\sphere\exp(\beta\ip{x}{Bz})\de\varrho(z)}Vy\de\varrho(y)
    &=
    \frac{1}{N(x)}
    \sum_{i=1}^n
    \exp\left(\beta\ip{x}{Bv_i}\right)V v_i
    \\
    &=
    \frac{\lambda}{N(x)}
    \sum_{i=1}^n
    \exp\left(\beta\ip{x}{Bv_i}\right) v_i,
    \qquad
    \forall x\in\sphere.
\end{align*}
Hence, for any smooth $\phi\in\mathcal C^\infty(\sphere,\R)$, using that $\proj{v_i}(v_i)=0$, we have
\begin{align}
    &
    \int_\sphere
    \left\langle
    \nabla\phi(x),
    \proj{x}
    \left(
    \int_\sphere
    \frac{\exp(\beta\ip{x}{By})}{\int_\sphere\exp(\beta\ip{x}{Bz})\de\varrho(z)}Vy\de\varrho(y)
    \right)
    \right\rangle
    \de\varrho(x)\nonumber
    \\
    \nonumber
    &\qquad\,=
    \frac{1}{n}
    \sum_{i=1}^n
    \left\langle
    \nabla\phi(v_i),
    \proj{v_i}
    \left(
    \frac{\lambda}{N(v_i)}
    \sum_{j=1}^n
    \exp\left(\beta\ip{v_i}{Bv_j}\right) v_j
    \right)
    \right\rangle
    \\
    \label{eq:drift_term}
    &\qquad\,=
    \frac{\lambda}{n}
    \sum_{i=1}^n
    \frac{1}{N(v_i)}
    \left\langle
    \nabla\phi(v_i),
    \proj{v_i}
    \left(
    \sum_{j\neq i}
    \exp\left(\beta\ip{v_i}{Bv_j}\right) v_j
    \right)
    \right\rangle.
\end{align}
If $n=1$, meaning that $\varrho=\delta_v$ for some $v\in\sphere$ with $Vv=\lambda v$, then \eqref{eq:drift_term} equals zero. 
This shows that $\varrho=\delta_v$ is a stationary distribution of \eqref{eq:SAdynamics_meanfield}.

A slightly more complex example is $n=2$ and $\varrho=\frac12\left(\delta_v+\delta_{-v}\right)$, for which we have $\proj{v}(-v)=-v-\ip{v}{-v}v=0$ and $\proj{-v}(v)=v-\ip{-v}{v}(-v)=0$.
Hence, \eqref{eq:drift_term} equals zero and $\varrho$ is a stationary distribution of \eqref{eq:SAdynamics_meanfield}.
Such situation is depicted in Figure~\ref{fig:conj}.

These two classes of stationary distributions concentrated on eigenvectors of $V$ were already identified in \cite[Lemma 7]{altafini2026multistability} and were referred to as consensus and bipartite consensus therein.
In \cite[Theorem 2]{altafini2026multistability}, it was also shown that consensus on an eigenvector with respect to the largest eigenvalue of $V$ is a stable equilibrium and that bipartite consensus can be stable under certain conditions.

\section{Additional numerical experiments}
\label{app:numericalexperiments}

This appendix contains implementation details and supplementary numerical experiments. We begin by detailing the time discretizations used in our simulations, together with the initialization procedure for the particle simulations. We then present several additional experiments: examples showing effects that can occur when $VB^\top$ is not symmetric, simulations in the gradient flow case $B=\pm V$, and numerical evidence supporting Conjecture~\ref{conj:v_max(V)}.

\subsection{Details on the discretization}

We provide pseudocode for the employed Euler discretization of \eqref{eq:SAdynamics}, and the zero-temperature limit \eqref{eq:zerotemp_meanfield}. Let us highlight here that Algorithm \ref{alg:sadyn} is not a faithful discretization of the zero-temperature equation. Namely, formally setting $\beta=\infty$ in \eqref{eq:SAdynamics} amounts to replacing the softmax in lines 2--3 of Algorithm~\ref{alg:sadyn} by a hardmax function. In the finite-particle case, however, this does not recover Algorithm~\ref{alg:zerodyn}. We note that whenever we run a numerical experiment in what follows with $\beta=\infty$, this means we employ Algorithm~\ref{alg:zerodyn}.

\begin{algorithm}[h]
\caption{Discretization of \eqref{eq:SAdynamics}}\label{alg:sadyn}
\begin{algorithmic}[1]
\Require Tokens $\{x_i\}_{i=1}^n \subset \mathbb{S}^{d-1}$,
 matrices $B, V \in \mathbb{R}^{d\times d}$,
 $\beta,\,\Delta t> 0.$
\For{$i = 1, \ldots, n$}
\State $w_{ij} \gets
\dfrac{\exp\left(\beta\ip{x_i}{B x_j}\right)}
    {\sum_{k=1}^{n}\exp\left(\beta\ip{x_i}{B x_k}\right)},
\quad j=1,\ldots,n$
\State $m_i \gets \sum_{j=1}^{n} w_{ij}\,x_j$
\Comment{softmax-weighted consensus}
\State $x_i \gets x_i
     + \Delta t\;\proj{x_i} V m_i$
\State $x_i \gets x_i \,/\, \|x_i\|$
\Comment{retract to $\mathbb{S}^{d-1}$}
\EndFor
\end{algorithmic}
\end{algorithm}

\begin{algorithm}[h]
\caption{Discretization of the zero-temperature limit \eqref{eq:zerotemp_meanfield}}\label{alg:zerodyn}
\begin{algorithmic}[1]
\Require Tokens $\{x_i\}_{i=1}^n \subset \mathbb{S}^{d-1}$,
         matrices $B, V \in \mathbb{R}^{d\times d}$,
         $\Delta t> 0$
\For{$i = 1, \ldots,n$}
    \State $m_i \gets B^\top x_i \,/\, \|B^\top x_i\|$
    \State $x_i \gets x_i
             + \Delta t\;\proj{x_i} V m_i$
    \State $x_i \gets x_i \,/\, \|x_i\|$
    \Comment{retract to $\mathbb{S}^{d-1}$}
\EndFor
\end{algorithmic}
\end{algorithm}

\paragraph{Initialization.} For the visualization in Figure~\ref{fig:thm1}, we employ a mixture of von Mises--Fisher distributions, with densities, that up to normalization are given as
\begin{align*}
p(x| \mu_i,\kappa_i) \sim 
\exp(\kappa_i \mu_i^\top x)
\end{align*}
with parameters $\kappa_i\in[0,\infty)$ and $\mu_i \in \R^d$ for $i=1,\ldots,m$. For weights $w_i$ with $\sum_{i=1}^m w_i=1$, the mixture then fulfills
\begin{align*}
p(x|\{(\mu_i,\kappa_i)\}_{i=1}^m) \sim \sum_{i=1}^m w_i \, p(x| \mu_i,\kappa_i).
\end{align*}
For our examples, we choose
\begin{gather*}
\mu_1 = (1,-0.3,-0.2),\, \mu_2 =(0,1,-0.3),\, \mu_3=(-1,1,1),\\ 
\kappa_1=2,\,\kappa_2 = 10,\,\kappa_3 = 5,\quad
w_1=w_2=w_3=1/3.
\end{gather*}

\subsection{Non-symmetric \texorpdfstring{$VB^\top$}{VBt} effects}
\label{sec:nonsymmetric}

We illustrate the role of the symmetry assumption on $VB^\top$ through two examples. To do so, we consider the following matrices
\begin{align}\label{eq:nonsym}
\begin{aligned}
V_1 &=
\begin{pmatrix}
-1   &1  &0\\
-2   &1  &0\\
0    &0  &-2
\end{pmatrix},\qquad
B_1 = \operatorname{diag}(-1, -1, 1),\\
V_2 &= \operatorname{diag}(-1, 1, -2),\qquad 
B_2 = 
\begin{pmatrix}
-1   &1  &0\\
-2   &1  &0\\
0    &0  &1
\end{pmatrix}.
\end{aligned}
\end{align}
In this case, note that the eigenvalues of $V_iB_i^\top$ and $V_i$ are not necessarily real. In fact, the eigenvalues of $V_1B_1^\top$ and $V_1$ equal $\pm i$ and $-2$, whereas the eigenvalues of $V_2B_2^\top$, despite being a non-symmetric matrix, are real and equal $1\pm\sqrt{2}$ and $-2$. Those of $V_2$ are equal to $\pm 1$ and $-2$. The trajectories for these two settings are displayed in Figure~\ref{fig:finiteBeta1}. In the top row corresponding to $V_1$ and $B_1$, the complex eigenvalues lead to a visible rotation of the trajectories, and no clear alignment is observed. In the bottom row, however, $V_2B_2^\top$ is still non-symmetric but has real eigenvalues, and the same two-step behavior as in the symmetric case is recovered.
\newcounter{subgroup}\setcounter{subgroup}{1}
\renewcommand\thesubfigure{\alph{subfigure}\arabic{subgroup}}
\begin{figure}[h]
\centering
\begin{subfigure}{.28\textwidth}
\includegraphics[width=\linewidth,trim={2cm 2cm 2cm 2cm}, clip]{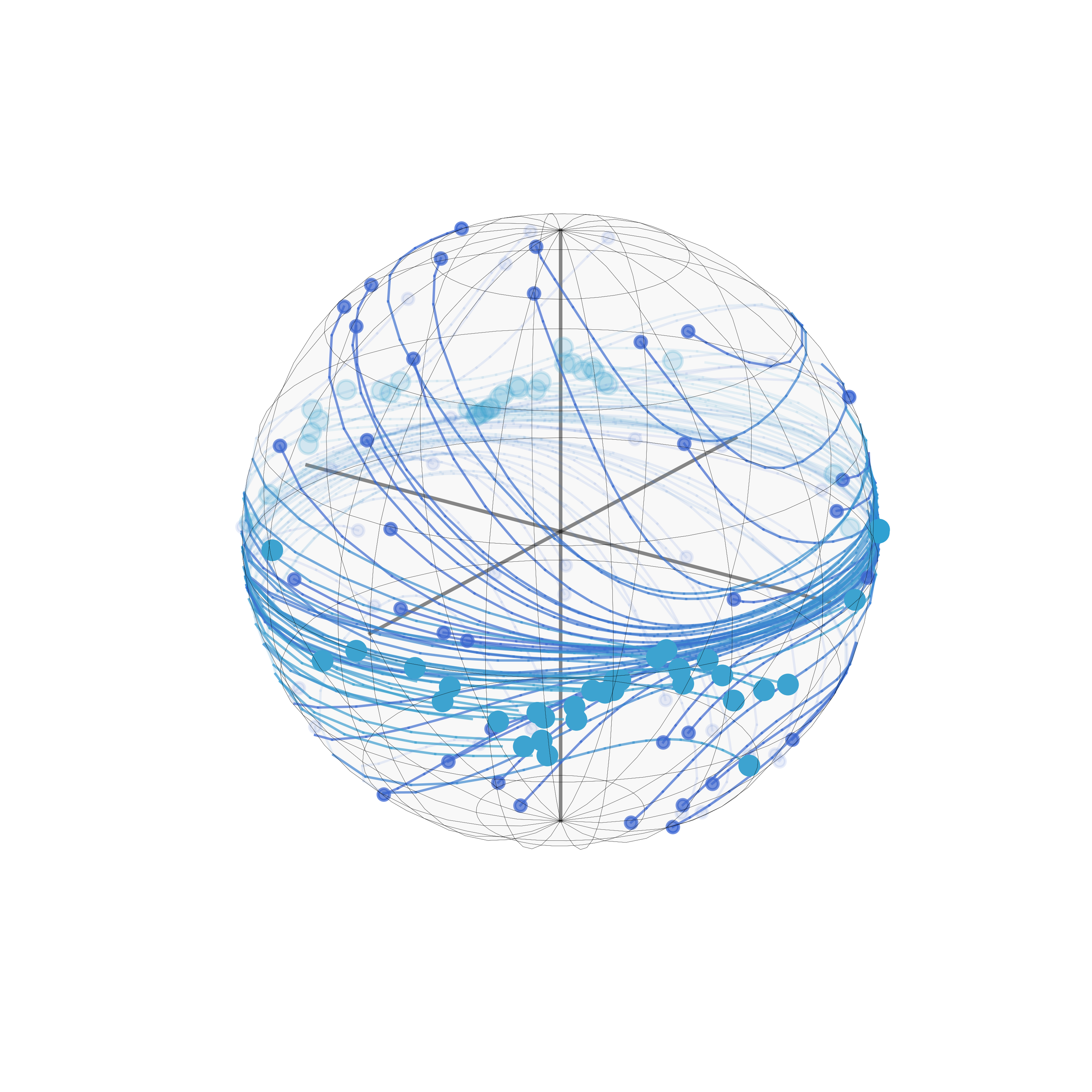}
\caption{$t\in [0,4)$}%
\end{subfigure}\hfill%
\begin{subfigure}{.28\textwidth}%
\includegraphics[width=\linewidth, trim={2cm 2cm 2cm 2cm}, clip]{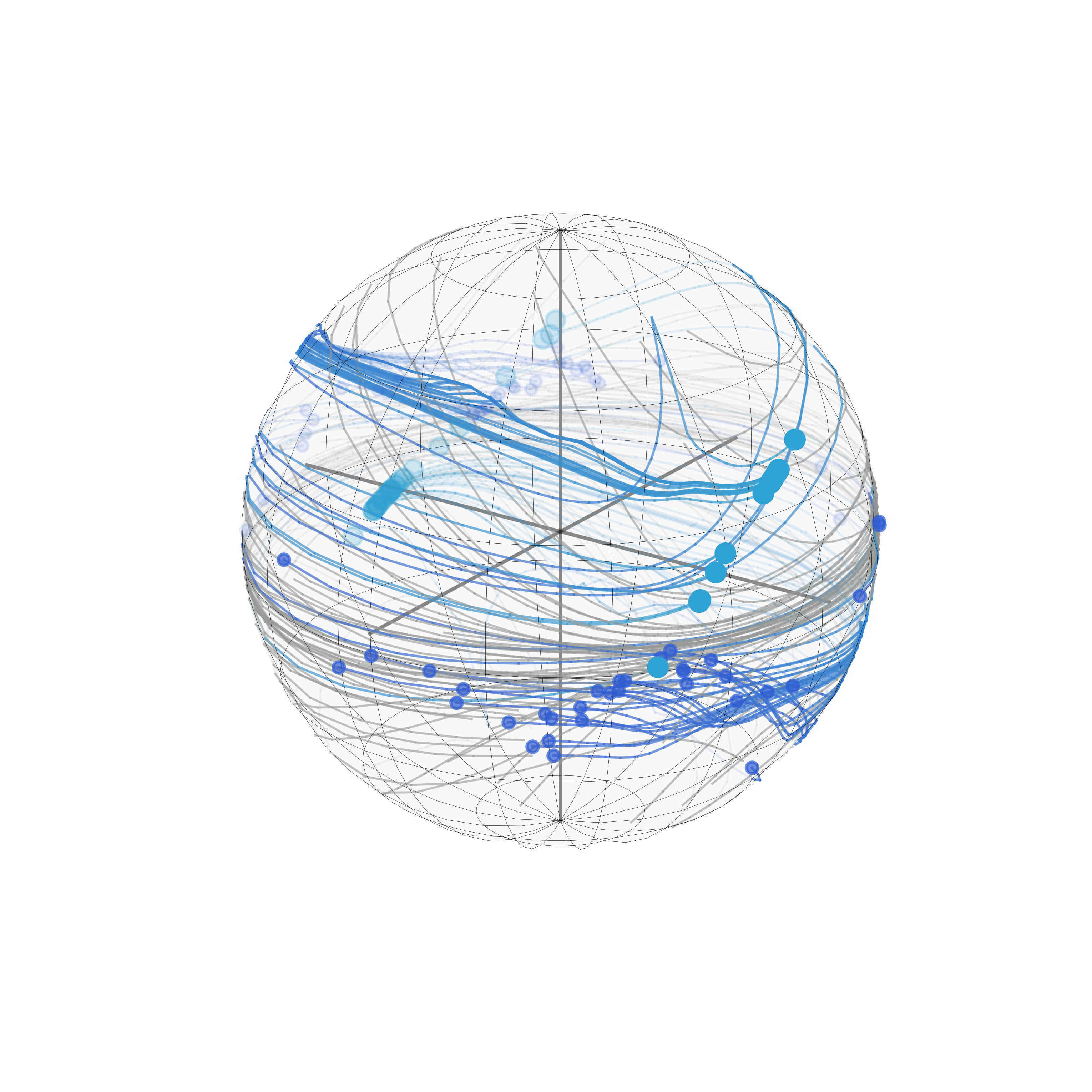}\hfill%
\caption{$t\in [4,9)$}%
\end{subfigure}\hfill%
\begin{subfigure}{.28\textwidth}%
\includegraphics[width=\linewidth, trim={2cm 2cm 2cm 2cm}, clip]{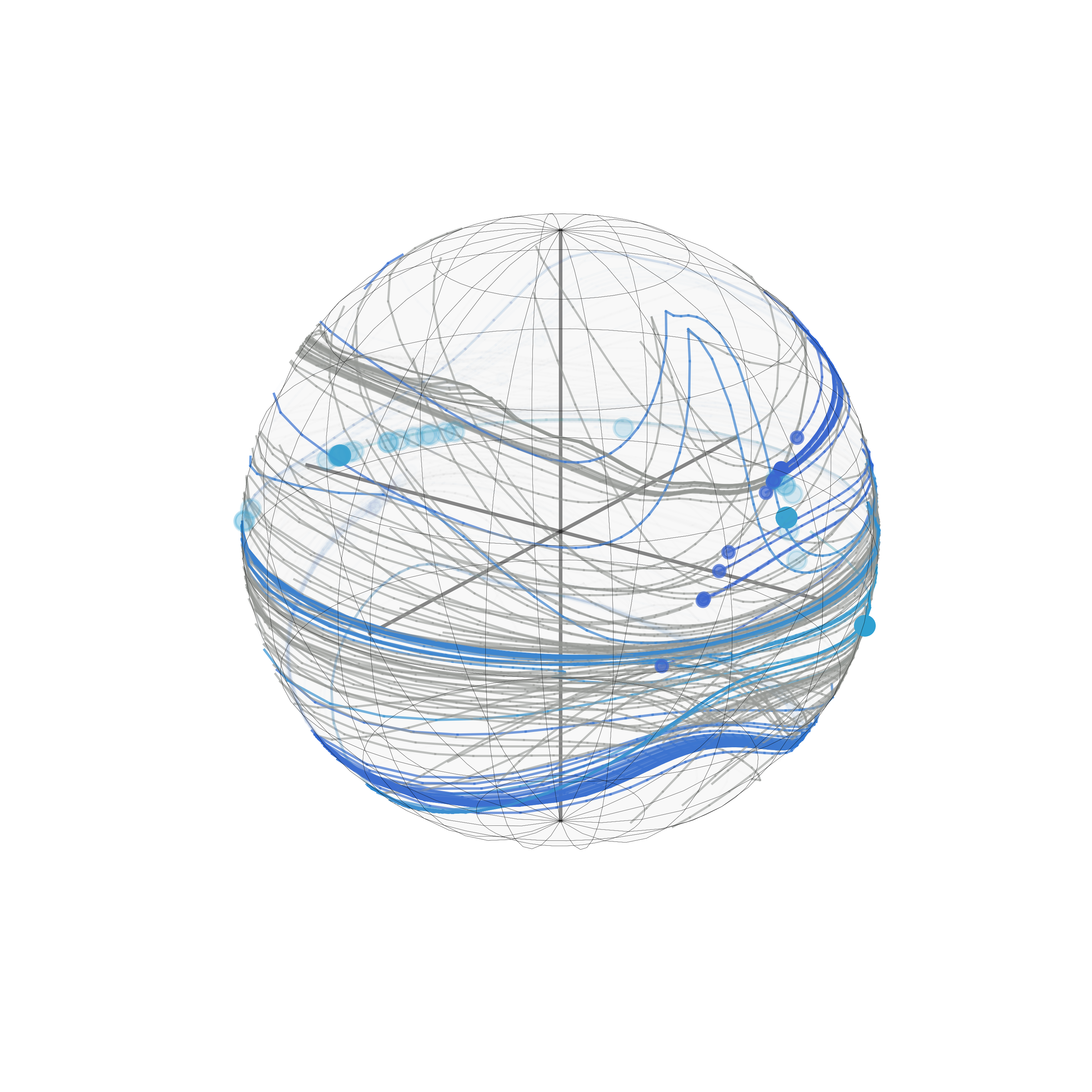}%
\caption{$t\in [9,20)$}%
\end{subfigure}%
\\
\stepcounter{subgroup}\setcounter{subfigure}{0}%
\begin{subfigure}{.28\textwidth}
\includegraphics[width=\linewidth,trim={2cm 2cm 2cm 2cm}, clip]{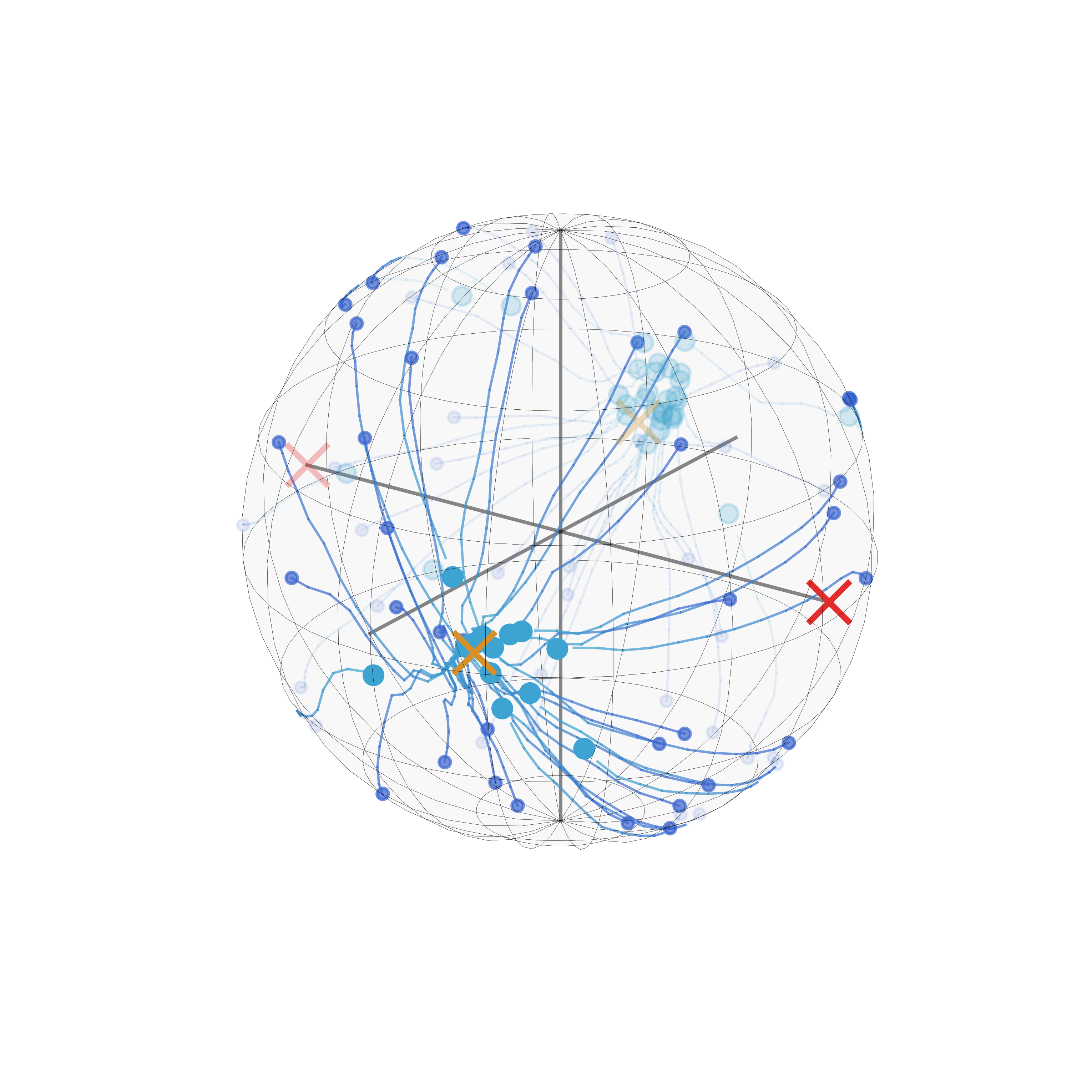}
\caption{$t\in [0,1.5)$}%
\end{subfigure}\hfill%
\begin{subfigure}{.28\textwidth}%
\includegraphics[width=\linewidth, trim={2cm 2cm 2cm 2cm}, clip]{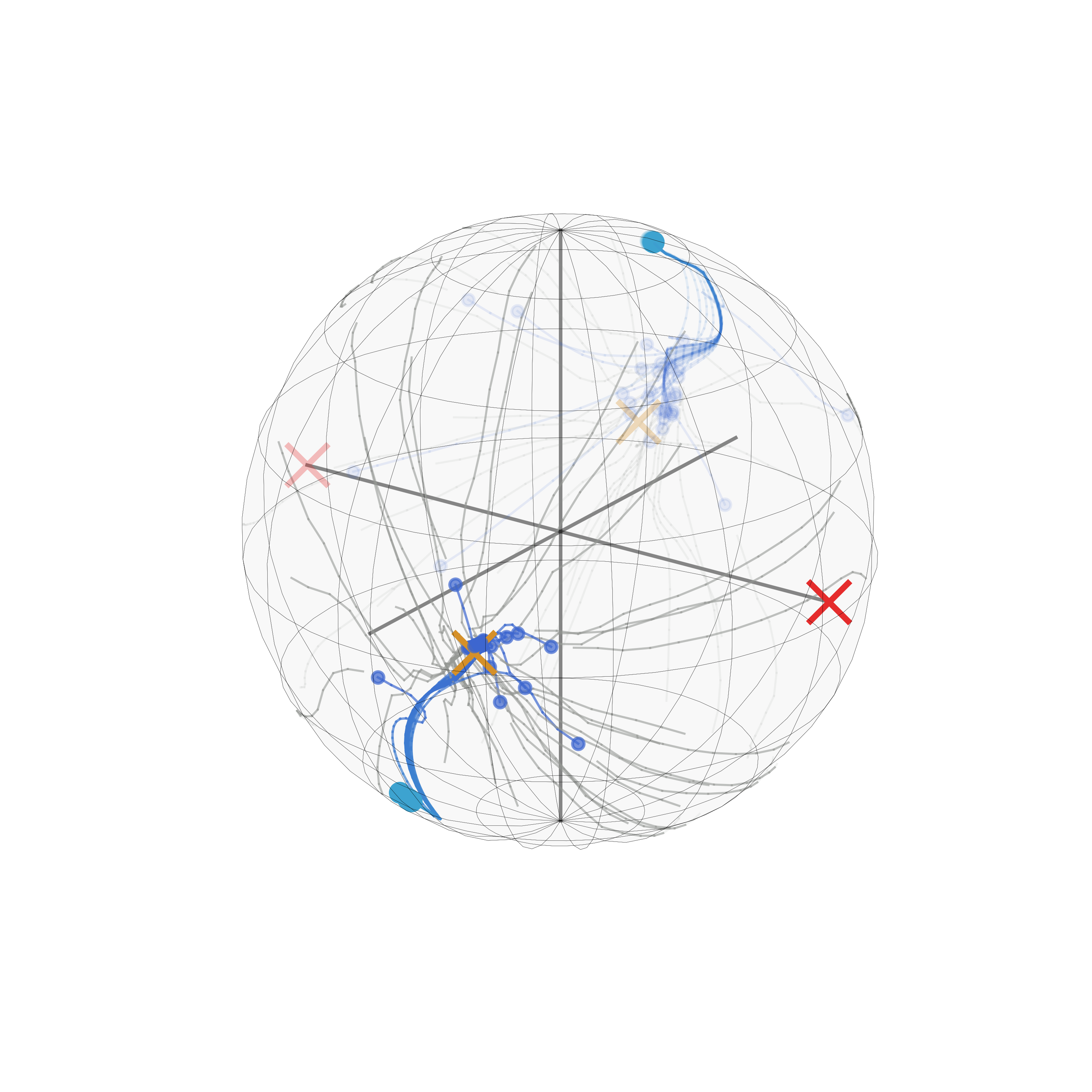}\hfill%
\caption{$t\in [1.5,6)$}%
\end{subfigure}\hfill%
\begin{subfigure}{.28\textwidth}%
\includegraphics[width=\linewidth, trim={2cm 2cm 2cm 2cm}, clip]{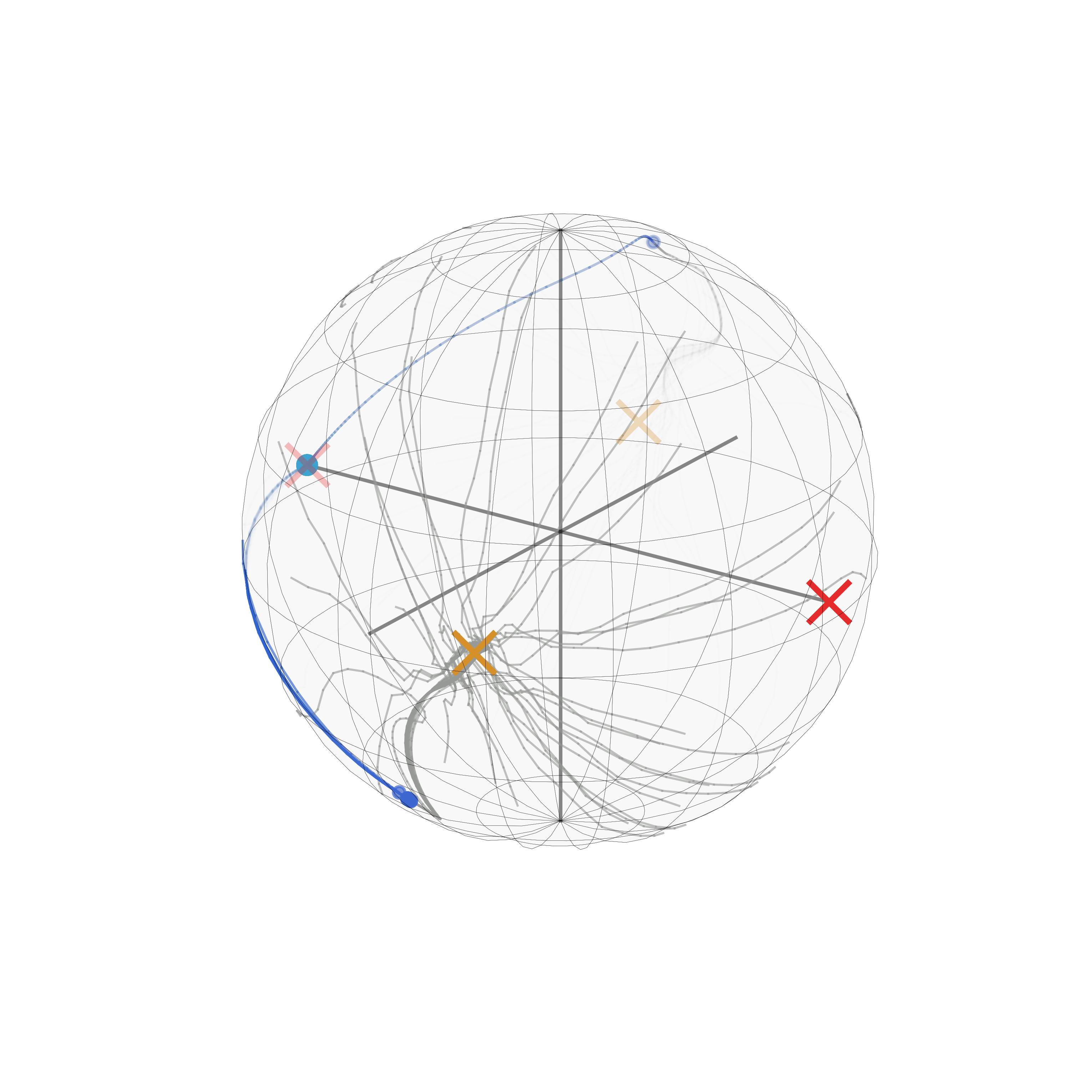}%
\caption{$t\in [6,20)$}%
\end{subfigure}%
\caption{Dynamics of $n=200$ tokens evolving according to \eqref{eq:SAdynamics} with $\beta=100$ for $VB^\top$ not symmetric.
In the top row, corresponding to the choice $V_1$, $B_1$ in \eqref{eq:nonsym}, we observe that trajectories rotate around the sphere and do not display clear convergence or clustering.
In the bottom row, corresponding to $V_2$, $B_2$ in \eqref{eq:nonsym}, we can see exactly the same behavior as in Figure \ref{fig:conj}: initial alignment with $E$, followed by a second phase of alignment with $F$.
}
\label{fig:finiteBeta1}
\end{figure}%
\renewcommand\thesubfigure{\alph{subfigure}}%

\subsection{Gradient flow case}\label{app:gradflow}

In this section, we study the alignment behavior in the gradient flow case, i.e., $B=\pm V$. In particular, as highlighted in \cite{burger2025analysis}, in this case the self-attention dynamics try to either maximize or minimize the energy
$$\mathcal{E}_B(\varrho) = \int_\sphere\int_\sphere e^{\beta\ip{x}{By}}\de\varrho(x)\de\varrho(y).$$
These experiments complement our theory by testing the predicted alignment behavior in the more specific gradient flow setting, where the dynamics admit an additional variational interpretation. In the following, we perform numerical experiments for four cases: maximization and minimization, each with positive and negative definite matrices. Throughout, $v_{\max}(B)$ and $v_{\min}(B)$ denote a unit eigenvector of $B$ associated with its largest and smallest eigenvalues, respectively.

\paragraph{Maximization case.}

For $B=V$, the self-attention dynamics try to maximize the energy $\mathcal{E}_B$ and the optimal stationary state is a single Dirac at $\pm v_{\max}(B)$, i.e., $\rho^{\text{opt}}=\delta_{v_{\max}(B)}$ or $\rho^{\text{opt}}=\delta_{-v_{\max}(B)}$. However, the example in \cite[Figure 2]{burger2025analysis} suggests that, for large $\beta$, the stationary state supported fully on the top eigenspace of $B$ is more likely to occur. This is consistent with Theorem~\ref{thm:main}, since in this case $VB^\top = B^2$. We display this behavior in Figure~\ref{fig:gfmaxspd}, where 
the purple line displays the evolution of ${\mathcal{E}_B(\rho_t^{\beta,n})} / {\mathcal{E}_B(\rho^{\text{opt}})}$, i.e., the energy of the ensemble normalized by the theoretical best energy. For small values of $\beta$, the iteration indeed converges towards an energetically optimal state, which is a single Dirac. However, since the initial configuration is sampled uniformly on the sphere, $\Pi_\sharp\rho_0^{n}$ is likely clustered both at $v_{\max}(B)$ and $-v_{\max}(B)$, which causes the Wasserstein distance in Figure~\ref{fig:gfmaxspd} to be large. For larger values of $\beta$, one instead observes convergence towards energetically suboptimal states, which are closer to $\Pi_\sharp\rho_0^n$.

On the other hand, if $B$ is negative semi-definite, the optimal state is $\rho^{\text{opt}}=\frac{1}{2}(\delta_{v_{\min}(B)} + \delta_{-v_{\min}(B)})$, where in this case $v_{\min}(B) = v_{\max}(B^2)$. In Figure~\ref{fig:gfmaxsnd}, we observe that even for small values of $\beta$, we indeed obtain convergence of $\rho_t^{\beta,n}$ towards $\Pi_\sharp\rho_0^n$. Moreover, we also obtain convergence towards energetically optimal states for large values of $\beta$.

\begin{figure}
\centering
\includegraphics[height=1.5em,trim=0cm 0.2cm 0cm 0.2cm,clip]{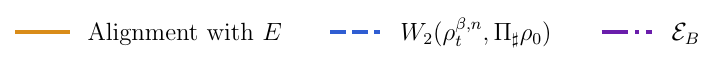}\\
\begin{subfigure}{.33\textwidth}
\includegraphics[width=\linewidth]{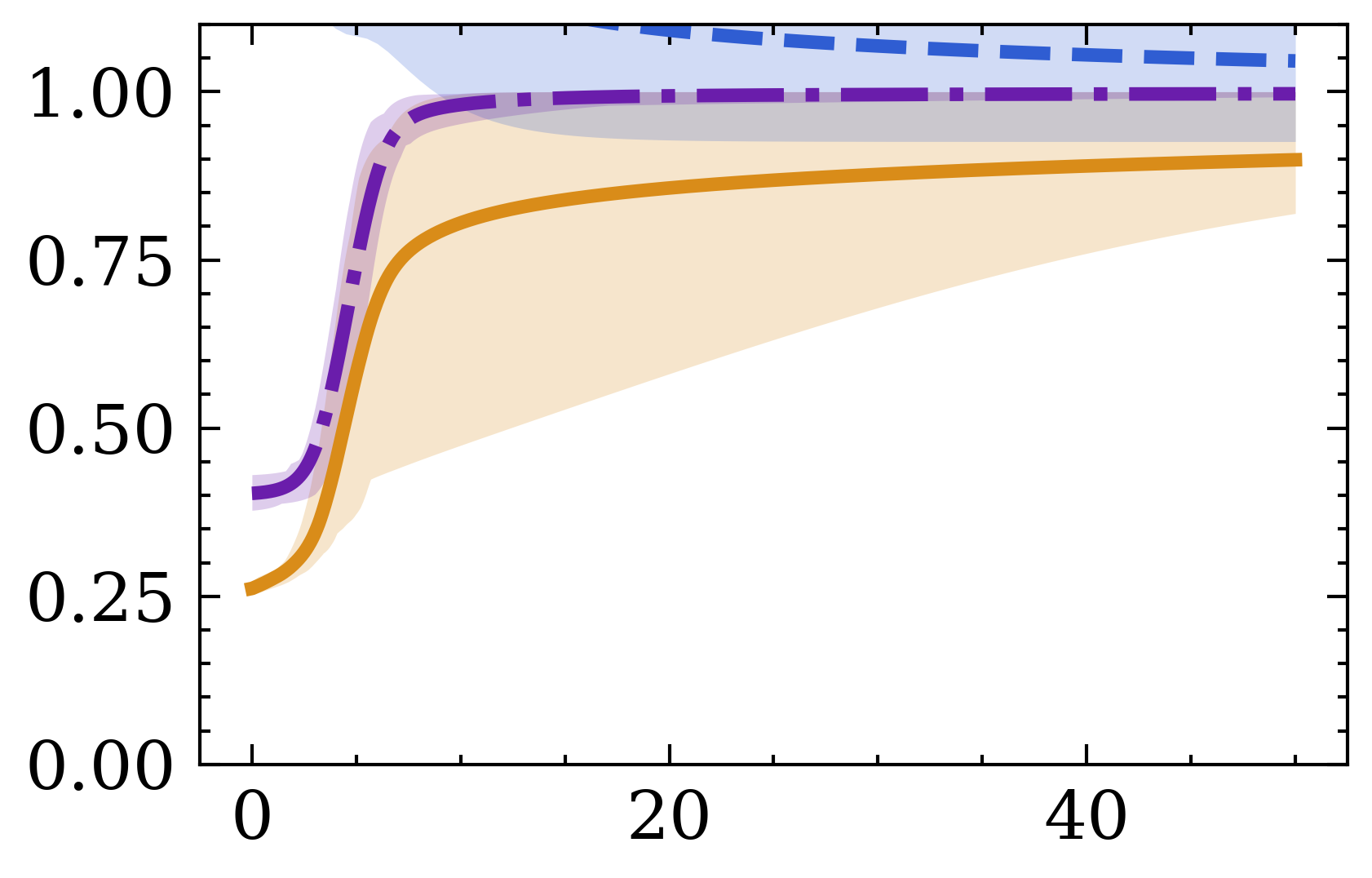}
\caption{$\beta=1$}
\end{subfigure}\hfill%
\begin{subfigure}{.33\textwidth}
\includegraphics[width=\linewidth]{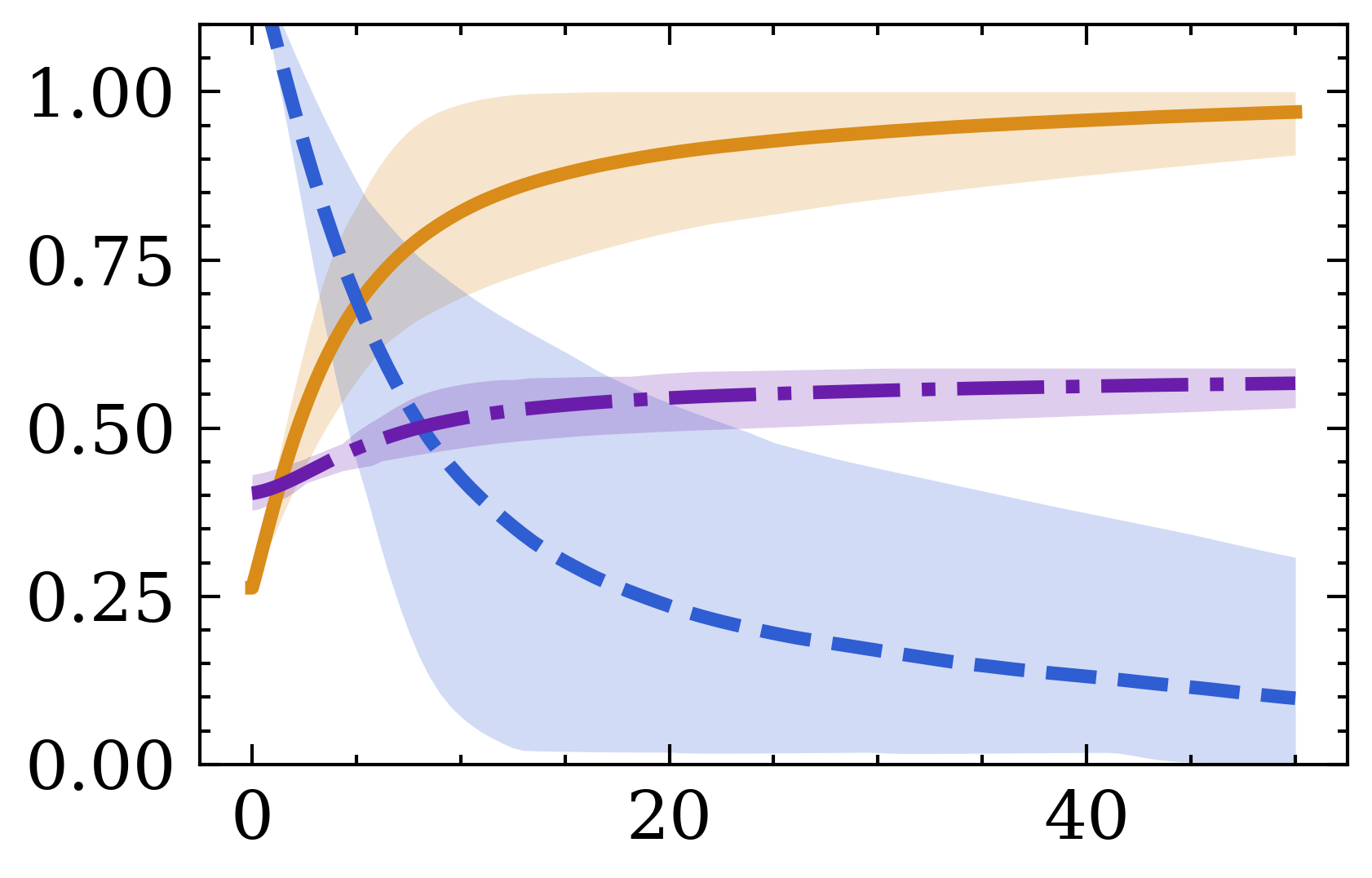}
\caption{$\beta=100$}%
\end{subfigure}\hfill%
\begin{subfigure}{.33\textwidth}
\includegraphics[width=\linewidth]{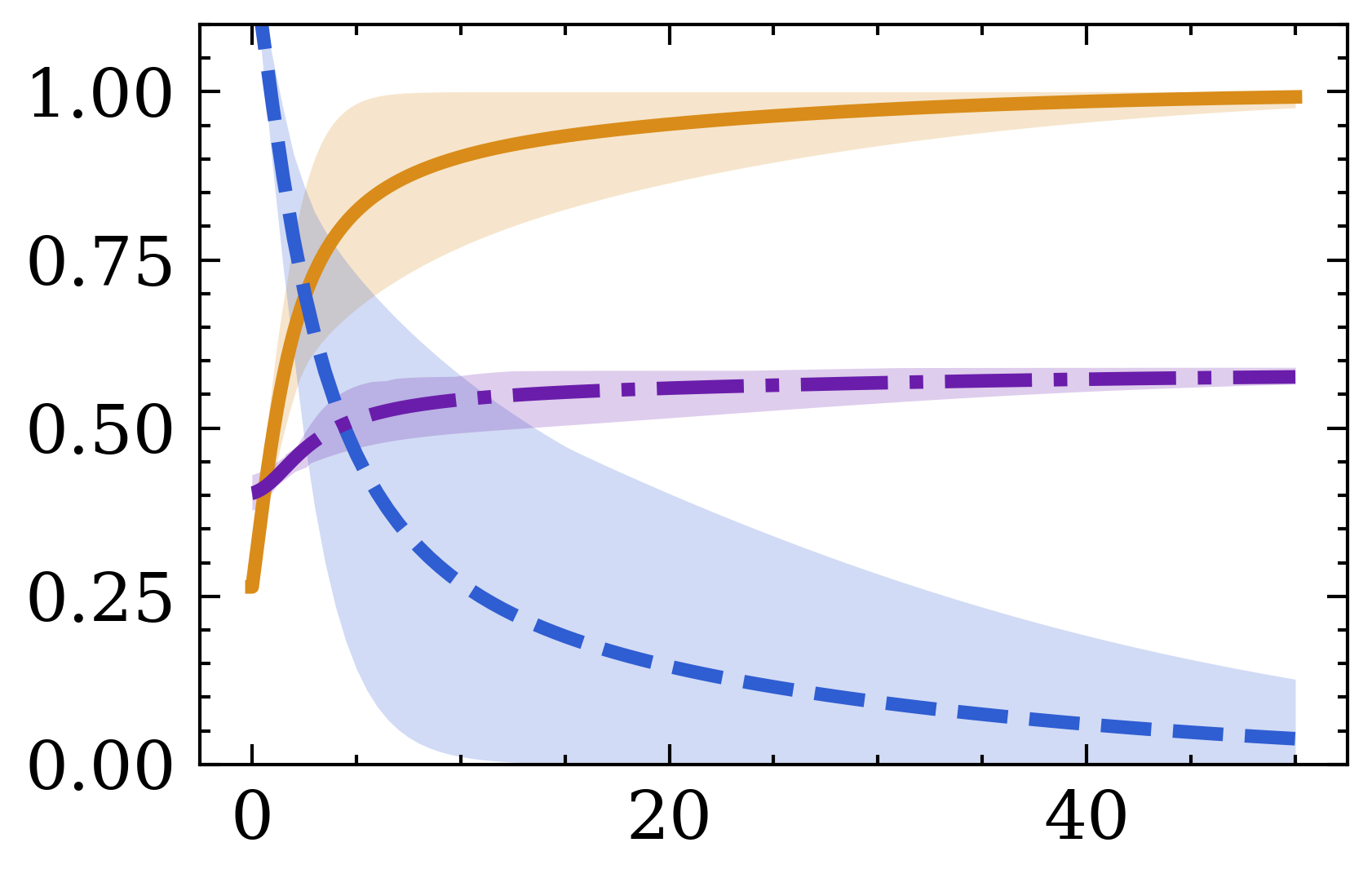}
\caption{$\beta=\infty$}
\end{subfigure}%
\caption{Alignment in the gradient flow maximization case, i.e., $B=V$, with $B$ being a symmetric positive definite matrix. We run the experiment for $10$ different random configurations. We consider $n=100$ particles initialized independently from $\rho_0$, the uniform distribution on $\sphere$.}
\label{fig:gfmaxspd}
\end{figure}

\begin{figure}
\centering
\includegraphics[height=1.5em,trim=0cm 0.2cm 0cm 0.2cm,clip]{figs/max_legend.pdf}\\
\begin{subfigure}{.33\textwidth}
\includegraphics[width=\linewidth]{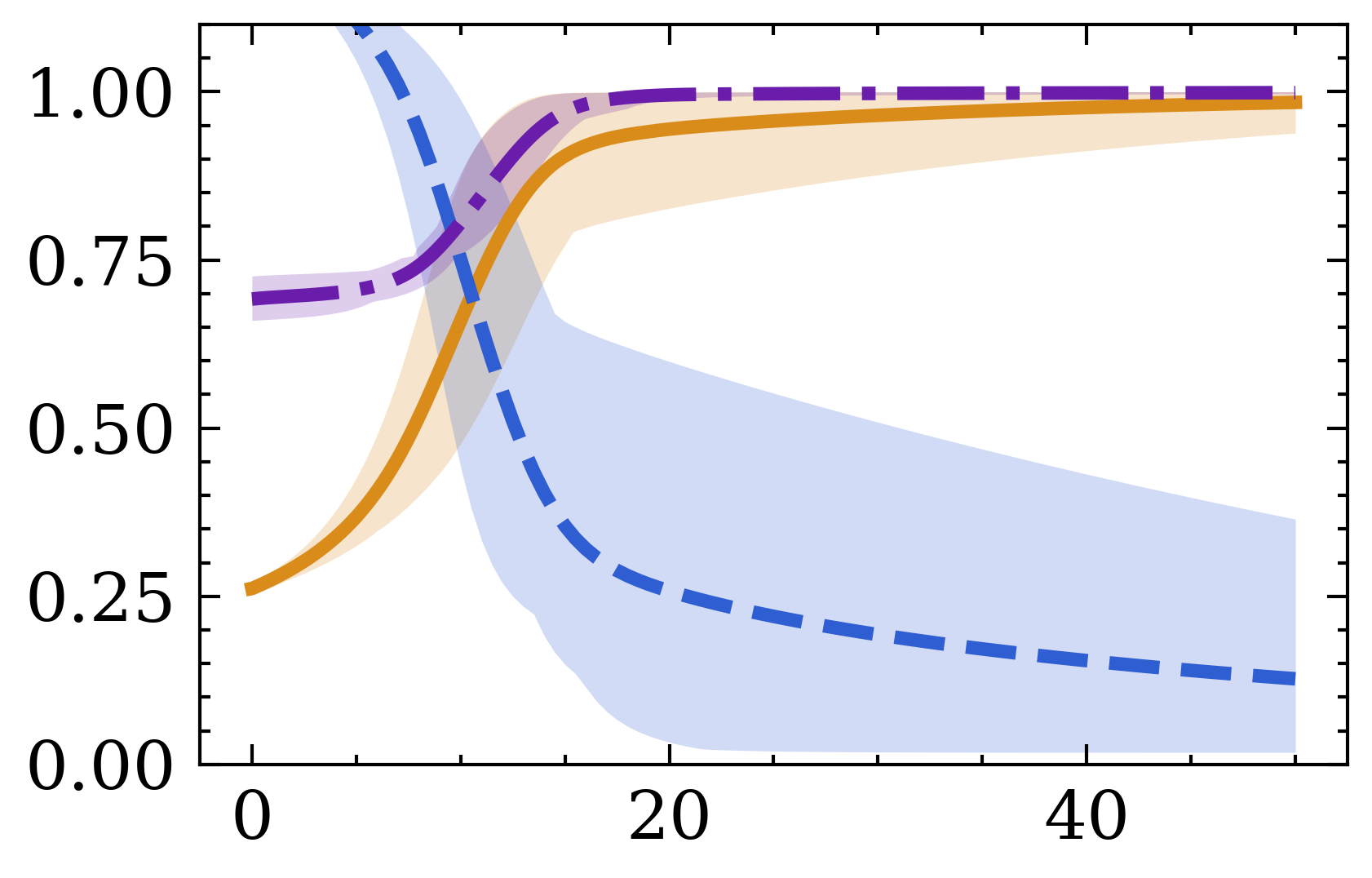}
\caption{$\beta=1$}
\end{subfigure}\hfill%
\begin{subfigure}{.33\textwidth}
\includegraphics[width=\linewidth]{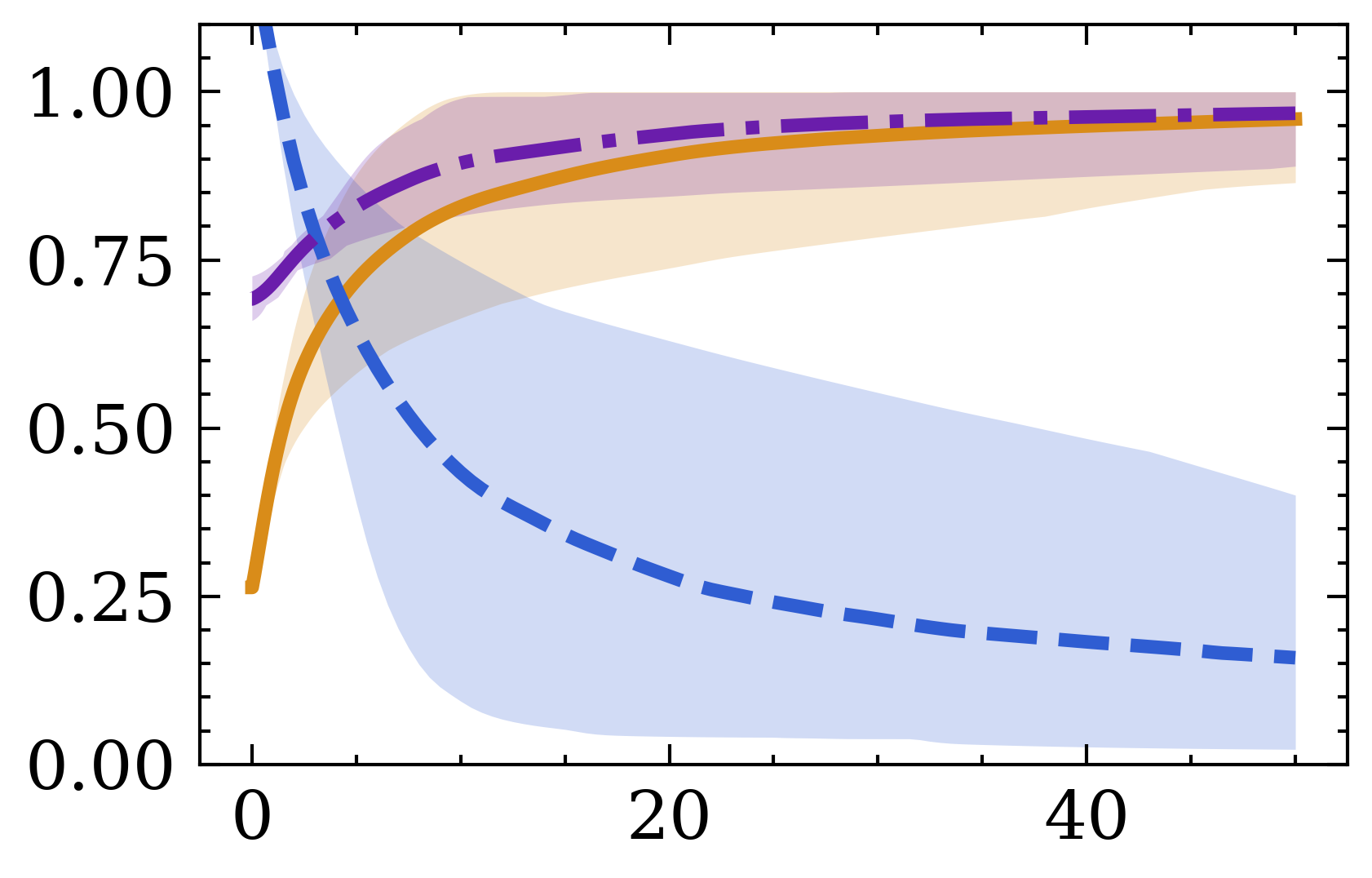}
\caption{$\beta=100$}%
\end{subfigure}\hfill%
\begin{subfigure}{.33\textwidth}
\includegraphics[width=\linewidth]{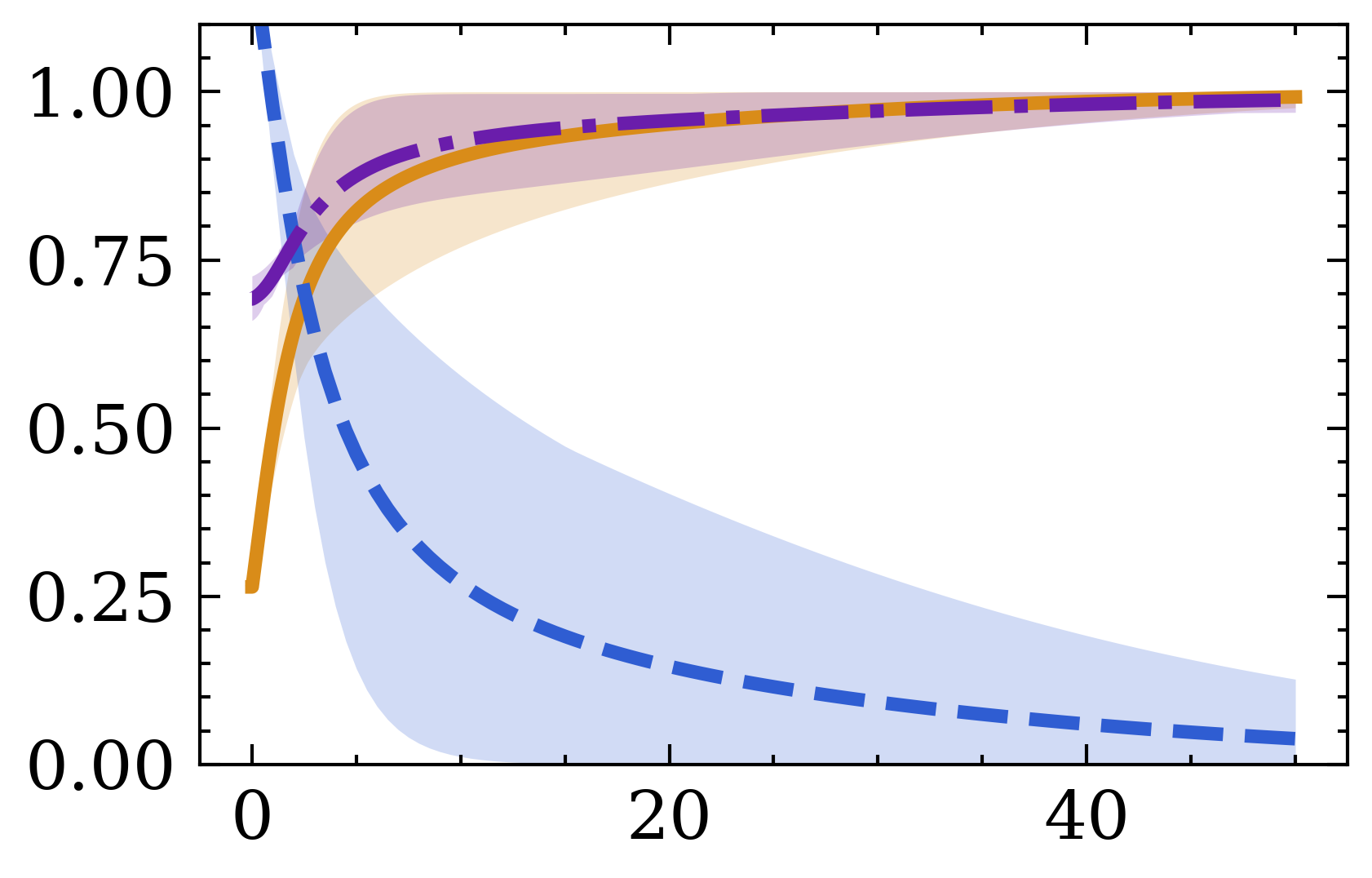}
\caption{$\beta=\infty$}
\end{subfigure}%
\caption{Alignment in the gradient flow maximization case, i.e., $B=V$, with $B$ being a symmetric negative definite matrix. The remaining setup is the same as in Figure~\ref{fig:gfmaxspd}.}
\label{fig:gfmaxsnd}
\end{figure}

\paragraph{Minimization case.} We now consider the case when $B=-V$, in which case the self-attention dynamics try to minimize the energy $\mathcal{E}_B$. Thus, the purple line in the following figures will now display $1 - {\mathcal{E}_B(\rho^{\text{opt}})} / {\mathcal{E}_B(\rho_t^{\beta,n})}$. We first consider the case where $B$ is negative definite. Then, the optimal state is given as $\rho^{\text{opt}}=\delta_{v_{\min}(B)}$ and here $v_{\min}(B)$ does not coincide with $v_{\max}(VB^\top)=v_{\max}(-B^2)$, but instead $v_{\max}(V) = v_{\max}(-B) = v_{\min}(B)$. This can be observed in Figure~\ref{fig:gfminspd}, where for small $\beta$, we obtain convergence towards an energetically optimal state and alignment with $F$. For intermediate $\beta$, we obtain the effect highlighted in Figure~\ref{fig:conj}, where we first obtain a phase of alignment with $E$ and then align to $F$, which here is also a phase where $\mathcal{E}_B$ is minimized. For $\beta=\infty$, as expected, we only obtain alignment towards $E$ and the energy of the ensemble stays constant. Since $B=-V$, this case suggests that the second phase emerges from a disagreement between the top eigenspaces of $B$ and $V$.
\begin{figure}
\centering
\includegraphics[height=1.5em,trim=0cm 0.2cm 0cm 0.2cm,clip]{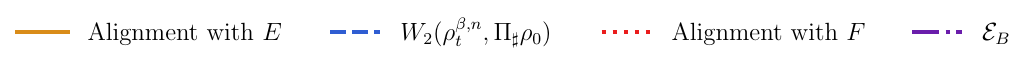}\\
\begin{subfigure}{.33\textwidth}
\includegraphics[width=\linewidth]{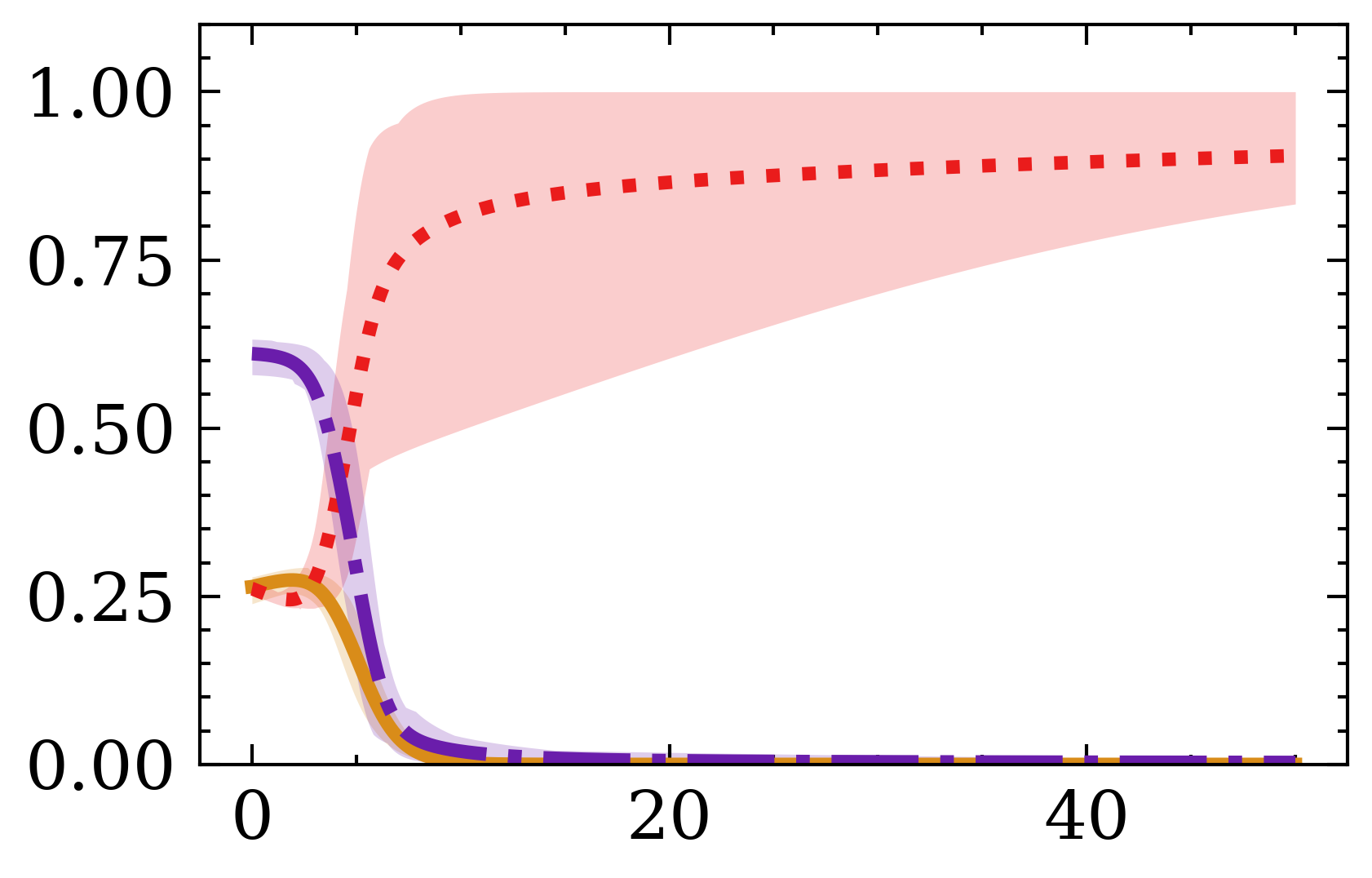}
\caption{$\beta=1$}
\end{subfigure}\hfill%
\begin{subfigure}{.33\textwidth}
\includegraphics[width=\linewidth]{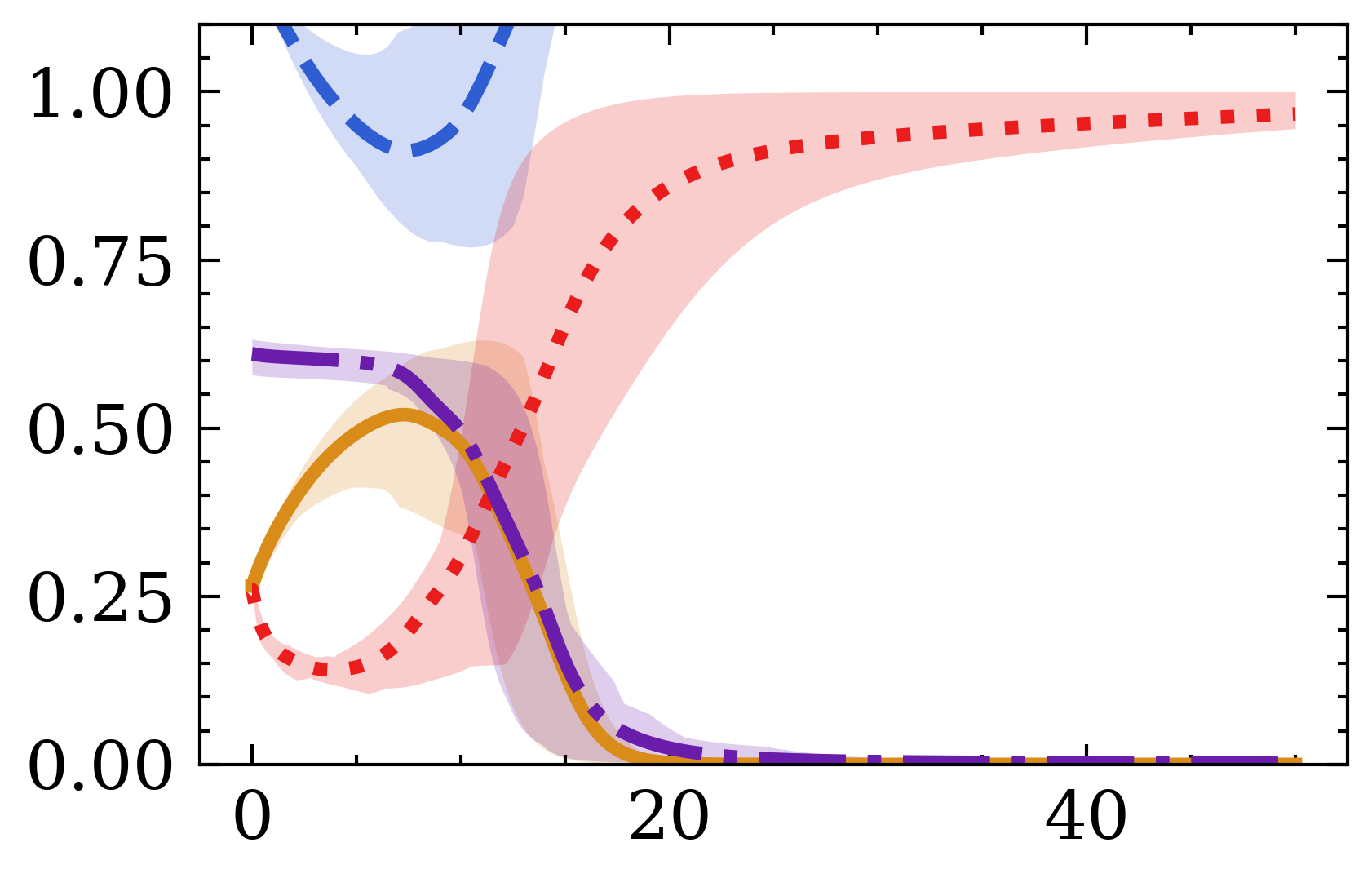}
\caption{$\beta=100$}%
\end{subfigure}\hfill%
\begin{subfigure}{.33\textwidth}
\includegraphics[width=\linewidth]{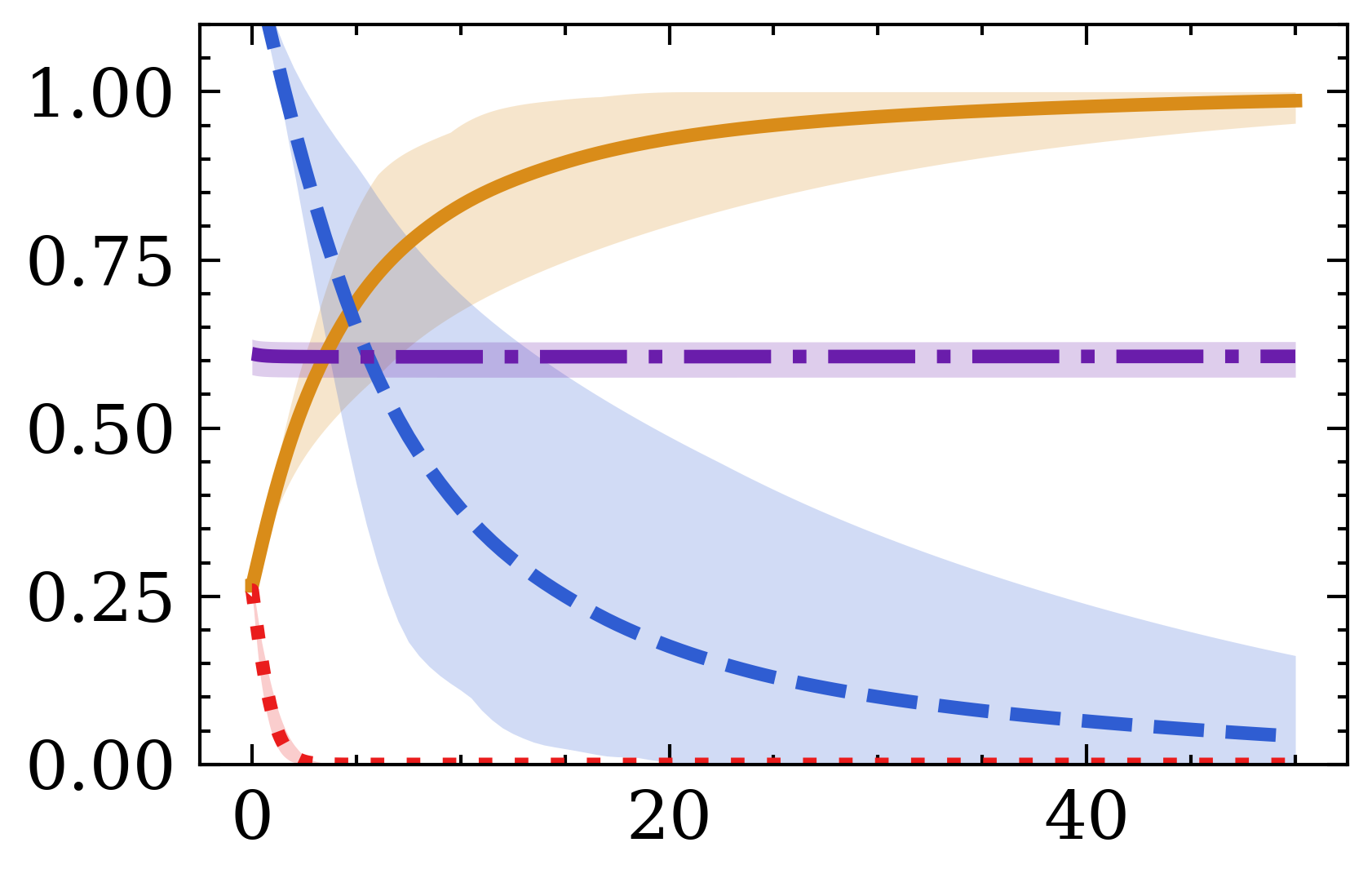}
\caption{$\beta=\infty$}
\end{subfigure}%
\caption{Alignment in the gradient flow minimization case, i.e., $B=-V$, with $B$ being a symmetric negative definite matrix. The remaining setup is the same as in Figure~\ref{fig:gfmaxspd}.}
\label{fig:gfminspd}
\end{figure}
Finally, we consider the minimization case with $B$ being symmetric positive definite. Here, the optimal states are not known analytically and we thus do not display the purple line in Figure~\ref{fig:gfminsnd}. However, the experiments in \cite{burger2025analysis} suggest that the optimal state has two modes at $v_{\min}(B)$ and $-v_{\min}(B)$. These modes are sharper the closer the smallest eigenvalue of $B$ is to zero, and in fact the optimal state is $\frac{1}{2}(\delta_{v_{\min}(B)} + \delta_{-v_{\min}(B)})$ when it is exactly zero. Since in this case $v_{\min}(B) = v_{\max}(VB^\top)$, this observation is in line with Theorem~\ref{thm:main}. While in \cite{burger2025analysis} this effect was observed for $\lambda_{\min}(B)\to 0$, we similarly observe it in Figure~\ref{fig:gfminsnd} for $\beta\to\infty$. We see alignment and Wasserstein convergence only for larger values of $\beta$.
\begin{figure}
\centering
\includegraphics[height=1.5em,trim=0cm 0.2cm 0cm 0.2cm,clip]{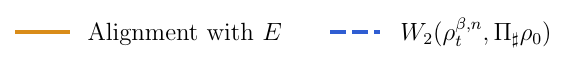}\\
\begin{subfigure}{.33\textwidth}
\includegraphics[width=\linewidth]{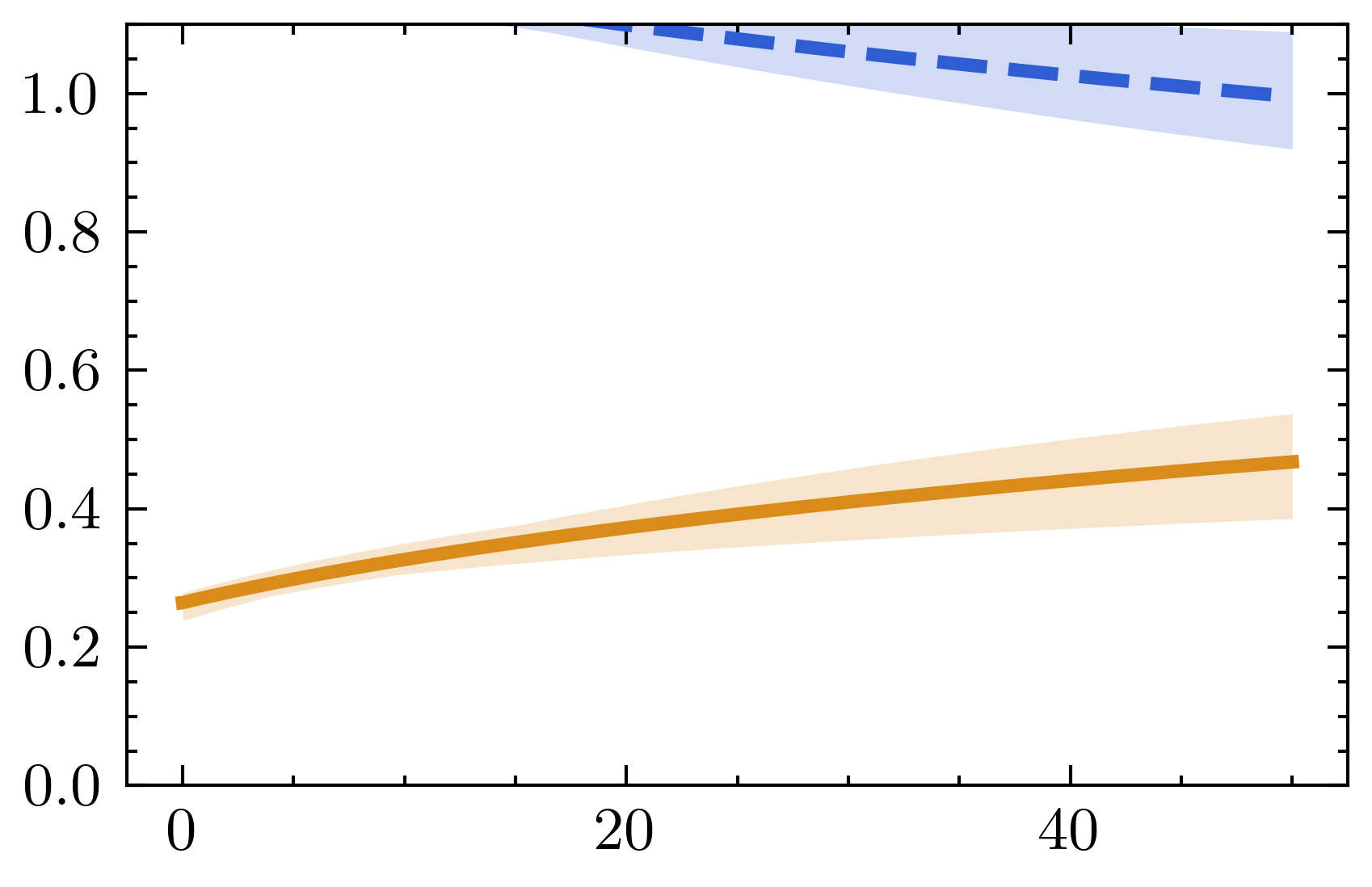}
\caption{$\beta=1$}
\end{subfigure}\hfill%
\begin{subfigure}{.33\textwidth}
\includegraphics[width=\linewidth]{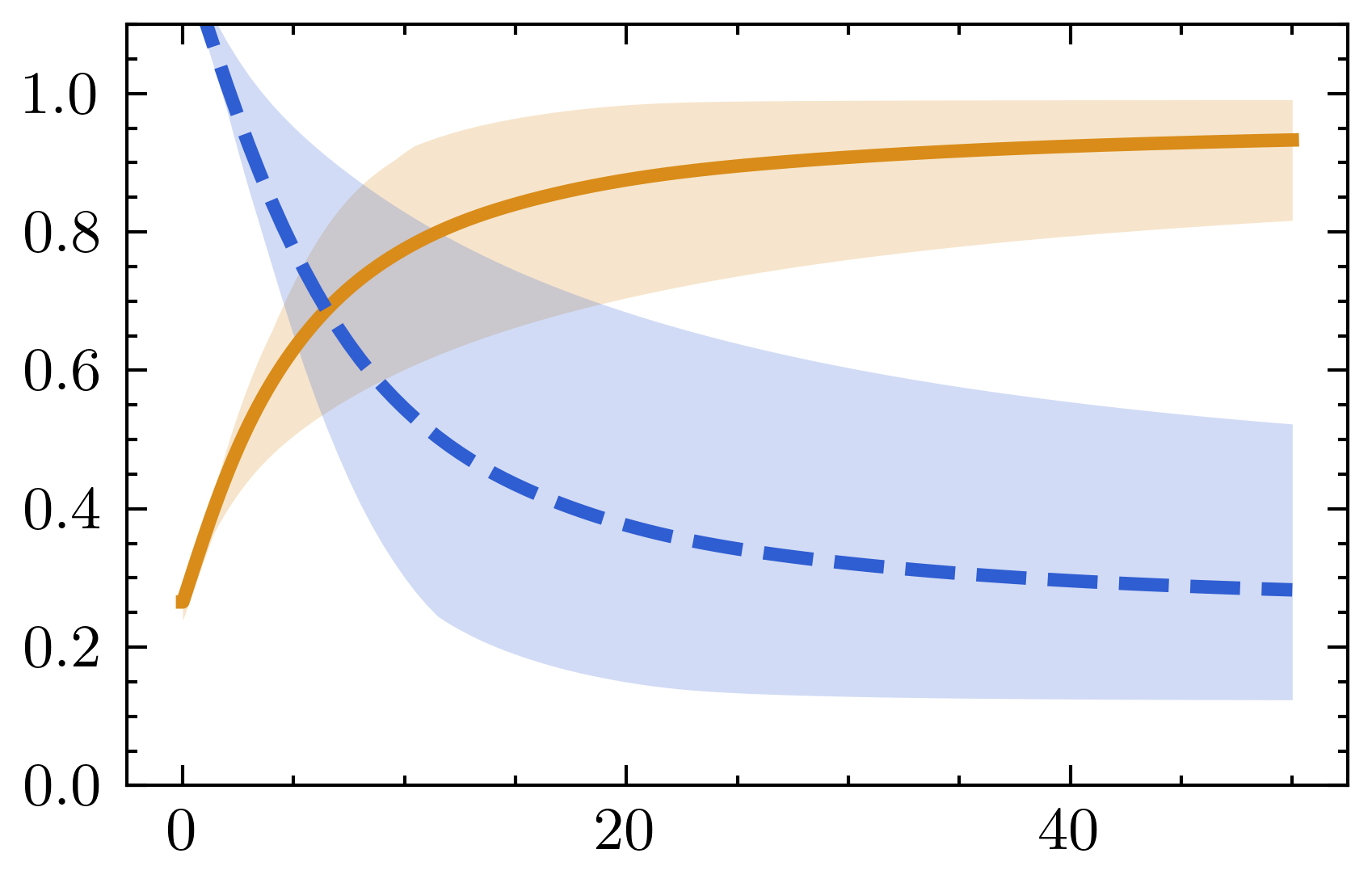}
\caption{$\beta=100$}%
\end{subfigure}\hfill%
\begin{subfigure}{.33\textwidth}
\includegraphics[width=\linewidth]{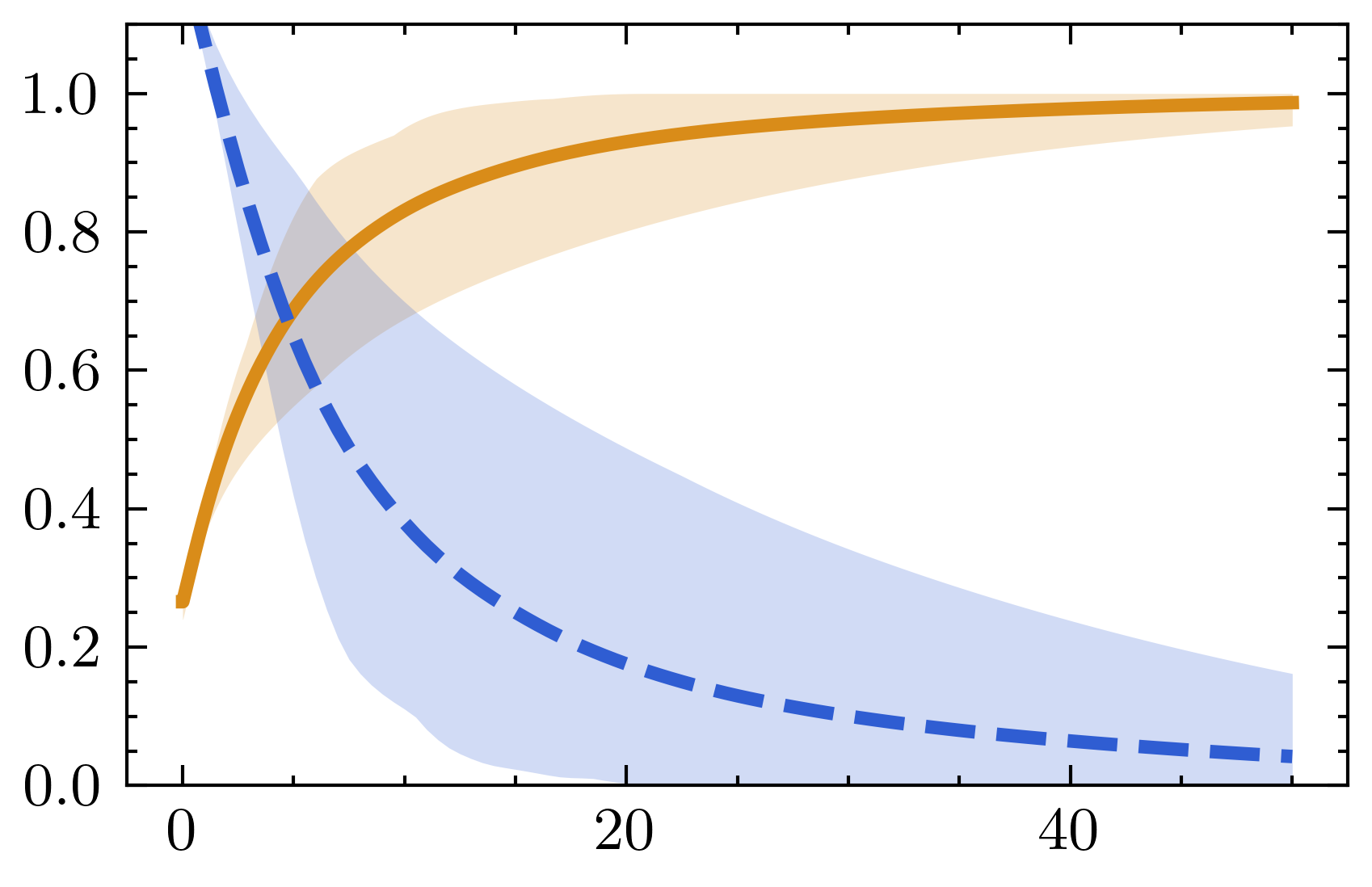}
\caption{$\beta=\infty$}
\end{subfigure}%
\caption{Alignment in the gradient flow minimization case, i.e., $B=-V$, with $B$ being a symmetric positive definite matrix. The remaining setup is the same as in Figure~\ref{fig:gfmaxspd}.}
\label{fig:gfminsnd}
\end{figure}

\subsection{Experiments supporting Conjecture~\ref{conj:v_max(V)}}\label{app:expSupportConj}

We provide additional numerical evidence for Conjecture~\ref{conj:v_max(V)}. For generic token configurations, we conjecture that the long-time behavior of \eqref{eq:SAdynamics} is governed by an eigenspace of $V$: either $F$, associated with the largest eigenvalue of $V$, or $F^{\mathrm{abs}}$, associated with an eigenvalue of largest magnitude. The simulations below illustrate both possibilities.

\begin{figure}[h]%
\centering
\includegraphics[width=0.85\textwidth,trim=0cm 0.2cm 0cm 0.2cm,clip]
{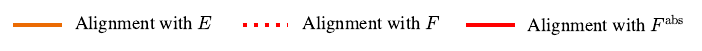}

\begin{subfigure}{.33\textwidth}
\includegraphics[width=\textwidth,trim={0cm .6cm 0cm 0cm},clip]{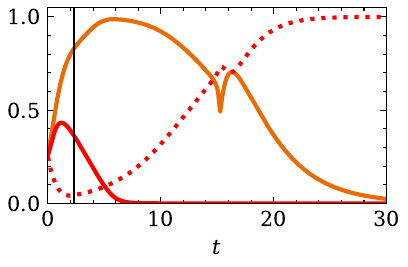}
\caption{$\beta=10$}
\end{subfigure}%
\begin{subfigure}{.33\textwidth}%
\includegraphics[width=\textwidth,trim={0cm .6cm 0cm 0cm},clip]{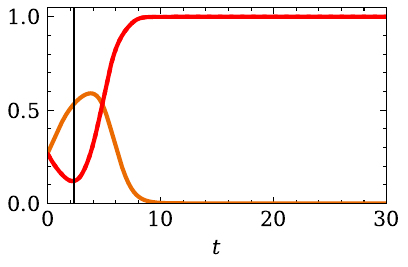}
\caption{$\beta=10$}
\end{subfigure}%
\begin{subfigure}{.33\textwidth}%
\includegraphics[width=\textwidth,trim={0cm .6cm 0cm 0cm},clip]{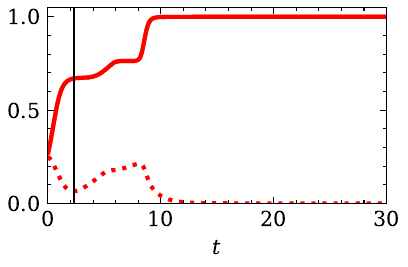}
\caption{$\beta=10$}
\end{subfigure} \\\vspace{1em}
\begin{subfigure}{.33\textwidth}
\includegraphics[width=\textwidth,trim={0cm .6cm 0cm 0cm},clip]{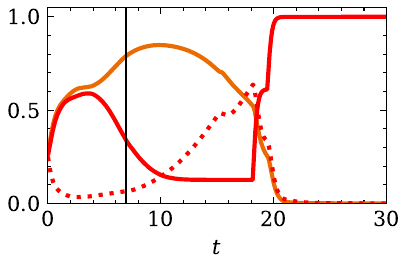}
\caption{$\beta=10^3$}
\end{subfigure}\hfill%
\begin{subfigure}{.33\textwidth}%
\includegraphics[width=\textwidth,trim={0cm .6cm 0cm 0cm},clip]{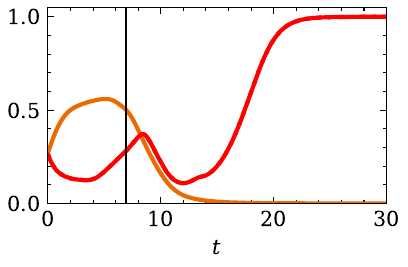}
\caption{$\beta=10^3$}
\end{subfigure}\hfill%
\begin{subfigure}{.33\textwidth}%
\includegraphics[width=\textwidth,trim={0cm .6cm 0cm 0cm},clip]{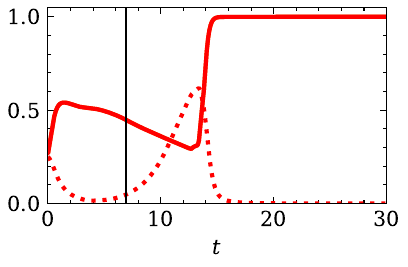}
\caption{$\beta=10^3$}
\end{subfigure}%
\caption{%
\label{fig:alignment_VBcases}%
Alignment of $\rho_t^{\beta,n}$ with the top eigenspace $E$ of $VB^\top$ and with the eigenspaces $F$ and $F^{\mathrm{abs}}$ of $V$, for $d=10$ (maximum alignment equals $1$). The vertical line marks $t=\log\beta$. Each column uses a different pair of random diagonal matrices $V$ and $B$, with independent normally distributed diagonal entries. The initial condition consists of $n=500$ tokens sampled from the uniform distribution $\rho_0$ on $\sphere$.}
\end{figure}
Figure~\ref{fig:alignment_VBcases} shows three representative random diagonal instances of $V$ and $B$. For $\beta=10$, panel~(a) exhibits asymptotic alignment with $F$. In panel~(b), $F=F^{\mathrm{abs}}$, and the dynamics align with this common eigenspace. In panel~(c), $E=F^{\mathrm{abs}}$, and the limiting alignment is with this eigenspace.

For $\beta=10^3$, panels~(e) and~(f) show the same qualitative behavior as panels~(b) and~(c), respectively, up to the expected delay on the time scale $t\approx \log\beta$. In panel~(d), however, the limiting alignment changes from $F$ to $F^{\mathrm{abs}}$. These simulations therefore support the prediction that the dynamics align with an eigenspace of $V$, while indicating that the selected eigenspace can be either $F$ or $F^{\mathrm{abs}}$, depending on the matrix realization, the initial configuration, and the temperature regime.

\subsection{Empirical validation of the \texorpdfstring{$\mathcal{O}(\log\beta)$}{O(log β)} time scale}\label{sec:empirical_O_beta}

In this section, we confirm experimentally the interpretation of Theorem~\ref{thm:main}, see the discussion ahead of and after Corollary~\ref{cor:time_scales},
which states that the relevant time scale for the exponential convergence to the top eigenspace $E$ of $VB^\top$ is of order $\mathcal{O}(\log\beta)$,
before the dynamics may enter a terminal phase of convergence to an eigenspace of $V$.

Since previous examples show that certain configurations exhibit alignment towards the top eigenspace $E$ as their long-term behavior (e.g., when the top eigenspace $F$ of $V$ coincides with $E$), we specifically select a scenario, where the alignment breaks after a certain time frame and then record
\begin{align}\label{eq:tminbet}
t^\beta \in \argmin_{t\in [0,T]} W_2(\rho_t^{\beta,n}, \Pi_{\sharp}\rho_0).
\end{align}
Specifically, we consider the gradient flow case with $V=-B$ and $B$ being an indefinite matrix in dimension $d=10$. In Figure~\ref{fig:tminbet},
we observe that the empirical times $t^\beta $ recorded in this setup indeed scale like a logarithmic curve.

\begin{figure}[h]%
\centering
\includegraphics[width=0.5\linewidth]{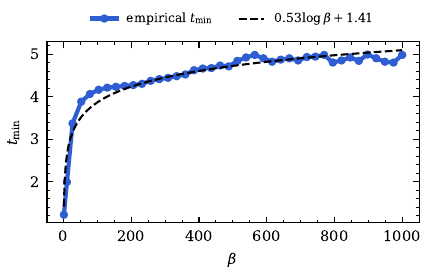}
\caption{We plot $t^\beta$ as defined in \eqref{eq:tminbet}, i.e., the time at which the distance between $\rho_t^{\beta,n}$ and the projection of the initial condition $\Pi_{\sharp}\rho_0$ is minimal. For illustrative purposes, we fit a logarithmic curve to support the claimed $\mathcal{O}(\log\beta)$ scaling.}
\label{fig:tminbet}
\end{figure}

\end{document}